\documentclass{article}
\usepackage[utf8]{inputenc}

\usepackage[T1]{fontenc}

\usepackage{amsmath}
\usepackage{amssymb}
\usepackage{amsfonts}
\usepackage{amsthm}
\usepackage{graphicx}
\usepackage{epstopdf}
\usepackage{enumerate}
\usepackage[ruled, linesnumbered]{algorithm2e}
\usepackage[colorinlistoftodos]{todonotes}
\usepackage{soul}
\usepackage{bm}
\usepackage{hyperref}
\usepackage{multirow}
\usepackage{tikz}
\usepackage{subcaption}

\usetikzlibrary{cd, positioning, shapes.geometric, calc}

\usepackage[a4paper,top=3cm,bottom=3cm,left=3cm,right=3cm,marginparwidth=4cm]{geometry}

\usepackage{natbib}

\definecolor{lightgreen}{RGB}{180,255,150}

\definecolor{lightblue}{RGB}{180,180,255}

\newcommand{\remie}[1]{#1}
\newcommand{\remiee}[1]{#1}

\newcommand{\yuki}[1]{#1}

\newtheorem{theorem}{Theorem}
\newtheorem{lemma}{Lemma}
\newtheorem{corollary}{Corollary}
\newtheorem{proposition}{Proposition}
\newtheorem{observation}{Observation}
\newtheorem{remark}{Remark}
\newtheorem{conjecture}{Conjecture}

\theoremstyle{definition}
\newtheorem{definition}{Definition}

\title{On Cherry-picking and Network Containment}
\author{Remie Janssen \footnote{Delft Institute of Applied Mathematics, Delft University of Technology, Van Mourik Broekmanweg 6,
2628 XE, Delft, The Netherlands, \{R.Janssen-2, Y.Murakami\}@tudelft.nl. Research funded by the Netherlands Organization for Scientific Research (NWO), with the Vidi grant 639.072.602.} \and Yukihiro Murakami \footnotemark[1]}
\date{\today}

\begin{document}

\maketitle

\begin{abstract}
Phylogenetic networks are used to represent evolutionary scenarios in biology and linguistics. To find the most probable scenario, it may be necessary to compare candidate networks, to distinguish different networks, and to see when one network is contained in another. 
%We show that the tree-child sequences introduced by Linz and Semple characterize when a tree is embedded in a tree-child network. 
%We take this one step further and show that the sequences can be used to characterize when a tree-child network is embedded in another tree-child network ({\sc Network Containment}), and that this can be decided in linear time.
\yuki{In this paper, we introduce cherry-picking networks, a class of networks that can be reduced by a sequence of two graph operations.
We show that some networks are uniquely determined by the sequences that reduce them---we call these the reconstructible cherry-picking networks, and further show that given two cherry-picking networks within the same reconstructible class, one is contained in the other if a sequence for the latter network reduces the former network.
%Following this, we show that it is possible to decide in linear time if two cherry-picking networks of the same class are isomorphic.
By restricting our scope to tree-child networks, we show that the converse of the above statement holds, thereby showing that {\sc Network Containment}, the problem of checking whether a network is contained in another, can be solved in linear time for tree-child networks.}
% We show that the tree-child sequences introduced by Linz and Semple characterize when a tree-child network is contained in another tree-child network ({\sc Network Containment}), and that this can be decided in linear time.
We implement this algorithm in Python and show that the linear-time theoretical bound on the input size is achievable in practice.
Lastly, we provide a linear time algorithm for deciding whether two tree-child networks are isomorphic.
% We also generalize tree-child sequences to cherry-picking sequences, and consequently define the class of cherry-picking networks---the networks that can be reduced by cherry-picking sequences.
%  and following this result, we show that it is possible to decide in linear time if two cherry-picking networks are isomorphic.
\end{abstract}

\section{Introduction}\label{sec:Introduction}
Phylogenetic networks are gaining popularity in the study of the evolutionary history of species \citep{morrison2005networks, bapteste2013networks}. However, small stretches of DNA (e.g., pieces of DNA coding for protein domains) mostly evolve tree-like. Therefore, the network representing the species' evolution must contain the trees for such pieces of DNA. This leads to the following mathematical problem. For a given network $N$ and a tree $T$ on the same set of species, decide whether $N$ contains~$T$. 
%For a given set of trees, decide whether there exists a ``simple'' network that contains all the trees.

% These two problems clearly go hand in hand, as it is imperative to understand when a tree is contained in a network, to find a network which contains a set of trees.
% In this paper we solve a generalization of the first problem, with tools that were developed for tackling the second.

% \todoRemie{Shorten the two paragraphs about tree-containment a bit? YM:Done}
This problem, called {\sc Tree Containment}, is NP-complete for general rooted phylogenetic networks \citep{kanj2008seeing}. 
%The problem remains NP-complete for certain network classes (networks with particular topological restrictions), such as tree-sibling, time-consistent, and regular networks~\citep{van2010locating}. However, for other network classes, the problem becomes easier. 
Many have overcome this computational challenge by considering inputs of topologically restricted networks.
It was shown initially that {\sc Tree Containment} can be solved in polynomial time for non-binary normal networks, binary tree-child networks, and binary level-$k$ networks \citep{van2010locating}.
Stronger results have been proven for genetically stable networks (quadratic time) \citep{gambette2015solving}, binary nearly-stable networks (linear time) \citep{gambette2018solving}, and for binary tree-child networks \citep{gambette2018solving,gunawan2017solving}.

% There are even stronger results for some network classes: deciding whether a tree is contained in a genetically stable network can be done in quadratic time \citep{gambette2015solving}, and making this decision for a binary nearly-stable network takes linear time \citep{gambette2018solving}. %also builds on \citep{gambette2015locating, fakcharoenphol2015faster}
% % Tree-child networks are particularly interesting as they have been attracting growing attention. 
% % \citep{Vincent_said_so}. 
% For the class of tree-child networks, {\sc Tree Containment} is known to be linear time solvable \citep{gambette2018solving,gunawan2017solving}. %as each tree-child network is reticulation visible and the problem is linear time solvable for reticulation visible networks

From a biological and a computational perspective, there is no reason why we should restrict ourselves to inputs of a tree and a network.
Indeed, while small stretches of DNA evolve tree-like, it is possible for larger parts of the genome to evolve as a network.
In such instances, it is of great interest to consider a more general version of {\sc Tree Containment}, which we call {\sc Network Containment}: For given networks~$N$ and~$N'$ on the same set of species, decide whether~$N$ contains~$N'$.
%Computationally, this problem---which we appropriately name {\sc Network Containment}---provides a framework for determining whether network classes that can be solved fast for {\sc Tree Containment} can still be solved fast for the more general case.
Computationally, it is natural to wonder whether we can also {\sc Network Containment} efficiently in all network classes where {\sc Tree Containment} can be solved efficiently as well.
To date, no study has ever considered this problem, and we take the first steps in this endeavour.
Furthermore, previous studies have focused primarily on binary networks.
In this paper, we keep our results as general as possible, mentioning in statements of lemmas whether the networks have node degree restrictions on them.

We solve the {\sc Network Containment} problem for tree-child networks by considering \emph{cherry-picking sequences}.
These sequences were developed to tackle the problem of finding a ``simple'' network that displays a given set of trees~\citep{docker2019deciding, linz2019attaching}.
Two leaves of a tree form a \emph{cherry} if they share a common parent---by successively `picking' cherries (removing one of the leaves in a cherry) from the set of input trees, we obtain a sequence of cherries that ultimately reduce each input tree to a tree on a single leaf.
This sequence of cherries then corresponds to some network that contains the set of all input trees.

Although it was mentioned how one can construct a network from a cherry-picking sequence, no further characterizations for the type of networks that can be obtained from such sequences have been investigated~\citep{docker2019deciding, linz2019attaching}.
Furthermore, actions of cherry-picking sequences were defined only on trees, and not on networks.
In this paper, we fill this gap by first defining the action of a cherry-picking sequence on a network, and introduce the class of \emph{cherry-picking networks}: networks that can be reduced to a single leaf by a cherry-picking sequence.
We investigate the correspondence between cherry-picking networks and the sequences that reduce them, and ultimately show that within a particular \emph{reconstructible} class, cherry-picking networks are characterized uniquely by their \emph{smallest} cherry-picking sequences (Theorem~\ref{thm:UniqueCPNConstruction}).

By carefully defining what it means for a network to be a subnetwork of and to be contained in another network, we derive results that connect these notions to cherry-picking reductions.
Roughly speaking, a network~$N'$ is a \emph{subnetwork} of another network~$N$ if~$N'$ can be obtained from~$N$ by deleting reticulation edges and suppressing degree-$2$ vertices.
A network~$N'$ is \emph{contained} in another network~$N$ if~$N'$ can be obtained from~$N$ by deleting reticulation edges, suppressing degree-$2$ vertices, and by contracting edges.
We show that within particular classes of cherry-picking networks, if a sequence for a network~$N$ reduces another network~$N'$, then~$N'$ is contained in~$N$ (Lemma~\ref{lem:reductionImpliesContainmentSF}).
Unfortunately the converse does not hold (Theorem~\ref{thm:CPNTCFails}), unless the two input networks are tree-child.

It turns out that the class of tree-child networks is contained in the class of cherry-picking networks, as each tree-child network has a special type of cherry-picking sequence---a \emph{tree-child sequence}---that reduces it.
We examine how these sequences can be used to solve {\sc Network Containment} for tree-child networks.
In particular, within some CPN classes, we show that a tree-child sequence for a tree-child network~$N$ reduces another tree-child network~$N'$ if and only if~$N'$ is contained in~$N$ (Theorem~\ref{the:SubnetworkIffTCSReduces}).
Following this, we provide a linear-time algorithm for {\sc Network Containment} for inputs of tree-child networks (Algorithm~\ref{alg:Subnetwork}).
A summary of all results that relate reduction of a network and subnetwork / containment can be found in Table~\ref{tab:Results}.

% \todoRemie{Question: do we have counterexamples for all of the directions not shown in the table?
% Yes. First 3 can be seen from Figure 10. 4th one can be seen from two tree-child networks of different classes that can be reduced by the same sequence.}
\begin{table}[h]
\caption{\remiee{Main results of the paper relating network containment and CPSs. 
The two networks $N$ and $N'$ are assumed to have the same leaf-sets.
The $N$ and $N'$ columns show the assumptions on both networks, where TC stands for tree-child. The direction column shows whether containment of $N'$ in $N$ implies~($\implies$), is implies by~($\impliedby$), or is equivalent to~($\iff$) reduction of $N'$ by any CPS for $N$. The subnetwork column contains a \checkmark when $N'$ can be assumed to be a subnetwork instead of a contained network of $N$. For results in which only one direction is presented, we have counter-examples showing that the converse does not hold.}}
\label{tab:Results}
\def\arraystretch{1.5}
\centering
\begin{tabular}{|c||c|c|c|c|}
\hline
\textbf{Result} &
  \textbf{$N$} &
  \textbf{$N'$} &
  \textbf{Direction} &
  \multicolumn{1}{l|}{\textbf{Subnetwork}} \\ \hline\hline
Lemma~\ref{lem:reductionImpliesContainment}     & \multicolumn{2}{c|}{Binary}                       & $\impliedby$  & \checkmark     \\ \hline
Theorem~\ref{thm:BNBRedimpliesCon}   & Binary               & Non-binary                            & $\impliedby$  &                \\ \hline
Lemma~\ref{lem:reductionImpliesContainmentSF}     & \multicolumn{2}{c|}{Same reconstructible class} & $\impliedby$  &                \\ \hline
Corollary~\ref{cor:ContainmentImpliesReductionTC} & \multicolumn{2}{c|}{Non-binary TC}              & $\implies$    & \checkmark     \\ \hline
Theorem~\ref{the:SubnetworkIffTCSReduces}   & \multicolumn{2}{c|}{Same reconstructible TC class}    & $\iff$        &                \\ \hline
Theorem~\ref{the:SubnetworkIffTCSReduces3}   & \multicolumn{2}{c|}{Semi-binary TC}                  & $\iff$        & \checkmark     \\ \hline
\end{tabular}
\end{table}

\paragraph{Structure of the paper}  
In Section~\ref{sec:Preliminaries}, we recall all relevant definitions and %generalize some concepts introduced in previous papers. 
outline how to construct networks from cherry-picking sequences. In this construction, one can choose from multiple options.
We show that, upon choosing one of these options to consistently use in the construction, only some networks are defined by their sequences.
That is, there exist networks that cannot be constructed from the sequences that reduce them (Corollary~\ref{cor:Unique(A,B)construction}).
In Section~\ref{sec:CPNproperties}, we investigate properties of cherry-picking networks, and show that the order in which cherries are picked does not matter (Theorem~\ref{thm:OrderDoesn'tMatter}).
Furthermore, we show that networks are unique up to a particular minimal cherry-picking sequence that reduces them, given an order on the set of species.
In Section~\ref{sec:reductioncontainment}, we show that if a sequence for a network~$N$ reduces another network~$N'$, then~$N'$ is contained in~$N$ (Lemma~\ref{lem:SBSFdispaysIsBinaryResolvedSubnetwork}).
We also give a counter-example for why the converse does not hold (Theorem~\ref{thm:CPNTCFails}).
% each cherry-picking network has a \emph{minimal} cherry-picking sequence,
% we look at cherry-picking networks in general, and what correspondences there are between cherry-picking networks and their cherry-picking sequences. This includes a short discussion of their distinguishability. 
In Section~\ref{sec:TCNetworkContainment}, we restrict our attention to tree-child networks.%, which form a subclass of cherry-picking networks. 
We show that a sequence for a tree-child network~$N$ reduces another network~$N'$ if and only if~$N'$ is contained in~$N$ (Theorem~\ref{the:SubnetworkIffTCSReduces}).
In Section~\ref{sec:Computation}, we use this characterization in an algorithm for {\sc Network Containment} for tree-child networks, and show that its running time is linear.
% To make this algorithm efficient, we prove additional statements about tree-child sequences.
%, the cherry-picking sequences that nicely correspond to tree-child networks. These statements loosely translate to ``cherries in a tree-child network can be picked in any order to get a minimal CPS.'' 
% We give the algorithm and prove that its running time is linear.
We also show that, by defining an ordering on the leaves, it is possible to check whether two cherry-picking networks are isomorphic in polynomial time (Theorem~\ref{thm:CPNIsomPolyTime}).
In Section~\ref{sec:Implementation}, we present an efficient implementation for solving {\sc Network Containment} in Python, and show that the theoretical linear running time is achievable in practice.
We test our implementation on simulated data, and show that even for large data sets (1000 leaves and 1000 reticulations), the software outputs the solution within a tenth of a second.
In Section~\ref{sec:Discussion}, we conclude with open problems and future directions for the use of cherry-picking strategies.
%elaborate on extensions to problems where input networks or trees are allowed to have different sets of species, and on the possible use of cherry-picking sequences in a version of {\sc Hybridization} where the input consists of networks. 

\section{Preliminaries}\label{sec:Preliminaries}

\begin{definition}
A \emph{phylogenetic network} $N$ is a directed acyclic multigraph with one \emph{root} (the outdegree-1 source), a set $L(N)$ of \emph{leaves} (indegree-1 sinks) bijectively labelled with a set $X$, and all other nodes are either \emph{tree nodes} (indegree-1 outdegree at least 2) or \emph{reticulations} (indegree at least 2, outdegree-1).

A phylogenetic network is \emph{semi-binary} if each tree node has outdegree $2$, and it is \emph{binary} if it is semi-binary and each reticulation has indegree $2$.
\end{definition}

In the rest of this paper, we drop the `phylogenetic' term as each network in this paper is a phylogenetic network. 
To make the assumptions on the degrees of the network nodes clear, we always mention in the statement of a claim whether a network has to be binary, semi-binary, or there are no assumptions on the degrees. In the last case, we call the network \emph{non-binary} even though it may be semi-binary or even binary.
The following definition gives us a way of relating two networks that are not of the same nature.

\begin{definition}
Let $N$ and $N'$ be non-binary networks on $X$. 
Then $N$ is a \emph{refinement} of $N'$ if $N'$ can be obtained from $N$ by contracting some edges.
\end{definition}

An edge feeding into a reticulation is called a \emph{reticulation edge}.
% We write~$v\in N$ to denote that~$v$ is a node in~$N$.
Given an edge~$uv$ in~$N$, we say that~$u$ is a \emph{parent} of~$v$ and that~$v$ is a \emph{child} of~$u$.
The node~$u$ is \emph{above}~$v$, or~$v$ is \emph{below}~$u$ if there is a directed path from~$u$ to~$v$ in~$N$.
We also call~$u$ and~$v$ the \emph{tail} and \emph{head} of the edge~$uv$, respectively.
We call an edge~$uv$ an \emph{rr-edge} if~$u$ and~$v$ are both reticulations, a \emph{tr-edge} if~$u$ is a tree vertex and~$v$ a reticulation, an \emph{rt-edge} if~$u$ is a reticulation and~$v$ a tree-vertex, and a \emph{tt-edge} if~$u$ and~$v$ are both tree-vertices.
We call a directed path~$u_1u_2\ldots u_n$ an \emph{rr-path} if~$u_i$ is a reticulation for all~$i\in[n]=\{1,\ldots,n\}$, an \emph{rtr-path} if~$u_1$ and~$u_n$ are reticulations and~$u_i$ is a tree vertex for at least one~$2\le i\le n-1$, a \emph{trt-path} if~$u_1$ and~$u_n$ are tree vertices and~$u_i$ is a reticulation for at least one~$2\le i\le n-1$, and a \emph{tt-path} if~$u_i$ is a tree vertex for all~$i\in[n]$.

A network~$N$ is \emph{stack-free} if~$N$ contains no rr-edges.
A network~$N$ is \emph{tree-child} if it is stack-free and every tree node in~$N$ is a parent of a tree node or a leaf.
The \emph{reticulation number} of a network is the total number of reticulation edges minus the total number of reticulations.

Finally, we introduce the idea of adding vertices to a network.
Let~$x$ be a leaf in a network with a parent vertex~$p$.
\emph{Adding a vertex~$q$ directly above~$x$} is the action of deleting the edge~$px$ and adding two edges~$pq$ and~$qx$.

\subsection{Cherry-picking sequences}
In this subsection, we introduce cherry-picking sequences and their action on networks. This starts with definitions 
of specific structures within the networks called cherries and reticulated cherries. 
We define what it means to \emph{reduce} such structures from networks, and show that reversing such reductions---called \emph{adding} pairs to networks---can be done in many ways.
We show that these additions can be applied to some sequence of ordered pairs of leaves to obtain a network. 
We impose conditions on the sequences to ensure that these additions are well-defined, and, in doing so, we formally define cherry-picking sequences and cherry-picking networks.
See Figure~\ref{fig:Definition} for an illustration of the terms defined in this subsection.

\subsubsection{Reducible pairs}

\begin{definition}\label{def:ReduciblePair}
Let $(x,y)$ be an ordered pair of leaves in a non-binary network $N$, and let $p_x,p_y$ denote the parents of $x,y$ respectively.
We call $(x,y)$ a \emph{cherry} if $p_x = p_y$, that is, if $x$ and $y$ share a common parent.
We call $(x,y)$ a \emph{reticulated cherry} if $p_x$ is a reticulation,~$p_y$ is a tree vertex, and $p_y$ is a parent of $p_x$.
If~$(x,y)$ is a cherry or a reticulated cherry in~$N$, we say $(x,y)$ is a \emph{reducible pair}.
\end{definition}

We may \emph{reduce} cherries and reticulated cherries from a network to obtain a network of smaller size.

\begin{definition}\label{def:Reducing}
Let $N$ be a network and let $(x,y)$ be an ordered pair of leaves. 
\emph{Reducing} $(x,y)$ in $N$ is the action of
\begin{itemize}
    \item deleting $x$ and suppressing its parent node in $N$ if $(x,y)$ is a cherry in $N$;
    \item deleting the reticulation edge between the parents of~$x$ and~$y$ and subsequently suppressing degree-$2$ nodes, if~$(x,y)$ is a reticulated cherry in~$N$;
    \item doing nothing to~$N$ otherwise.
\end{itemize}
In all cases, the resulting network is denoted~$N(x,y)$. We sometimes refer to this as \emph{picking a cherry or pair~$(x,y)$ from~$N$}.
\end{definition}

Given a network~$N$ and a sequence of ordered pairs~$S$, we denote by~$NS$ the network obtained by repeatedly reducing~$N$ with each element of~$S$ in order.
We say that~$S$ reduces~$N$ if~$NS$ is a network with a single leaf (for any leaf in~$N$), a root, and no other vertices.
In particular, we call these networks \emph{single-leaf networks}.

\subsubsection{Adding pairs to networks}

As each reduction makes a simple change to a network, it is natural to attempt to reverse this change. 
Such reversals can be done by adding a leaf to obtain a new cherry in the network, or by adding a reticulation edge to create a new reticulated cherry. 
If the reduction involved the pair $(x,y)$, then we call the reverse action \emph{adding~$(x,y)$ to the network}.
Since we allow for non-binary networks, it is possible to reduce reticulated cherries with a \emph{multi-reticulation} (a reticulation with indegree at least 3).
Because of this, there may not be a unique way to add the reticulation edge back: we have the option of choosing an existing reticulation vertex or a newly created reticulation vertex as the head of this reticulation edge.
% either add the edge to a newly created reticulation node, or we can add the edge to an already existing one. 

A similar observation can be made for tree nodes. 
Just like multi-reticulations, reductions may pick cherries or reticulated cherries that contain \emph{multifurcations} (tree nodes of outdegree more than~$2$).
Here, we have the option of choosing an existing tree vertex or a newly created tree vertex as the tail of the inserted edge.
% Since the tail of the inserted edge (for both cherries and reticulated cherries) is always a tree vertex,
% instead of creating new tree vertices each time, we may append these edges to existing tree vertices.
% The construction where an edge is connected to an existing tree node whenever possible, results in a network with many \emph{multifurcations} (tree nodes of outdegree more than 2). 
%Multifurcations generally arise as a result of not knowing the order in how consecutive speciation events have unfolded in the past.
%We do not consider this here, and instead focus on the case where each tree node corresponds to speciation events which have resulted in the production of two species.
%we would like to find networks where the tree signal is clear, hence we do not consider this construction type.
%Therefore, our two construction methods, which we introduce below, will yield either a semi-binary network with no stacks or a binary network.

With this in mind, there are~$6$ ways of adding~$(x,y)$ to a network:~$2$ ways of adding cherries and~$4$ ways of adding reticulated cherries.
\begin{definition}\label{def:AddingPairs}
Let~$N$ be a non-binary network with a reducible pair~$(x,y)$.
Let~$p_x$ and~$p_y$ denote the parents of~$x$ and~$y$ in~$N(x,y)$, respectively (note that~$x$ and~$p_x$ may not be nodes in~$N(x,y)$ if~$(x,y)$ is a cherry in~$N$).
Then we may \emph{add~$(x,y)$ to~$N(x,y)$} to obtain~$N$ by using one of the following six \emph{constructions}:
\begin{enumerate}
    \item If~$x$ is not a leaf in~$N(x,y)$ (i.e., if~$(x,y)$ is a cherry in~$N$), then add a labelled node~$x$, add a node~$q$ directly above~$y$, and add an edge~$qx$.
    \begin{enumerate}
        \item Do not contract any edges; or\label{Const:CherryResolved}
        \item If~$p_y$ is a tree node, then contract~$p_yq$.\label{Const:CherryUnresolved}
    \end{enumerate}
    
    \item If~$x$ is a leaf in~$N(x,y)$ (i.e., if~$(x,y)$ is a reticulated cherry in~$N$), then add nodes~$p,q$ directly above~$x,y$, respectively, and add an edge~$qp$.
    \begin{enumerate}
        \item Do not contract any edges;\label{Const:RCherryResolved}
        \item If~$p_x$ is a reticulation, then contract~$p_xp$;\label{Const:RCherrySF}
        \item If~$p_y$ is a tree vertex, then contract~$p_yq$; or\label{Const:RCherryStack}
        \item If~$p_x$ is a reticulation, contract~$p_xp$; and, if~$p_y$ is a tree vertex, contract~$p_yq$.\label{Const:RCherryUnresolved} 
    \end{enumerate}
    
    % \begin{enumerate}
    %     \item add a labelled node~$x$, insert a node~$q$ directly above~$y$, and add an edge~$(q,x)$; or
    %     \item if~$p_y$ is a tree node, then add a labelled node~$x$ and add an edge~$(p_y,x)$.
    % \end{enumerate}
    % \item If~$x$ is a leaf in~$N(x,y)$ (i.e., if~$(x,y)$ is a reticulated cherry in~$N$), then
    % \begin{enumerate}
    %     \item add nodes~$p,q$ directly above~$x,y$, respectively, and add edges~$(q,p)$; or\label{Const:RCherryResolved}
    %     \item if~$p_x$ is a reticulation, insert a node~$q$ directly above~$y$ and add an edge~$(q,p_x)$; or\label{Const:RCherrySF}
    %     \item if~$p_y$ is a tree vertex, insert a node~$p$ directly above~$x$ and add an edge~$(p_y,p)$; or\label{Const:RCherryStack}
    %     \item if~$p_x$ is a reticulation and if~$p_y$ is a tree vertex, then add an edge~$(p_y,p_x)$.\label{Const:RCherryUnresolved}
    % \end{enumerate}
\end{enumerate}
\end{definition}

Since all tree vertices have indegree-$1$ and all reticulations have outdegree-$1$, there are no other ways of adding a reducible pair to a network other than the six ways mentioned above.
Note that the constructions~\ref{Const:CherryUnresolved},~\ref{Const:RCherrySF},~\ref{Const:RCherryStack}, and~\ref{Const:RCherryUnresolved} may only be used if the `if' conditions are satisfied.
Also note that the above actions are only well-defined if~$y$ is a leaf in~$N(x,y)$.
In the setting of Definition~\ref{def:AddingPairs}, this is not an issue: since we assume that~$(x,y)$ is a reducible pair of~$N$, it is indeed the case that~$y$ is a leaf in~$N(x,y)$.

On the other hand, if we were to start with any sequence of ordered pairs and sought to construct a network by successively adding ordered pairs backwards through the sequence, the story would be a little different.
That is, we may come across a case where, upon trying to add a reducible pair~$(x,y)$ to a network,~$y$ does not already exist in the network as a leaf.
Let~$S = S_1S_2\cdots S_{|S|} = (x_1,y_1),(x_2,y_2),\ldots (x_{|S|},y_{|S|})$ be a sequence of ordered pairs.
Starting with a network on a single leaf~$y_{|S|}$, we may iteratively add~$S_i$ to the network for~$i=|S|,|S|-1,\ldots,1$ (i.e., backwards through the sequence~$S$), choosing a suitable construction for each ordered pair, to obtain some network.
We call this \emph{a network obtained from~$S$}.
Now, if~$y_i$ was not a leaf in the network when adding~$S_i$, then such a construction would not be well-defined.
Fortunately, we can fix this by imposing a simple condition on the sequences.
This motivates the following definition.

\begin{definition}\label{def:CPS}
A \emph{cherry-picking sequence (CPS)} on a set~$X$ is a sequence of ordered pairs on distinct elements from~$X$, such that the second coordinate of each ordered pair occurs as a first coordinate in some ordered pair in the rest of the sequence, or as the second coordinate of the last pair. 
\end{definition}
%Descriptive names for first and second coordinates?

Returning to the example that we had before, we observe if~$S$ was a CPS, then~$y_i$ must already have been a leaf in the network when adding~$S_i=(x_i,y_i)$.
By definition of CPSs,~$y_i$ appears as a first coordinate in some ordered pair that appears after~$S_i$, or~$y_i$ appears as the second coordinate of the final ordered pair, which implies that the network contains the leaf~$y_i$ in both cases when adding~$S_i=(x_i,y_i)$.
Therefore, this construction is well-defined, and we can always obtain a network from a CPS.
This brings us to the definition of a cherry-picking network.

\begin{definition}
A network on~$X$ is a \emph{cherry-picking network (CPN)} if it can be obtained from some CPS~$S$.
Equivalently, a CPN is a network that can be reduced by some CPS.
\end{definition}

In particular, single-leaf networks are also CPNs, since these can be reduced by the empty CPS.
By definition, a CPN with at least two leaves contains either a cherry or a reticulated cherry.
Intuitively, reducing these structures returns a network of smaller size that is a CPN; we may repeatedly reduce cherries and reticulated cherries until the network has been reduced.

A \emph{subsequence} of a CPS refers to any sequence of ordered pairs that can be obtained by deleting some elements from the CPS.
Note that a subsequence need not be a CPS.
In what follows, we will often have to reduce a network by a subsequence of a CPS. These subsequences are most often the initial parts of the sequence, and hence we introduce notation for them.
% Hence, we introduce notation that makes it easy to refer to the initial part of a sequence. 
Let~$S = (x_1,y_1),(x_2,y_2), \ldots (x_n,y_n)$ be a CPS.
For~$i\in[n]$, we use the following notations to denote some subsequence of~$S$.
The~$i$th ordered pair of~$S$ is $S_i = (x_i,y_i)$.
The first~$i$ ordered pairs in~$S$ is denoted by~$S_{[:i]} = (x_1,y_1),\dots, (x_i, y_i)$.
The subsequence of~$S$ without the first~$i$ ordered pairs is denoted by~$S_{[i+1:]} = (x_{i+1},y_{i+1}), (x_{i+2},y_{i+2}), \dots, (x_n,y_n)$.
We let~$S_{[:0]}$ denote the empty sequence.

A CPS~$S$ is \emph{minimal} for a CPN~$N$ if~$S$ reduces~$N$ and each ordered pairs~$S_i$ of~$S$ reduces something in~$NS_{[:i-1]}$ for all~$i\in[|S|]$.
In other words,~$NS_{[:i-1]}\neq NS_{[:i]}$ for all~$i\in [|S|]$.
We often write \emph{a CPS of/for a network~$N$} to refer to a minimal CPS that for~$N$.

A \emph{partial CPS}~$S'$ of length~$i$ is a sequence of ordered pairs such that there exists a CPS~$S$ where~$S_{[:i]} = S'$.
If $S$ and $S'$ are partial CPSs and $N$ is a non-binary network, then applying $S$ and then $S'$ is the same as appending $S'$ to $S$, denoted $SS'$, and applying the whole sequence. In notation, we write
\[(NS)S'=N(SS'),\]
and hence we denote this network without brackets as $NSS'$.

\begin{observation}\label{obs:CPNAfterReducingfirstielementsremainCPN}
Let~$N$ be a non-binary CPN that can be reduced by a CPS~$S$.
Then the network~$NS_{[:i]}$ is a CPN for all~$i=1,\dots,|S|$.
\end{observation}

By choosing a suitable construction, we may obtain a CPN from any of its minimal CPSs.

\begin{observation}
Every non-binary CPN can be obtained from a minimal CPS that reduces it.
\end{observation}
% \begin{proof}
% Let~$N$ be a non-binary CPN with a minimal CPS~$S$.
% We may construct the network~$N$ from~$S$ by constructing~$NS_{[:i-1]}$ from~$NS_{[:i]}$ for~$i=1,\ldots, |S|$ using Definition~\ref{def:AddingPairs}.
% By choosing a suitable construction from Definition~\ref{def:AddingPairs} for each reducible pair in~$S$, we obtain~$N$.
% Since~$S$ is a CPS, the construction is well-defined.
% \end{proof}

\subsection{CPN Classes}

Using different combination of the six constructions from Definition~\ref{def:AddingPairs} can yield different CPNs from the same CPS.
These differences could be due to the nature of the network vertex degrees (binary, semi-binary, non-binary) or from their topological features (stack-free, tree-child).
One way of categorizing these CPNs is to choose and stay consistent with one particular construction for adding cherries and reticulated cherries to networks.
That is, we construct networks from CPSs with a chosen construction~$A$ to add cherries and a chosen construction~$B$ to add reticulated cherries.

% We first show that it is possible to obtain two networks from the same CPS, for which one network is a binary refinement of the other.
% Later on, we show that choosing an construction for cherries and reticulated cherries and staying consistent with it will return a unique network within that class.
% Some may be more refined than others, and others may exclude certain topological features.
% We first focus on the refinedness of the constructed CPNs.

There are two motivations to do so.
Firstly, this categorizes CPNs into classes defined by their topological restrictions.
We may specify classes of CPNs that contain only binary networks, those that contain only semi-binary networks, those without stacks, and many more.
Secondly, and more importantly, we can introduce some notion of a correspondence between CPNs and minimal CPSs that reduce them.
Within some CPN classes, it turns out that if a sequence is minimal for two networks, then the networks must be isomorphic.
% As shown in Section~\ref{subsubsec:Refinement}, different networks may be obtained from the same CPS whenever different constructions are applied.
% We show here that within a class of CPNs, each CPS gives rise to a unique networks.

% We note that the CPN classes that we mention here differs to the network classes that are widely used in phylogenetics.
% Phylogenetic network classes refer to topologically restricted classes of networks such as tree-child, stack-free, tree-based, and level-$k$.
% CPN classes categorize CPNs into those that can be obtained from a CPS via a chosen cherry construction, and a chosen reticulated cherry construction.\todoYuki{maybe delete}

\begin{definition}
Let~$A$ and~$B$ be a cherry construction and a reticulated cherry construction, respectively.
We let~$(A,B)$ denote the class of all CPNs that can be obtained from CPSs by using the suitable constructions~$A$ or~$B$.
\end{definition}

Within the CPN class~$(A,B)$, we say that we use the~\emph{$(A,B)$-construction} to obtain CPNs from CPSs.
We write~$N\in(A,B)$ or say that~$N$ is an~$(A,B)$-CPN to mean that~$N$ is a CPN in the CPN class~$(A,B)$.

Since there are two cherry constructions and four reticulated cherry constructions, there are in total eight CPN classes.
For example, the CPN class~$(\ref{Const:CherryResolved}, \ref{Const:RCherryResolved})$ contains all binary CPNs (see Figure~\ref{fig:Definition} for an example of obtaining a~$(\ref{Const:CherryResolved},\ref{Const:RCherryResolved})$-class CPN from a CPS).
We note that it is possible obtain the same CPN from the same CPS within different CPN classes.
Indeed, a CPS corresponding to a tree will give the same network in all the CPN classes that use the \ref{Const:CherryResolved} cherry construction, and the same can be said for CPN classes that use the \ref{Const:CherryUnresolved} cherry construction.
A CPS that is long enough (with enough leaves and reticulations) returns a network that is distinct amongst the different CPN classes.
An example of this is shown in Figure~\ref{fig:CPNClass}.

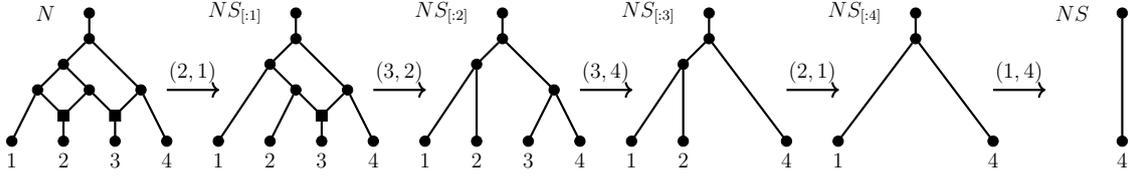
\begin{figure}
    \centering
    $S=(2,1),(3,2),(3,4),(2,1),(1,4)$
    
    \vspace{0.5cm}
    \resizebox{\columnwidth}{!}{
    \begin{tikzpicture}[every node/.style = {draw, circle, fill, inner sep = 0pt, minimum size = 2mm},
    square/.style = {regular polygon, regular polygon sides = 4, minimum size = 3 mm}]
    \tikzset {edge/.style = {very thick, shorten >= -0.5 pt}}
        %Nodes

        \node[] (-1) at (0.0,-0.0){};
        \node[] (0) at (0.0,-0.5){};
        \node[] (5) at (-0.5,-1.0){};
        \node[] (8) at (1,-1.5){};
        \node[] (6) at (-1.0,-1.5){};
        \node[] (7) at (0.0,-1.5){};
        \node[square] (9) at (-0.5,-2.0){};
        \node[square] (10) at (0.5,-2.0){};

        \node[draw=none, fill=none, left = 5mm of -1] {\large{$N$}};
        %Leaves

        \node[] (1) at (-1.5,-2.5){};
        \node[draw=none, fill=none, below=1mm of 1] (leaf_1){\large $1$};
        \node[] (2) at (-0.5,-2.5){};
        \node[draw=none, fill=none, below=1mm of 2] (leaf_2){\large $2$};
        \node[] (3) at (0.5,-2.5){};
        \node[draw=none, fill=none, below=1mm of 3] (leaf_3){\large $3$};
        \node[] (4) at (1.5,-2.5){};
        \node[draw=none, fill=none, below=1mm of 4] (leaf_4){\large $4$};

        %Edges

        \draw[edge] (-1) edge (0);
        \draw[edge] (0) edge (5);
        \draw[edge] (0) edge (8);
        \draw[edge] (5) edge (6);
        \draw[edge] (5) edge (7);
        \draw[edge] (6) edge (1);
        \draw[edge] (6) edge (9);
        \draw[edge] (9) edge (2);
        \draw[edge] (7) edge (9);
        \draw[edge] (7) edge (10);
        \draw[edge] (10) edge (3);
        \draw[edge] (8) edge (10);
        \draw[edge] (8) edge (4);

        \begin{scope}[shift = {(4,0)}]
        %Nodes

        \node[] (-1) at (0.0,-0.0){};
        \node[] (0) at (0.0,-0.5){};
        \node[] (5) at (-0.5,-1.0){};
        \node[] (8) at (1,-1.5){};
        \node[] (7) at (0.0,-1.5){};
        \node[square] (10) at (0.5,-2.0){};

        \node[draw=none, fill=none, left = 5mm of -1] {\large{$NS_{[:1]}$}};
        %Leaves

        \node[] (1) at (-1.5,-2.5){};
        \node[draw=none, fill=none, below=1mm of 1] (leaf_1){\large $1$};
        \node[] (2) at (-0.5,-2.5){};
        \node[draw=none, fill=none, below=1mm of 2] (leaf_2){\large $2$};
        \node[] (3) at (0.5,-2.5){};
        \node[draw=none, fill=none, below=1mm of 3] (leaf_3){\large $3$};
        \node[] (4) at (1.5,-2.5){};
        \node[draw=none, fill=none, below=1mm of 4] (leaf_4){\large $4$};

        %Edges

        \draw[edge] (-1) edge (0);
        \draw[edge] (0) edge (5);
        \draw[edge] (0) edge (8);
        \draw[edge] (5) edge (1);
        \draw[edge] (5) edge (7);
        \draw[edge] (7) edge (2);
        \draw[edge] (7) edge (10);
        \draw[edge] (10) edge (3);
        \draw[edge] (8) edge (10);
        \draw[edge] (8) edge (4);
        \end{scope}
        
        \begin{scope}[shift = {(8,0)}]
        %Nodes

        \node[] (-1) at (0.0,-0.0){};
        \node[] (0) at (0.0,-0.5){};
        \node[] (5) at (-0.5,-1.0){};
        \node[] (8) at (1,-1.5){};

        \node[draw=none, fill=none, left = 5mm of -1] {\large{$NS_{[:2]}$}};
        %Leaves

        \node[] (1) at (-1.5,-2.5){};
        \node[draw=none, fill=none, below=1mm of 1] (leaf_1){\large $1$};
        \node[] (2) at (-0.5,-2.5){};
        \node[draw=none, fill=none, below=1mm of 2] (leaf_2){\large $2$};
        \node[] (3) at (0.5,-2.5){};
        \node[draw=none, fill=none, below=1mm of 3] (leaf_3){\large $3$};
        \node[] (4) at (1.5,-2.5){};
        \node[draw=none, fill=none, below=1mm of 4] (leaf_4){\large $4$};

        %Edges

        \draw[edge] (-1) edge (0);
        \draw[edge] (0) edge (5);
        \draw[edge] (0) edge (8);
        \draw[edge] (5) edge (1);
        \draw[edge] (5) edge (2);
        \draw[edge] (8) edge (3);
        \draw[edge] (8) edge (4);
        \end{scope}
        
        \begin{scope}[shift = {(12,0)}]
        %Nodes

        \node[] (-1) at (0.0,-0.0){};
        \node[] (0) at (0.0,-0.5){};
        \node[] (5) at (-0.5,-1.0){};

        \node[draw=none, fill=none, left = 5mm of -1] {\large{$NS_{[:3]}$}};
        %Leaves

        \node[] (1) at (-1.5,-2.5){};
        \node[draw=none, fill=none, below=1mm of 1] (leaf_1){\large $1$};
        \node[] (2) at (-0.5,-2.5){};
        \node[draw=none, fill=none, below=1mm of 2] (leaf_2){\large $2$};
        \node[] (4) at (1.5,-2.5){};
        \node[draw=none, fill=none, below=1mm of 4] (leaf_4){\large $4$};

        %Edges

        \draw[edge] (-1) edge (0);
        \draw[edge] (0) edge (5);
        \draw[edge] (0) edge (4);
        \draw[edge] (5) edge (1);
        \draw[edge] (5) edge (2);
        \end{scope}
        
        \begin{scope}[shift = {(16,0)}]
        %Nodes

        \node[] (-1) at (0.0,-0.0){};
        \node[] (0) at (0.0,-0.5){};

        \node[draw=none, fill=none, left = 5mm of -1] {\large{$NS_{[:4]}$}};
        %Leaves

        \node[] (1) at (-1.5,-2.5){};
        \node[draw=none, fill=none, below=1mm of 1] (leaf_1){\large $1$};
        \node[] (4) at (1.5,-2.5){};
        \node[draw=none, fill=none, below=1mm of 4] (leaf_4){\large $4$};

        %Edges

        \draw[edge] (-1) edge (0);
        \draw[edge] (0) edge (1);
        \draw[edge] (0) edge (4);
        \end{scope}
        
        \begin{scope}[shift = {(20,0)}]
        %Nodes

        \node[] (-1) at (0.0,-0.0){};

        \node[draw=none, fill=none, left = 5mm of -1] {\large{$NS$}};
        %Leaves

        \node[] (4) at (0,-2.5){};
        \node[draw=none, fill=none, below=1mm of 4] (leaf_4){\large $4$};

        %Edges

        \draw[edge] (-1) edge (4);
        \end{scope}
        
        \draw[edge] (1.5,-1.5) edge [->] node[draw = none, fill = none, midway, yshift=10pt]{\large $(2,1)$} (2.5,-1.5);
        
        \draw[edge] (5.5,-1.5) edge [->] node[draw = none, fill = none, midway, yshift=10pt]{\large $(3,2)$} (6.5,-1.5);
        
        \draw[edge] (9.5,-1.5) edge [->] node[draw = none, fill = none, midway, yshift=10pt]{\large $(3,4)$} (10.5,-1.5);
        
        \draw[edge] (13.5,-1.5) edge [->] node[draw = none, fill = none, midway, yshift=10pt]{\large $(2,1)$} (14.5,-1.5);
        
        \draw[edge] (17.5,-1.5) edge [->] node[draw = none, fill = none, midway, yshift=10pt]{\large $(1,4)$} (18.5,-1.5);
        
    \end{tikzpicture}
    }
    \caption{A binary cherry-picking network~$N$ reduced to a leaf~$4$ by a CPS~$S$. 
    The reduction is shown as a sequence of networks~$NS_{[:i]}$ for~$i=0,1,\ldots,5$ from left to right, in which an element of~$S$ is applied to the network successively.
    This sequence is minimal for the network, as every element of the sequence reduces either a cherry or a reticulated cherry of the network.
    An example of a cherry~$(3,4)$ can be seen in the network~$NS_{[:2]}$, and a reticulated cherry~$(2,1)$ can be seen in the network~$N$.
    The reduction of both reducible pairs is carried out as in Definition~\ref{def:ReduciblePair}.
    Observe that this sequence is not a tree-child sequence, as the element~$2$ appears as a first coordinate in~$S_1$ and as a second coordinate in~$S_2$.
    Constructing a network from a sequence~$S$ in the class~$(\ref{Const:CherryResolved}, \ref{Const:RCherryResolved})$ can be seen by moving through the six networks in reverse order (from right to left).}
    \label{fig:Definition}
\end{figure}

% \begin{observation}\label{obs:UniqueConstruction}
% Let~$S$ be a CPS.
% Using the~$(A,B)$-construction returns a unique network~$N$, for which~$S$ is a minimal CPS.
% \end{observation}

% The properties of the networks within these CPN classes differ due to their construction methods.
% For example, networks in classes that use the reticulated cherry constructions~\ref{Const:RCherrySF} or~\ref{Const:RCherryUnresolved} do not contain rr-edges.
% Networks in classes~$(\ref{Const:CherryResolved},\ref{Const:RCherryStack})$ and~$(\ref{Const:CherryResolved},\ref{Const:RCherryUnresolved})$ have the property that the grandparents of all reticulations are again reticulations or 

% We briefly comment on the characteristics of networks contained in each of these classes with Table~\ref{tab:CPNClass}.

% \begin{table}[]
% \caption{Reconstructible CPN classes and their characterizations.
% For each CPN class, a checkmark (\checkmark) is given if all networks within the class satisfies the condition.}
% \label{tab:CPNClass}
% \begin{tabular}{|c|c|c|c|c|}
% \hline
%           & Binary      & Semi-binary   & No rr-paths   & No tt-paths   \\ \hline
% $(1a,2a)$ & \checkmark  & \checkmark    &               &               \\ \hline
% $(1a,2a)$ &             & \checkmark    & \checkmark    &               \\ \hline
% $(1a,2a)$ &             &               &               & \checkmark    \\ \hline
% $(1a,2a)$ &             &               & \checkmark    & \checkmark    \\ \hline
% \end{tabular}
% \end{table}

\begin{figure}
    \centering
    $S = (3,2),(1,2),(3,2),(3,4),(2,4)$
    
    \vspace{0.5cm}
    \begin{subfigure}[b]{.24\textwidth}
    \begin{tikzpicture}[every node/.style = {draw, circle, fill, inner sep = 0pt, minimum size = 2mm},
    square/.style = {regular polygon, regular polygon sides = 4, minimum size = 3 mm}]
    \tikzset {edge/.style = {very thick, shorten >= -0.5 pt}}

        %Nodes

        \node[] (-1) at (0.0,-0.0) {};
        \node[] (0) at (0.0,-0.5) {};
        \node[] (5) at (-0.5,-1.0) {};
        \node[] (6) at (0.5,-1.0) {};
        \node[] (7) at (-1.0,-1.5) {};
        \node[square] (8) at (0.0,-1.5) {};
        \node[] (9) at (-0.5,-2.0) {};
        \node[square] (10) at (0.5,-2.5) {};

        \node[draw=none, fill=none, left = 3mm of 0] {\large{$(\ref{Const:CherryResolved}, \ref{Const:RCherryResolved})$}};
        %Leaves

        \node[] (1) at (-1.5,-3.0) {};
        \node[draw=none, fill=none, below=1mm of 1] (leaf_1) {\large $1$};
        \node[] (2) at (-0.5,-3.0) {};
        \node[draw=none, fill=none, below=1mm of 2] (leaf_2) {\large $2$};
        \node[] (3) at (0.5,-3.0) {};
        \node[draw=none, fill=none, below=1mm of 3] (leaf_3) {\large $3$};
        \node[] (4) at (1.5,-3.0) {};
        \node[draw=none, fill=none, below=1mm of 4] (leaf_4) {\large $4$};

        %Edges

        \draw[edge] (-1) edge (0);
        \draw[edge] (0) edge (5);
        \draw[edge] (0) edge (6);
        \draw[edge] (5) edge (7);
        \draw[edge] (5) edge (8);
        \draw[edge] (6) edge (8);
        \draw[edge] (6) edge (4);
        \draw[edge] (7) edge (1);
        \draw[edge] (7) edge (9);
        \draw[edge] (8) edge (10);
        \draw[edge] (9) edge (10);
        \draw[edge] (9) edge (2);
        \draw[edge] (10) edge (3);
    \end{tikzpicture}
    \end{subfigure}
    \begin{subfigure}[b]{.24\textwidth}
    \begin{tikzpicture}[every node/.style = {draw, circle, fill, inner sep = 0pt, minimum size = 2mm},
    square/.style = {regular polygon, regular polygon sides = 4, minimum size = 3 mm}]
    \tikzset {edge/.style = {very thick, shorten >= -0.5 pt}}

        %Nodes

        \node[] (-1) at (0.0,-0.0) {};
        \node[] (0) at (0.0,-0.5) {};
        \node[] (5) at (-0.5,-1.0) {};
        \node[] (6) at (0.5,-1.0) {};
        \node[] (7) at (-1.0,-1.5) {};
        \node[] (9) at (-0.5,-2.0) {};
        \node[square] (10) at (0.5,-2.5) {};

        \node[draw=none, fill=none, left = 3mm of 0] {\large{$(\ref{Const:CherryResolved}, \ref{Const:RCherrySF})$}};
        %Leaves

        \node[] (1) at (-1.5,-3.0) {};
        \node[draw=none, fill=none, below=1mm of 1] (leaf_1) {\large $1$};
        \node[] (2) at (-0.5,-3.0) {};
        \node[draw=none, fill=none, below=1mm of 2] (leaf_2) {\large $2$};
        \node[] (3) at (0.5,-3.0) {};
        \node[draw=none, fill=none, below=1mm of 3] (leaf_3) {\large $3$};
        \node[] (4) at (1.5,-3.0) {};
        \node[draw=none, fill=none, below=1mm of 4] (leaf_4) {\large $4$};

        %Edges

        \draw[edge] (-1) edge (0);
        \draw[edge] (0) edge (5);
        \draw[edge] (0) edge (6);
        \draw[edge] (5) edge (7);
        \draw[edge] (5) edge (10);
        \draw[edge] (6) edge (10);
        \draw[edge] (6) edge (4);
        \draw[edge] (7) edge (1);
        \draw[edge] (7) edge (9);
        \draw[edge] (9) edge (10);
        \draw[edge] (9) edge (2);
        \draw[edge] (10) edge (3);
    \end{tikzpicture}
    \end{subfigure}
    \begin{subfigure}[b]{.24\textwidth}
    \begin{tikzpicture}[every node/.style = {draw, circle, fill, inner sep = 0pt, minimum size = 2mm},
    square/.style = {regular polygon, regular polygon sides = 4, minimum size = 3 mm}]
    \tikzset {edge/.style = {very thick, shorten >= -0.5 pt}}

        %Nodes

        \node[] (-1) at (0.0,-0.0) {};
        \node[] (0) at (0.0,-0.5) {};
        \node[] (6) at (0.5,-1.0) {};
        \node[] (7) at (-1.0,-1.5) {};
        \node[square] (8) at (0.0,-1.5) {};
        \node[square] (10) at (0.5,-2.5) {};

        \node[draw=none, fill=none, left = 3mm of 0] {\large{$(\ref{Const:CherryResolved}, \ref{Const:RCherryStack})$}};
        %Leaves

        \node[] (1) at (-1.5,-3.0) {};
        \node[draw=none, fill=none, below=1mm of 1] (leaf_1) {\large $1$};
        \node[] (2) at (-0.5,-3.0) {};
        \node[draw=none, fill=none, below=1mm of 2] (leaf_2) {\large $2$};
        \node[] (3) at (0.5,-3.0) {};
        \node[draw=none, fill=none, below=1mm of 3] (leaf_3) {\large $3$};
        \node[] (4) at (1.5,-3.0) {};
        \node[draw=none, fill=none, below=1mm of 4] (leaf_4) {\large $4$};

        %Edges

        \draw[edge] (-1) edge (0);
        \draw[edge] (0) edge (7);
        \draw[edge] (0) edge (6);
        \draw[edge] (0) edge (8);
        \draw[edge] (6) edge (8);
        \draw[edge] (6) edge (4);
        \draw[edge] (7) edge (1);
        \draw[edge] (8) edge (10);
        \draw[edge] (7) edge (10);
        \draw[edge] (7) edge (2);
        \draw[edge] (10) edge (3);
    \end{tikzpicture}
    \end{subfigure}
    \begin{subfigure}[b]{.24\textwidth}
    \begin{tikzpicture}[every node/.style = {draw, circle, fill, inner sep = 0pt, minimum size = 2mm},
    square/.style = {regular polygon, regular polygon sides = 4, minimum size = 3 mm}]
    \tikzset {edge/.style = {very thick, shorten >= -0.5 pt}}

        %Nodes

        \node[] (-1) at (0.0,-0.0) {};
        \node[] (0) at (0.0,-0.5) {};
        \node[] (6) at (0.5,-1.0) {};
        \node[] (7) at (-1.0,-1.5) {};
        \node[square] (10) at (0.5,-2.5) {};

        \node[draw=none, fill=none, left = 3mm of 0] {\large{$(\ref{Const:CherryResolved}, \ref{Const:RCherryUnresolved})$}};
        %Leaves

        \node[] (1) at (-1.5,-3.0) {};
        \node[draw=none, fill=none, below=1mm of 1] (leaf_1) {\large $1$};
        \node[] (2) at (-0.5,-3.0) {};
        \node[draw=none, fill=none, below=1mm of 2] (leaf_2) {\large $2$};
        \node[] (3) at (0.5,-3.0) {};
        \node[draw=none, fill=none, below=1mm of 3] (leaf_3) {\large $3$};
        \node[] (4) at (1.5,-3.0) {};
        \node[draw=none, fill=none, below=1mm of 4] (leaf_4) {\large $4$};

        %Edges

        \draw[edge] (-1) edge (0);
        \draw[edge] (0) edge (7);
        \draw[edge] (0) edge (6);
        \draw[edge] (0) edge (10);
        \draw[edge] (6) edge (10);
        \draw[edge] (6) edge (4);
        \draw[edge] (7) edge (1);
        \draw[edge] (7) edge (10);
        \draw[edge] (7) edge (2);
        \draw[edge] (10) edge (3);
    \end{tikzpicture}
    \end{subfigure}
    
    \vspace{1cm}
    
    \begin{subfigure}[b]{.24\textwidth}
    \begin{tikzpicture}[every node/.style = {draw, circle, fill, inner sep = 0pt, minimum size = 2mm},
    square/.style = {regular polygon, regular polygon sides = 4, minimum size = 3 mm}]
    \tikzset {edge/.style = {very thick, shorten >= -0.5 pt}}

        %Nodes

        \node[] (-1) at (0.0,-0.0) {};
        \node[] (0) at (0.0,-0.5) {};
        \node[] (5) at (-0.5,-1.0) {};
        \node[square] (8) at (0.0,-1.5) {};
        \node[] (9) at (-0.5,-2.0) {};
        \node[square] (10) at (0.5,-2.5) {};

        \node[draw=none, fill=none, left = 3mm of 0] {\large{$(\ref{Const:CherryUnresolved}, \ref{Const:RCherryResolved})$}};
        %Leaves

        \node[] (1) at (-1.5,-3.0) {};
        \node[draw=none, fill=none, below=1mm of 1] (leaf_1) {\large $1$};
        \node[] (2) at (-0.5,-3.0) {};
        \node[draw=none, fill=none, below=1mm of 2] (leaf_2) {\large $2$};
        \node[] (3) at (0.5,-3.0) {};
        \node[draw=none, fill=none, below=1mm of 3] (leaf_3) {\large $3$};
        \node[] (4) at (1.5,-3.0) {};
        \node[draw=none, fill=none, below=1mm of 4] (leaf_4) {\large $4$};

        %Edges

        \draw[edge] (-1) edge (0);
        \draw[edge] (0) edge (5);
        \draw[edge] (5) edge (8);
        \draw[edge] (0) edge (8);
        \draw[edge] (0) edge (4);
        \draw[edge] (5) edge (1);
        \draw[edge] (5) edge (9);
        \draw[edge] (8) edge (10);
        \draw[edge] (9) edge (10);
        \draw[edge] (9) edge (2);
        \draw[edge] (10) edge (3);
    \end{tikzpicture}
    \end{subfigure}
    \begin{subfigure}[b]{.24\textwidth}
    \begin{tikzpicture}[every node/.style = {draw, circle, fill, inner sep = 0pt, minimum size = 2mm},
    square/.style = {regular polygon, regular polygon sides = 4, minimum size = 3 mm}]
    \tikzset {edge/.style = {very thick, shorten >= -0.5 pt}}

        %Nodes

        \node[] (-1) at (0.0,-0.0) {};
        \node[] (0) at (0.0,-0.5) {};
        \node[] (5) at (-0.5,-1.0) {};
        \node[] (9) at (-0.5,-2.0) {};
        \node[square] (10) at (0.5,-2.5) {};

        \node[draw=none, fill=none, left = 3mm of 0] {\large{$(\ref{Const:CherryUnresolved}, \ref{Const:RCherrySF})$}};
        %Leaves

        \node[] (1) at (-1.5,-3.0) {};
        \node[draw=none, fill=none, below=1mm of 1] (leaf_1) {\large $1$};
        \node[] (2) at (-0.5,-3.0) {};
        \node[draw=none, fill=none, below=1mm of 2] (leaf_2) {\large $2$};
        \node[] (3) at (0.5,-3.0) {};
        \node[draw=none, fill=none, below=1mm of 3] (leaf_3) {\large $3$};
        \node[] (4) at (1.5,-3.0) {};
        \node[draw=none, fill=none, below=1mm of 4] (leaf_4) {\large $4$};

        %Edges

        \draw[edge] (-1) edge (0);
        \draw[edge] (0) edge (5);
        \draw[edge] (5) edge (10);
        \draw[edge] (0) edge (10);
        \draw[edge] (0) edge (4);
        \draw[edge] (5) edge (1);
        \draw[edge] (5) edge (9);
        \draw[edge] (9) edge (10);
        \draw[edge] (9) edge (2);
        \draw[edge] (10) edge (3);
    \end{tikzpicture}
    \end{subfigure}
    \begin{subfigure}[b]{.24\textwidth}
    \begin{tikzpicture}[every node/.style = {draw, circle, fill, inner sep = 0pt, minimum size = 2mm},
    square/.style = {regular polygon, regular polygon sides = 4, minimum size = 3 mm}]
    \tikzset {edge/.style = {very thick, shorten >= -0.5 pt}}

        %Nodes

        \node[] (-1) at (0.0,-0.0) {};
        \node[] (0) at (0.0,-0.5) {};
        \node[square] (8) at (0.0,-1.5) {};
        \node[square] (10) at (0.5,-2.5) {};

        \node[draw=none, fill=none, left = 3mm of 0] {\large{$(\ref{Const:CherryUnresolved}, \ref{Const:RCherryStack})$}};
        %Leaves

        \node[] (1) at (-1.5,-3.0) {};
        \node[draw=none, fill=none, below=1mm of 1] (leaf_1) {\large $1$};
        \node[] (2) at (-0.5,-3.0) {};
        \node[draw=none, fill=none, below=1mm of 2] (leaf_2) {\large $2$};
        \node[] (3) at (0.5,-3.0) {};
        \node[draw=none, fill=none, below=1mm of 3] (leaf_3) {\large $3$};
        \node[] (4) at (1.5,-3.0) {};
        \node[draw=none, fill=none, below=1mm of 4] (leaf_4) {\large $4$};

        %Edges

        \draw[edge] (-1) edge (0);
        \draw[edge, bend right=10] (0) edge (8);
        \draw[edge, bend left=10] (0) edge (8);
        \draw[edge] (0) edge (4);
        \draw[edge] (0) edge (1);
        \draw[edge] (8) edge (10);
        \draw[edge] (0) edge (10);
        \draw[edge] (0) edge (2);
        \draw[edge] (10) edge (3);
    \end{tikzpicture}
    \end{subfigure}
    \begin{subfigure}[b]{.24\textwidth}
    \begin{tikzpicture}[every node/.style = {draw, circle, fill, inner sep = 0pt, minimum size = 2mm},
    square/.style = {regular polygon, regular polygon sides = 4, minimum size = 3 mm}]
    \tikzset {edge/.style = {very thick, shorten >= -0.5 pt}}

        %Nodes

        \node[] (-1) at (0.0,-0.0) {};
        \node[] (0) at (0.0,-0.5) {};
        \node[square] (10) at (0.5,-2.5) {};

        \node[draw=none, fill=none, left = 3mm of 0] {\large{$(\ref{Const:CherryUnresolved}, \ref{Const:RCherryUnresolved})$}};
        %Leaves

        \node[] (1) at (-1.5,-3.0) {};
        \node[draw=none, fill=none, below=1mm of 1] (leaf_1) {\large $1$};
        \node[] (2) at (-0.5,-3.0) {};
        \node[draw=none, fill=none, below=1mm of 2] (leaf_2) {\large $2$};
        \node[] (3) at (0.5,-3.0) {};
        \node[draw=none, fill=none, below=1mm of 3] (leaf_3) {\large $3$};
        \node[] (4) at (1.5,-3.0) {};
        \node[draw=none, fill=none, below=1mm of 4] (leaf_4) {\large $4$};

        %Edges

        \draw[edge] (-1) edge (0);
        \draw[edge, bend right=10] (0) edge (10);
        \draw[edge, bend left = 10] (0) edge (10);
        \draw[edge] (0) edge (4);
        \draw[edge] (0) edge (1);
        \draw[edge] (0) edge (10);
        \draw[edge] (0) edge (2);
        \draw[edge] (10) edge (3);
    \end{tikzpicture}
    \end{subfigure}
    \caption{A CPS~$S$ and the unique networks obtained from it within the eight CPN classes.
    Observe that the eight networks are distinct.}
    \label{fig:CPNClass}
\end{figure}
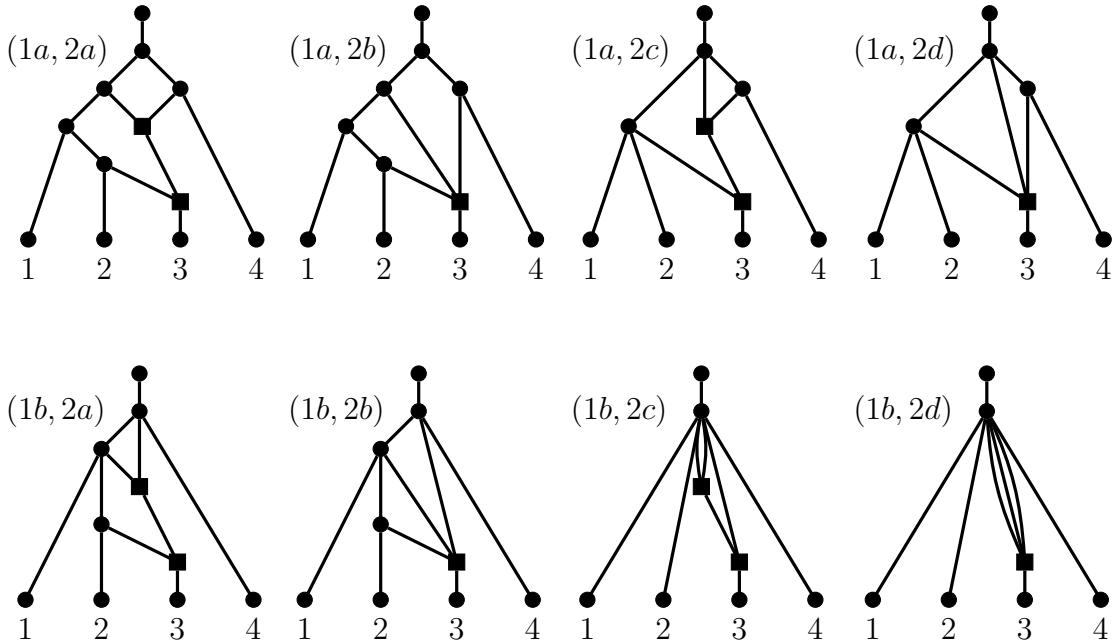

Suppose that we are given a CPN~$N$ within a CPN class~$(A,B)$, and let~$S$ be a minimal CPS that reduces~$N$.
To form some notion of correspondence between CPNs and the sequences that reduce them, we pose the following question:
is it always the case that applying the~$(A,B)$ construction on~$S$ returns the network~$N$?
It turns out that this is true only for half of the CPN classes.
We start by defining what it means for a CPN class to be \emph{reconstructible}.

\begin{definition}
A CPN class~$C=(A,B)$ is called \emph{reconstructible} if for any two networks~$N,N'\in C$ with a common minimal CPS, we have that~$N$ and~$N'$ are isomorphic.
\end{definition}

Since the construction is fixed, each CPS gives rise to a unique network within each of the CPN classes.
Then if two distinct networks~$N$ and~$N'$ have a common minimal CPS~$S$, at most one of these networks, say~$N$, can be constructed from the sequence.
This means that although~$S$ is a minimal CPS of~$N'$, it cannot be used to construct the network~$N'$.
Indeed, there does exist some minimal CPS of~$N'$ which can be used to construct~$N'$. 
Reconstructible CPN classes have the nice property that for a given CPN~$N$, \emph{any} minimal CPS for~$N$ can be used to construct~$N$.

\begin{lemma}\label{lem:Unique(A,B)Construction}
Let~$(A,B) \in \{(\ref{Const:CherryResolved},\ref{Const:RCherryResolved}); (\ref{Const:CherryResolved},\ref{Const:RCherrySF}); (\ref{Const:CherryUnresolved},\ref{Const:RCherryStack});
(\ref{Const:CherryUnresolved},\ref{Const:RCherryUnresolved})\}$.
Let~$N$ be a CPN in the~$(A,B)$-class, and let~$(x,y)$ be a reducible pair in~$N$.
Then adding~$(x,y)$ to~$N(x,y)$ using the~$(A,B)$ construction results in~$N$.
\end{lemma}
\begin{proof}
%\todoRemie{The proof seems slightly incomplete: you have to note somewhere that when adding a pair, no rr-edge is subdivided with a tree node, and no tt-edge is subdivided with a reticulation node. Otherwise, the order in which you add and contract stuff matters. \textbf{TODO YUKI}}
Observe that these four classes are characterised by the following properties.
The networks in~$(\ref{Const:CherryResolved},\ref{Const:RCherryResolved})$ are binary; the networks in~$(\ref{Const:CherryResolved},\ref{Const:RCherrySF})$ do not contain rr-edges; the networks in~$(\ref{Const:CherryResolved},\ref{Const:RCherrySF})$ do not contain tt-edges; and
the networks in~$(\ref{Const:CherryUnresolved},\ref{Const:RCherryStack})$ do not contain rr-edges nor tt-edges.
Since~$N$ is a network of one of these classes,~$N$ must also have these properties.
\yuki{Furthermore, the network~$N(x,y)$ also has these properties, since deleting edges and potentially suppressing vertices does not create new vertices, which may subdivide existing rr-edges or tt-edges.}
This means that upon inserting an edge to the network (as a result of adding a cherry or a reticulated cherry~$(x,y)$), one should either do nothing; contract all rr-edges; contract all tt-edges; or contract all rr-edges and tt-edges, depending on which class of CPNs are being considered.
\yuki{Note that these contractions, should they occur, only involve vertices that have just been added as a result of adding the reducible pair, since~$N(x,y)$ also has the properties.}
This is precisely what happens when we add~$(x,y)$ back to the network~$N(x,y)$ using the respective constructions, and it follows immediately that the constructions defined in these CPN classes returns the original network~$N$.
\end{proof}

Lemma~\ref{lem:Unique(A,B)Construction} states that four of the eight CPN classes have the property that adding back a reduced pair to the network returns the original network. 
By applying this lemma in the construction of a network from a CPS, we obtain the following corollary.

\begin{corollary}\label{cor:Unique(A,B)construction}
Let~$(A,B) \in \{(\ref{Const:CherryResolved},\ref{Const:RCherryResolved}); (\ref{Const:CherryResolved},\ref{Const:RCherrySF}); (\ref{Const:CherryUnresolved},\ref{Const:RCherryStack});
(\ref{Const:CherryUnresolved},\ref{Const:RCherryUnresolved})\}$.
Then~$(A,B)$ is reconstructible.
\end{corollary}

% The implications of Corollary~\ref{cor:Unique(A,B)construction} can be rephrased as follows.
% Let~$N$ be a CPN in the~$(A,B)$-class, and let~$S = (x_1,y_1),\ldots, (x_n,y_n)$ be a minimal CPS for~$N$.
% Then, adding each pair from~$S$ to the CPN with the single leaf~$y_n$, using the~$(A,B)$ construction, results in~$N$.

\begin{figure}
    \centering
    \begin{subfigure}[b]{.49\textwidth}
    \centering
    \begin{tikzpicture}[every node/.style = {draw, circle, fill, inner sep = 0pt, minimum size = 2mm},
square/.style = {regular polygon, regular polygon sides = 4, minimum size = 3 mm}]
    \tikzset {edge/.style = {very thick, shorten >= -0.5 pt}}

        %Nodes

        \node[] (-1) at (0.0,-0.0) {};
        \node[] (0) at (0.0,-0.5) {};
        \node[] (5) at (0.5,-1.0) {};
        \node[square] (4) at (0,-1.5) {};
        \node[square] (6) at (0.0,-2.0) {};

        \node[draw=none, fill=none, left = 5mm of 0] {\large{$N$}};
        %Leaves

        \node[] (1) at (-1.0,-2.5) {};
        \node[draw=none, fill=none, below=1mm of 1] (leaf_1) {\large $1$};
        \node[] (2) at (0.0,-2.5) {};
        \node[draw=none, fill=none, below=1mm of 2] (leaf_2) {\large $2$};
        \node[] (3) at (1.0,-2.5) {};
        \node[draw=none, fill=none, below=1mm of 3] (leaf_3) {\large $3$};

        %Edges

        \draw[edge] (-1) edge (0);
        \draw[edge] (0) edge (1);
        \draw[edge] (0) edge (4);
        \draw[edge] (0) edge (5);
        \draw[edge] (4) edge (6);
        \draw[edge] (5) edge (4);
        \draw[edge] (5) edge (6);
        \draw[edge] (5) edge (3);
        \draw[edge] (6) edge (2);
        
        \begin{scope}[shift={(3,0)}]
        %Nodes

        \node[] (-1) at (0.0,-0.0) {};
        \node[] (0) at (0.0,-0.5) {};
        \node[] (4) at (-0.5,-1.0) {};
        \node[square] (5) at (-0.25,-1.5) {};
        \node[square] (6) at (0.0,-2.0) {};

        \node[draw=none, fill=none, left = 5mm of 0] {\large{$N'$}};
        %Leaves

        \node[] (1) at (-1.0,-2.5) {};
        \node[draw=none, fill=none, below=1mm of 1] (leaf_1) {\large $1$};
        \node[] (2) at (0.0,-2.5) {};
        \node[draw=none, fill=none, below=1mm of 2] (leaf_2) {\large $2$};
        \node[] (3) at (1.0,-2.5) {};
        \node[draw=none, fill=none, below=1mm of 3] (leaf_3) {\large $3$};

        %Edges

        \draw[edge] (-1) edge (0);
        \draw[edge] (0) edge (4);
        \draw[edge] (0) edge (5);
        \draw[edge] (0) edge (6);
        \draw[edge] (0) edge (3);
        \draw[edge] (4) edge (1);
        \draw[edge] (4) edge (5);
        \draw[edge] (5) edge (6);
        \draw[edge] (6) edge (2);
        \end{scope}
    \end{tikzpicture}
    \caption{$(\ref{Const:CherryResolved}, \ref{Const:RCherryStack}) : (2,3),(2,1),(2,3),(1,3)$}
    \end{subfigure}
    \begin{subfigure}[b]{.49\textwidth}
    \centering
    \begin{tikzpicture}[every node/.style = {draw, circle, fill, inner sep = 0pt, minimum size = 2mm},
square/.style = {regular polygon, regular polygon sides = 4, minimum size = 3 mm}]
    \tikzset {edge/.style = {very thick, shorten >= -0.5 pt}}

        %Nodes

        \node[] (-1) at (0.0,-0.0) {};
        \node[] (0) at (0.0,-0.5) {};
        \node[] (5) at (0.5,-1.0) {};
        \node[square] (6) at (0.0,-2.0) {};

        \node[draw=none, fill=none, left = 5mm of 0] {\large{$N$}};
        %Leaves

        \node[] (1) at (-1.0,-2.5) {};
        \node[draw=none, fill=none, below=1mm of 1] (leaf_1) {\large $1$};
        \node[] (2) at (0.0,-2.5) {};
        \node[draw=none, fill=none, below=1mm of 2] (leaf_2) {\large $2$};
        \node[] (3) at (1.0,-2.5) {};
        \node[draw=none, fill=none, below=1mm of 3] (leaf_3) {\large $3$};

        %Edges

        \draw[edge] (-1) edge (0);
        \draw[edge] (0) edge (1);
        \draw[edge] (0) edge (6);
        \draw[edge] (0) edge (5);
        \draw[edge, bend left = 20] (5) edge (6);
        \draw[edge, bend right = 20] (5) edge (6);
        \draw[edge] (5) edge (3);
        \draw[edge] (6) edge (2);
        
        \begin{scope}[shift={(3,0)}]
        %Nodes

        \node[] (-1) at (0.0,-0.0) {};
        \node[] (0) at (0.0,-0.5) {};
        \node[] (4) at (-0.5,-1.0) {};
        \node[square] (6) at (0.0,-2.0) {};

        \node[draw=none, fill=none, left = 5mm of 0] {\large{$N'$}};
        %Leaves

        \node[] (1) at (-1.0,-2.5) {};
        \node[draw=none, fill=none, below=1mm of 1] (leaf_1) {\large $1$};
        \node[] (2) at (0.0,-2.5) {};
        \node[draw=none, fill=none, below=1mm of 2] (leaf_2) {\large $2$};
        \node[] (3) at (1.0,-2.5) {};
        \node[draw=none, fill=none, below=1mm of 3] (leaf_3) {\large $3$};

        %Edges

        \draw[edge] (-1) edge (0);
        \draw[edge] (0) edge (4);
        \draw[edge, bend left = 20] (0) edge (6);
        \draw[edge, bend right = 20] (0) edge (6);
        \draw[edge] (0) edge (3);
        \draw[edge] (4) edge (1);
        \draw[edge] (4) edge (6);
        \draw[edge] (6) edge (2);
        \end{scope}
    \end{tikzpicture}
    \caption{$(\ref{Const:CherryResolved}, \ref{Const:RCherryUnresolved}) : (2,3),(1,3),(2,3),(1,3)$}
    \end{subfigure}
    
    \vspace{1cm}
    
    \begin{subfigure}[b]{.49\textwidth}
    \centering
    \begin{tikzpicture}[every node/.style = {draw, circle, fill, inner sep = 0pt, minimum size = 2mm},
square/.style = {regular polygon, regular polygon sides = 4, minimum size = 3 mm}]
    \tikzset {edge/.style = {very thick, shorten >= -0.5 pt}}

        %Nodes

        \node[] (-1) at (0.0,-0.0) {};
        \node[] (0) at (0.0,-0.5) {};
        \node[] (4) at (-0.5,-1.0) {};
        \node[square] (5) at (0,-1.5) {};

        \node[draw=none, fill=none, left = 5mm of 0] {\large{$N$}};
        %Leaves

        \node[] (1) at (-1.0,-2.0) {};
        \node[draw=none, fill=none, below=1mm of 1] (leaf_1) {\large $1$};
        \node[] (2) at (0.0,-2.0) {};
        \node[draw=none, fill=none, below=1mm of 2] (leaf_2) {\large $2$};
        \node[] (3) at (1.0,-2.0) {};
        \node[draw=none, fill=none, below=1mm of 3] (leaf_3) {\large $3$};

        %Edges

        \draw[edge] (-1) edge (0);
        \draw[edge] (0) edge (4);
        \draw[edge] (0) edge (5);
        \draw[edge] (0) edge (3);
        \draw[edge] (4) edge (1);
        \draw[edge] (4) edge (5);
        \draw[edge] (5) edge (2);
        
        \begin{scope}[shift={(3,0)}]
        %Nodes

        \node[] (-1) at (0.0,-0.0) {};
        \node[] (0) at (0.0,-0.5) {};
        \node[] (4) at (0.5,-1.0) {};
        \node[square] (5) at (0,-1.5) {};

        \node[draw=none, fill=none, left = 5mm of 0] {\large{$N'$}};
        %Leaves

        \node[] (1) at (-1.0,-2.0) {};
        \node[draw=none, fill=none, below=1mm of 1] (leaf_1) {\large $1$};
        \node[] (2) at (0.0,-2.0) {};
        \node[draw=none, fill=none, below=1mm of 2] (leaf_2) {\large $2$};
        \node[] (3) at (1.0,-2.0) {};
        \node[draw=none, fill=none, below=1mm of 3] (leaf_3) {\large $3$};

        %Edges

        \draw[edge] (-1) edge (0);
        \draw[edge] (0) edge (4);
        \draw[edge] (0) edge (5);
        \draw[edge] (4) edge (3);
        \draw[edge] (0) edge (1);
        \draw[edge] (4) edge (5);
        \draw[edge] (5) edge (2);
        \end{scope}
        
    \end{tikzpicture}
    \caption{$(\ref{Const:CherryUnresolved}, \ref{Const:RCherryResolved}): (2,1),(2,1),(1,3)$}
    \end{subfigure}
    \begin{subfigure}[b]{.49\textwidth}
    \centering
    \begin{tikzpicture}[every node/.style = {draw, circle, fill, inner sep = 0pt, minimum size = 2mm},
square/.style = {regular polygon, regular polygon sides = 4, minimum size = 3 mm}]
    \tikzset {edge/.style = {very thick, shorten >= -0.5 pt}}

        %Nodes

        \node[] (-1) at (0.0,-0.0) {};
        \node[] (0) at (0.0,-0.5) {};
        \node[] (4) at (-0.5,-1.0) {};
        \node[] (6) at (0.5,-1.0) {};
        \node[square] (5) at (0.0,-1.5) {};

        \node[draw=none, fill=none, left = 5mm of 0] {\large{$N$}};
        %Leaves

        \node[] (1) at (-1.0,-2.0) {};
        \node[draw=none, fill=none, below=1mm of 1] (leaf_1) {\large $1$};
        \node[] (2) at (0.0,-2.0) {};
        \node[draw=none, fill=none, below=1mm of 2] (leaf_2) {\large $2$};
        \node[] (3) at (1.0,-2.0) {};
        \node[draw=none, fill=none, below=1mm of 3] (leaf_3) {\large $3$};

        %Edges

        \draw[edge] (-1) edge (0);
        \draw[edge] (0) edge (4);
        \draw[edge] (0) edge (5);
        \draw[edge] (0) edge (6);
        \draw[edge] (4) edge (1);
        \draw[edge] (4) edge (5);
        \draw[edge] (5) edge (2);
        \draw[edge] (6) edge (5);
        \draw[edge] (6) edge (3);
        
        \begin{scope}[shift = {(3,0)}]
        %Nodes

        \node[] (-1) at (0.0,-0.0) {};
        \node[] (0) at (0.0,-0.5) {};
        \node[] (4) at (0.25,-0.75) {};
        \node[] (6) at (0.5,-1.0) {};
        \node[square] (5) at (0.0,-1.5) {};

        \node[draw=none, fill=none, left = 5mm of 0] {\large{$N'$}};
        %Leaves

        \node[] (1) at (-1.0,-2.0) {};
        \node[draw=none, fill=none, below=1mm of 1] (leaf_1) {\large $1$};
        \node[] (2) at (0.0,-2.0) {};
        \node[draw=none, fill=none, below=1mm of 2] (leaf_2) {\large $2$};
        \node[] (3) at (1.0,-2.0) {};
        \node[draw=none, fill=none, below=1mm of 3] (leaf_3) {\large $3$};

        %Edges

        \draw[edge] (-1) edge (0);
        \draw[edge] (0) edge (4);
        \draw[edge] (4) edge (5);
        \draw[edge] (4) edge (6);
        \draw[edge] (0) edge (1);
        \draw[edge] (0) edge (5);
        \draw[edge] (5) edge (2);
        \draw[edge] (6) edge (5);
        \draw[edge] (6) edge (3);
        \end{scope}
    \end{tikzpicture}
    \caption{$(\ref{Const:CherryUnresolved}, \ref{Const:RCherrySF}): (2,1),(2,1),(2,3),(1,3)$}
    \end{subfigure}
    \caption{Two distinct networks~$N$ and~$N'$ that can be reduced by the same minimal CPS for the $(\ref{Const:CherryResolved},\ref{Const:RCherryStack}), (\ref{Const:CherryResolved},\ref{Const:RCherryUnresolved}), (\ref{Const:CherryUnresolved},\ref{Const:RCherryResolved}), (\ref{Const:CherryUnresolved},\ref{Const:RCherrySF})$-classes.
    The networks obtained by using the respective constructions are~$N$.
    This means that given a network and a minimum sequence that reduces it, the sequence cannot always be used to construct the original network (let~$N'$ be the original network in these four cases).}
    \label{fig:CPNNonUniqueConstruction}
\end{figure}
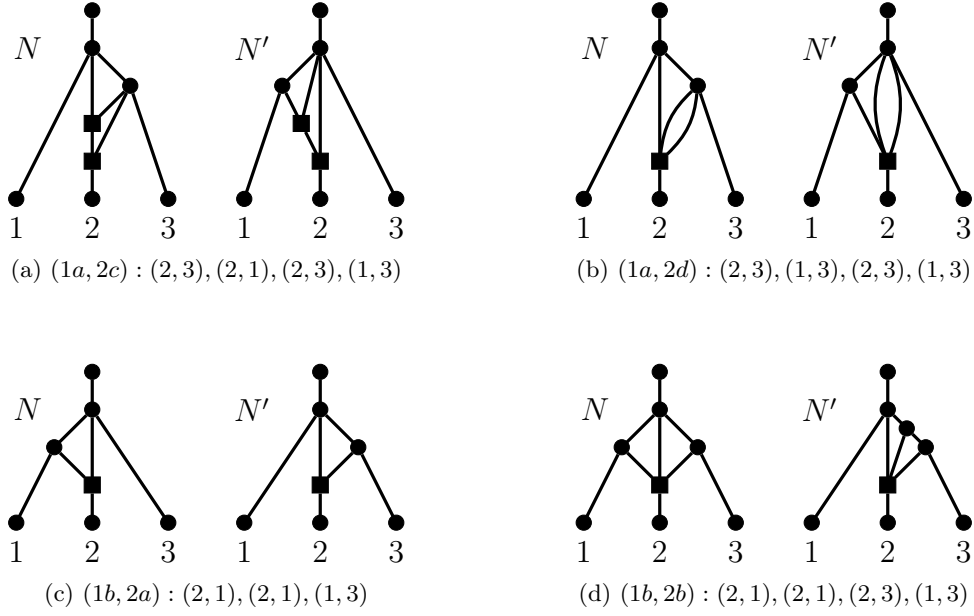

% The reason why these particular four classes are reconstructible lies in the nature of two special edges within the networks.
% Recall that rr-edges and tt-edges denote edges of the network whose endvertices are both reticulations and tree vertices, respectively.
% We consider networks on at least~$3$ leaves.
% These properties allow for adding reticulated cherries to be well-defined.
% Since contracting particular edges (tt-edges or rr-edges) are fixed in these four classes, the construction is well-defined.

% The reason why these particular four CPN classes are not reconstructible lies in the nature of the networks themselves.
To show that the above lemma and the corollary do not hold for the other four CPN classes, we present two networks with a common minimal CPS for each of the CPN classes in Figure~\ref{fig:CPNNonUniqueConstruction}.
Unlike their reconstructible counter-parts, the networks in these four classes can contain both tt-edges and rr-edges whilst also containing multifurcations and multi-reticulations. 
When constructing networks from CPSs, this allows for a mixture of choosing to contract some tt-edges and some rr-edges, but not all. 
This can make adding reticulated cherries problematic.
Take the~$(\ref{Const:CherryResolved},\ref{Const:RCherryStack})$ class for example.
Since there can exist tt-edges that are binary, we may, in particular, assume that a network in the class contains a reticulated cherry~$(x,y)$ where the parent of~$y$ is a head of a binary tt-edge~$e$.
But this means that upon reducing~$(x,y)$ and adding back the reticulated cherry using the~\ref{Const:RCherryStack} construction, we essentially contract this tt-edge, which returns a different network (see Figure~\ref{fig:(1a,2c)NonUnique}).

\begin{figure}
    \centering
    \begin{tikzpicture}[every node/.style = {draw, circle, fill, inner sep = 0pt, minimum size = 2mm},
    square/.style = {regular polygon, regular polygon sides = 4, minimum size = 3 mm}]
    \tikzset {edge/.style = {very thick, shorten >= -0.5 pt}}

        %Nodes
        \node[] (7) at (-0.5,0){};
        \node[] (4) at (0.5,0){};
        \node[] (5) at (0.5,-0.5){};
        \node[] (6) at (0.5,-1.0){};
        \node[square] (3) at (-0.5,-1.5){};

        %Leaves

        \node[] (1) at (-0.5,-2.0){};
        \node[draw=none, fill=none, below=1mm of 1] (leaf_1){\large $x$};
        \node[] (2) at (0.5,-2.0){};
        \node[draw=none, fill=none, below=1mm of 2] (leaf_2){\large $y$};

        %Edges

        \draw[edge] (7) edge (3);
        \draw[edge] (3) edge (1);
        \draw[edge] (4) edge (5);
        \draw[edge] (5) edge (1,-1);
        \draw[edge] (5) edge (6);
        \draw[edge] (6) edge (3);
        \draw[edge] (6) edge (2);
        
        \begin{scope}[shift = {(5,0)}]
        %Nodes
        
        \node[] (7) at (-0.5,0){};
        \node[] (4) at (0.5,0){};
        \node[] (5) at (0.5,-0.5){};

        %Leaves

        \node[] (1) at (-0.5,-2.0){};
        \node[draw=none, fill=none, below=1mm of 1] (leaf_1){\large $x$};
        \node[] (2) at (0.5,-2.0){};
        \node[draw=none, fill=none, below=1mm of 2] (leaf_2){\large $y$};

        %Edges

        \draw[edge] (7) edge (1);
        \draw[edge] (4) edge (5);
        \draw[edge] (5) edge (1,-1);
        \draw[edge] (5) edge (2);
        \end{scope}
        
        \begin{scope}[shift = {(10,0)}]
        %Nodes
        \node[] (7) at (-0.5,0){};
        \node[] (4) at (0.5,0){};
        \node[] (5) at (0.5,-0.5){};
        \node[square] (3) at (-0.5,-1.5){};

        %Leaves
        
        \node[] (1) at (-0.5,-2.0){};
        \node[draw=none, fill=none, below=1mm of 1] (leaf_1){\large $x$};
        \node[] (2) at (0.5,-2.0){};
        \node[draw=none, fill=none, below=1mm of 2] (leaf_2){\large $y$};

        %Edges

        \draw[edge] (7) edge (3);
        \draw[edge] (3) edge (1);
        \draw[edge] (4) edge (5);
        \draw[edge] (5) edge (1,-1);
        \draw[edge] (5) edge (3);
        \draw[edge] (5) edge (2);
        \end{scope}
        
        \draw[edge] (1.5,-1) edge [->] 
        node[draw = none, fill = none, midway, yshift=10pt]{reduce $(x,y)$} 
        (3.5,-1);
        
        \draw[edge] (6.5,-1) edge [->] 
        node[draw = none, fill = none, midway, yshift=10pt]{add $(x,y)$} 
        node[draw = none, fill = none, midway, yshift=-10pt]{with $(\ref{Const:CherryResolved}, \ref{Const:RCherryStack})$}
        (8.5,-1);
    \end{tikzpicture}
    \caption{Reducing a reticulated cherry~$(x,y)$ and adding it back using the construction~$(\ref{Const:CherryResolved}, \ref{Const:RCherryStack})$ can return a different network.}
    \label{fig:(1a,2c)NonUnique}
\end{figure}
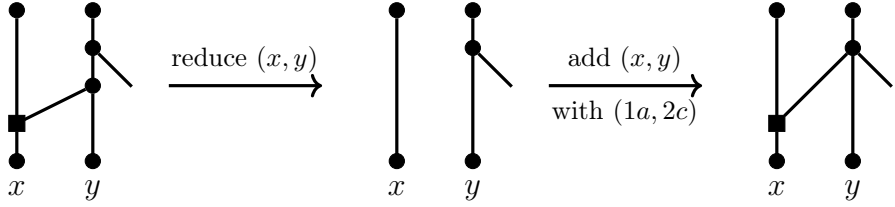

% The CPN classes that are of particular interest to us are the~$(\ref{Const:CherryResolved}, \ref{Const:RCherryResolved})$-class and the~$(\ref{Const:CherryResolved}, \ref{Const:RCherrySF})$-class.
% The~$(\ref{Const:CherryResolved}, \ref{Const:RCherryResolved})$-class contains all binary CPNs, whereas the~$(\ref{Const:CherryResolved}, \ref{Const:RCherrySF})$-class contains all stack-free semi-binary CPNs.
% In particular, the stack-free property of the networks in this CPN class follows from the fact that every time a new reticulation vertex is added, an edge connecting two reticulation vertices, should they exist, is suppressed.

% The former construction may introduce a \emph{stack} to the network: two adjacent reticulation nodes.
% In the latter construction, our networks never contain stacks. Therefore, with this construction, we retrieve \emph{stack-free} networks (sometimes called \emph{compressed}~\citep{huber2016folding}), which are networks without stacks.

\subsubsection{Refinement of constructed networks}\label{subsubsec:Refinement}

The six constructions that were introduced in Definition~\ref{def:AddingPairs} can be rephrased as follows.
When adding~$(x,y)$ to a network, check if~$x$ is a leaf in the network. 
If~$x$ is not a leaf in the network, then add a labelled leaf~$x$ and an edge from the (newly added) parent of~$y$ to~$x$ (add a cherry).
If~$x$ is a leaf in the network, then add an edge between the newly added parents of~$y$ and~$x$ (add a reticulated cherry).
Decide whether or not to contract some of the edges incident to the parent of~$x$ and edges incident to the parent of~$y$.
This means that given some CPS~$S$, the binary network~$N$ in the~$(\ref{Const:CherryResolved},\ref{Const:RCherryResolved})$-class constructed from~$S$ is a refinement of all networks that can be constructed from~$S$, using any combination of the constructions.
This gives the following observation.

% \begin{observation}\label{obs:ResolvedConstruction}
% Let~$N$ be a non-binary CPN.
% Then there exists a binary refinement of~$N$ that is a CPN.
% \end{observation}

% In particular, for every minimal CPS~$S$ of~$N$, there is a corresponding binary refinement of~$N$ that can be obtained by applying the~$(\ref{Const:CherryResolved},\ref{Const:RCherryResolved})$ construction on~$S$.

\begin{lemma}\label{lem:SBSFBinaryrefinement}
Given a \remie{non-binary} %semi-binary stack-free 
CPN~$N$ and a minimal CPS~$S$ for~$N$, there exists a binary refinement~$N_b$ of~$N$ such that~$S$ is a minimal CPS for~$N_b$.
\end{lemma}
\begin{proof}
The unique binary network $N_b$ obtained by using construction~$(\ref{Const:CherryResolved},\ref{Const:RCherryResolved})$ on $S$ is a refinement of $N$, and $S$ is a minimal CPS for $N_b$ by definition of this network.
\end{proof}

Finally, the following lemma shows how general refinements of CPNs (not necessarily binary) are related to the CPNs.

\begin{lemma}\label{lem:CPNrefinement}
Let~$N_r$ be a refinement of a non-binary network~$N$ that is a CPN.
Then~$N$ is a CPN, and every minimal CPS of~$N_r$ is also a minimal CPS of~$N$.
\end{lemma}

\begin{proof}
We prove by induction on~$|N|$, the number of edges in~$N$.
For the base case take the single-leaf network.
So suppose that for every network of size at most~$|N|-1$, the claim is true.

Let~$S$ be a minimal CPS of~$N_r$, and let~$S_1=(x,y)$ be the first element of~$S$.
Since~$N_r$ can be obtained from~$N$ by refining vertices, it must be the case that~$S_1$ is also a reducible pair in~$N$.
Furthermore, if~$S_1$ is a cherry in~$N_r$ then~$S_1$ is also a cherry in~$N$; if~$S_1$ is a reticulated cherry in~$N_r$ then~$S_1$ is also a reticulated cherry in~$N$.
Now it is easy to see that~$N_rS_1$ is a refinement of~$NS_1$
Note that~$|NS_1|<|N|$ since every reduction reduces the size of the network.
The network~$N_rS_1$ is a CPN by Observation~\ref{obs:CPNAfterReducingfirstielementsremainCPN}.
By induction hypothesis,~$NS_1$ is a CPN and every minimal CPS of~$N_rS_1$ is also a minimal CPS of~$NS_1$.
Then in particular,~$S_{[:2]}$ is a minimal CPS of~$NS_1$.
It follows then that~$S$ is a minimal CPS of~$N$.
\end{proof}

Note that the converse of Lemma~\ref{lem:CPNrefinement} does not hold in general.
Consider the tree~$T$ on three leaves~$\{x,y,z\}$ that all share a common parent (the claw graph with a root).
Let~$T_r$ be a binary refinement of~$T$ in which~$x$ and~$y$ form a cherry.
Then the CPS~$(y,z),(x,z)$ is minimal for~$T$ but not for~$T_r$.

\section{Properties of cherry-picking networks}\label{sec:CPNproperties}

In this section, we investigate properties of cherry-picking networks. First, we continue where we left off in the previous section: we inspect the relation between CPSs and CPNs. This includes the reticulation number defined by a CPS, changes in the sets of reducible pairs that are ready for picking after picking a pair, and the order in which we can reduce a network. The last of these allows us to consider distinguishability of two CPNs by their CPSs. 
Then, we use this to investigate the relation between embedded networks of a CPN and its CPSs. 
% For tree-child networks and TCSs, tree containment has been studied implicitly by \citet{linz2019attaching}. 
%\todoYuki{Move this sentence to TCS section}
% We provide explicit results in this domain, and generalize some of these results to CPNs.
%---not all results can be generalized.}\todoYuki{But they haven't? Maybe we can say that their results imply that tree containment is possible with TCSs but its not explicitly stated in their paper.}

\subsection{Why CPNs are nice: order doesn't matter}\label{sec:OrderCPS}
%reducing sequences are CPSs
%implies networks with a reducing sequence have a CPS, as a reducing sequence has a minimal reducing subsequence.
\begin{lemma}\label{lem:OptimumCPN}
Let $S$ be a minimal length sequence of ordered pairs of leaves that reduces a non-binary network $N$. Then $S$ is a CPS.
Furthermore, $|S| = n+r-1$, where~$n$ and~$r$ denote the number of leaves and the reticulation number of~$N$, respectively.
\end{lemma}
\begin{proof}
Suppose for a contradiction that~$S$ is not a CPS.
Then, there is an~$i<|S|$ with $S_i=(x,y)$ such that $y$ is not a first coordinate in any of the elements of $S_{[i+1:]}$ or the second coordinate of $S_{|S|}$. This means $y$ cannot be a leaf in $NS_{[:i-1]}$ (if it were, then $S$ does not reduce $N$). This implies $NS_{[:i-1]} = NS_{[:i]}$, and there is a shorter sequence $S_{[:i-1]}S_{[i+1:]}$ that reduces $N$, a contradiction. We conclude that $S$ is a CPS.

We now prove the second part of the lemma.
Let $S_i = (x,y)$.
We first construct a binary network~$M$ from~$S$ using the~$(\ref{Const:CherryResolved},\ref{Const:RCherryResolved})$ construction.
Upon constructing $MS_{[:i-1]}$ from $MS_{[:i]}$, 
%\todoLeo{Lemma 1 and 2 only say that this is possible when the network is either binary or stack free. I guess it works more generally for networks in which each stack is fully binary. I don't think it works for networks in which there are stacks with mutibrizations? Maybe say in the preliminaries that you will always assume that all networks are either binary or stack free.}
%\todoRemie{It does work in general, I just do not know how to quickly fix this...}
a new leaf $x$ is added if $x$ is not a leaf in $MS_{[:i]}$, and a reticulation is added otherwise.
By Lemma~\ref{lem:SBSFBinaryrefinement},~$M$ is a binary refinement of~$N$, and therefore~$N$ has the same leaf set and it has the same reticulation number as that of~$M$.
Since~$S$ is a minimal CPS for~$N$, it follows that $|S| = n+r-1$.
\end{proof}

% %remove redundant from sequence
% \begin{lemma}
% Let $S$ be a CPS, and $N$ a network. Suppose the leaves of $s_i$ do not form a cherry or a reticulated cherry in $NS_{[0,i-1]}$, then $NS_{[:i-1]} S_{[i+1:]}=NS$.
% \end{lemma}
% \begin{proof}
% Note that $NS_{[0,i-1]}S_{i}S_{[i+1:]}=NS$, and that $NS_{[0,i-1]}S_{i}=NS_{[0,i-1]}$. Putting these two together, we get 
% \[NS_{[:i-1]}S_{[i+1:]}=NS_{[0,i-1]}S_{i}S_{[i+1:]}=NS.\]
% \todo{It might not be a CPS anymore, but if we remove all the redundant pairs, it definitely is because it becomes a minimal reducing sequence for $N$.}
% \end{proof}

%sets of cherries definitions
\begin{definition}
Let $N$ be a non-binary network. Denote with $\mathcal{C}_c(N)$ the set of cherries of $N$, and with $\mathcal{C}_r(N)$ the set of reticulated cherries of $N$. The set of all reducible pairs is denoted $\mathcal{C}(N)=\mathcal{C}_c(N)\cup \mathcal{C}_r(N)$.
\end{definition}

The following lemma states that all new reducible pairs after picking a pair $(x,y)$ must involve either $x$ or $y$.

\begin{lemma}\label{lem:PossibleCherriesAfterReduction}
Let $N$ be a non-binary network on a taxa set~$X$, and let $(x,y)$ be a reducible pair of $N$. 
Then we have the following inclusion:
\[\mathcal{C}(N(x,y))\setminus \mathcal{C}(N)\subseteq \left(\{x,y\}\times X\right) \cup \left(X\times \{x,y\}\right).\]
% Furthermore, we have that
% \[\mathcal{C}_r(N(x,y))\setminus \mathcal{C}_r(N)\subseteq \left(\{x,y\}\times X\right) \cup \left(X\times \{x,y\}\right)\]
% and
% \[\mathcal{C}_c(N(x,y))\setminus \mathcal{C}_c(N)\subseteq \left(\{x,y\}\times X\right) \cup \left(X\times \{x,y\}\right).\]
% \todoRemie{Should we also do these for cherries and reticulated cherries separately? What about when the network is stack-free/tree-child (we kind of use that later).
% \textbf{YM:} I think it's fine, since the set on the right is exactly the same for both. I've added this in without proof for now.
% In part., if~$(x,y)$ was a ret. cherry in~$N$, then~$x$ can be contained in a cherry in~$N(x,y)$. \textbf{RJ:} I think I meant that when reducing cherries, we get a smaller set than when we reduce a reticulated cherry. But this seems fine as well.}
\end{lemma}
% \todoYuki{Shouldn't~$\mathcal{N}$ be $\mathcal{C}(N) \setminus (x,y)$?}
% \todoRemie{Yes, I changed it to something that's more intuitive I hope}
\begin{proof}
Note that the LHS of the containment relation represents the reducible pairs in $N(x,y)$ that were not present in $N$. Suppose, for contradiction, that this set contains a pair $(z,w)$ not involving $x$ or $y$. Then, this pair is not reducible in $N$, but it is in $N(x,y)$. Adding the pair $(x,y)$ back into $N(x,y)$ may only subdivide the pendant edges leading to $x$ and $y$. This implies that this action will not change the fact that $z$ and $w$ form a reducible pair. Therefore, $(z,w)$ is a reducible pair in $N$ as well, a contradiction. Hence, all new cherries and reticulated cherries of~$N(x,y)$ involve~$x$ or~$y$.
\end{proof}

We also have similar inclusions for looking at reducible pairs in the original network that are not reducible pairs in the new network.
The argument follows in a similar manner as the one presented in the proof of Lemma~\ref{lem:PossibleCherriesAfterReduction}, so we include it as an observation.
Roughly speaking, the following observation states that reducing a network by the element~$(x,y)$ preserves the other reducible pairs. 
%sets of cherries after reductions
\begin{observation}\label{obs:ChangingReducibleCherries}
Let $N$ be a network on~$X$, and $(x,y)$ a reducible pair of $N$. 
Then, if $N$ is non-binary, we have the inclusion $\mathcal{C}(N)\setminus \mathcal{C}(N(x,y))\subseteq\{(y,x)\}\cup \{x\}\times X$, and in particular
\[\mathcal{C}_r(N)\setminus \mathcal{C}_r(N(x,y)) \subseteq\{x\}\times X,\]
and
\[\mathcal{C}_c(N)\setminus \mathcal{C}_c(N(x,y)) \subseteq\{(y,x)\}\cup \{x\}\times X.\]
If $N$ is semi-binary, the inclusions can be sharpened to $\mathcal{C}(N)\setminus \mathcal{C}(N(x,y))\subseteq\{(x,y),(y,x)\}$, with
\[\mathcal{C}_r(N)\setminus \mathcal{C}_r(N(x,y))\subseteq\{(x,y)\},\]
and
\[\mathcal{C}_c(N)\setminus \mathcal{C}_c(N(x,y))\subseteq\{(x,y),(y,x)\}.\]
%\todoYuki{Shouldn't these sets be equal rather than contained for semi-binary stack-free?\\ \textbf{RJ:} No, reducing $(x,y)$ can introduce new cherries that aren't in $C(N)$ yet.}
\end{observation}

We now start our investigation into the order in which pairs can be reduced. We start with a lemma that \remie{implies} a cherry on two leaves~$x$ and~$y$ can be reduced either as $(x,y)$ or as $(y,x)$. Then we show that reducing an arbitrary pair in a CPN gives a new CPN.

%swap a cherry
\begin{lemma}\label{lem:reverseCherry}
Let $S$ be a minimal CPS for a non-binary CPN~$N$ and suppose $S_i=(x,y)$ reduces a cherry when applying the sequence. \remie{Let $z$ and $w$ be distinct leaves (not necessarily different from $x$ and $y$) that have a common parent, equal to the parent of $x$ and $y$. Let $S'$ be the sequence $S_{[i+1:]}$ where each occurrence of $z$ is replaced by $x$.
Then $S_{[:i-1]}(z,w)S'$ is a minimal CPS for $N$.}
\end{lemma}
\begin{proof}
\remie{Because $(x,y)$ forms a cherry in $N'=NS_{[:i-1]}$, and $x$ and $y$ share their parents with $z$ and $w$, the reduced network $N'(x,y)$ is equal to the network $N'(z,w)$ when $z$ is replaced by $x$. Hence, if we switch the roles of $x$ and $z$ in the remaining part of the sequence, the result after reduction by both sequences is the same modulo the $x\leftrightarrow z$ replacement.}
\end{proof}

%Move forward one by one
\begin{lemma}\label{lem:MoveForward}
Let $N$ be a non-binary CPN that can be reduced by a minimal CPS $S=S_1,S_2,\ldots,S_{|S|}$ such that $S_2\in\mathcal{C}(N)$. Then $NS_2$ is a CPN. 
\end{lemma}
\begin{proof}
Note that $S_1,S_2\in\mathcal{C}(N)$ by assumption. We distinguish several cases and prove in every case that $NS_2$ is a CPN.
\begin{itemize}
    \item \textbf{The leaves in $\bm{S_1}$ and $\bm{S_2}$ are the same.} Then either $S_1=S_2$, or $S_1=(x,y)$ and $S_2=(y,x)$ for some pair of leaves $x,y$. In the first case $NS_2=NS_1$, which is a CPN. 
    In the second case, as $(x,y)$ and $(y,x)$ are both present in $N$, $(x,y)$ must be a cherry in~$N$.
    \yuki{This means that~$NS_1S_2 = NS_1$, and thus~$S$ is not a minimal CPS for~$N$.
    This case is not possible.}
    % , and, by Lemma~\ref{lem:reverseCherry}, $NS_2$ is a CPN.
\end{itemize}
Let $S_1:=(x,y)$.
\begin{itemize}
%    \item \textbf{$\bm{S_2}$ is not equal to either $\bm{(x,y)}$ or $\bm{(y,x)}$.} Then $S_1\in\mathcal{C}(NS_2)$ and $S_2\in\mathcal{C}(NS_1)$ by Observation~\ref{obs:ChangingReducibleCherries}.
%    \begin{itemize}
        \item \textbf{The pairs $\bm{S_1}$ and $\bm{S_2}$ have exactly one leaf in common.} 
        \begin{itemize}
            \item $\bm{S_2 = (x,z).}$ The common leaf~$x$ is below the reticulation common to the two reticulated cherries.
            Applying $S_1$ and $S_2$ in any order removes these two reticulation edges, so clearly $NS_1S_2=NS_2S_1$. 
            By Observation~\ref{obs:CPNAfterReducingfirstielementsremainCPN},~$NS_1S_2$ is a CPN.
            This implies $NS_2S_1$ is a CPN and, therefore, that $NS_2$ is also a CPN.
            \item $\bm{S_2 = (z,x).}$
            Observe first that~$(x,y)$ cannot form a reticulated cherry, as otherwise the first coordinate of every reducible pair that involves~$x$ is~$x$, which contradicts our assumption that~$S_2 = (z,x)\in\mathcal{C}(N)$.
            Therefore~$(x,y)$ must be a cherry.
            Then the network~$NS_1 = N(x,y)$ does not have the leaf~$x$, which implies that~$S_2 = (z,x)$ is not a reducible pair of~$NS_1$. This contradicts the fact that~$S$ was a minimal CPS for~$N$, and therefore this case is not possible.
            \item $\bm{S_2 = (y,z).}$
            The two possibilities for this case are either that~$x,y,z$ all share the same parent, or that~$(x,y)$ form a reticulated cherry in~$N$ and~$z$ shares a common parent with~$y$.
            In the former case,~$NS_1S_2$ is the CPN obtained by deleting the leaves~$x$ and~$y$ and suppressing all degree-$2$ vertices.
            We obtain the same CPN by picking the cherries~$S_2 = (y,z)$ and~$(x,z)$ in succession, that is,~$NS_2(x,z) = NS_1S_2$.
            This implies that~$NS_2$ is also a CPN.
            A similar argument can be done for the reticulated cherry case---it is easy to see that~$NS_2(x,z) = NS_1S_2$.
            \item $\bm{S_2 = (z,y).}$
            This is the case where either~$y$ and~$z$ share a common parent, or~$(z,y)$ forms a reticulated cherry.
            In both of these cases, the leaf~$x$ could share a common parent with~$y$, or~$(x,y)$ could be a reticulated cherry (there are in total~$4$ possible cases).
            In all cases, reducing~$N$ by~$S_1$ first or by~$S_2$ first has no real difference, and so~$NS_1S_2 = NS_2S_1$.
            For the same reason as before,~$NS_2$ is a CPN.
        \end{itemize}
        
        % This implies the leaf in common must be the first element of both $S_1$ and $S_2$, and it is below the reticulation common to the two reticulated cherries. 
        % Otherwise we would have that~$S_2\notin\mathcal{C}(N)$, which contradicts our original assumption.
        
        \item \textbf{The pairs $\bm{S_1}$ and $\bm{S_2}$ have no leaf in common.} Then obviously, $S_1$ and $S_2$ independently remove edges in $N$, not influenced by the order of $S_1$ and $S_2$. Hence we get $NS_1S_2=NS_2S_1$ and for the same reason as before, $NS_2$ is a CPN.
%    \end{itemize}
\end{itemize}
In all cases, we have concluded that $NS_2$ is a CPN, so the result follows.
\end{proof}

\begin{lemma}\label{lem:ReduciblePairMayAppear}
\remie{Let $N$ be a non-binary network, and $(x,y)\in \mathcal{C}(N)$. Then, there exists a minimal CPS $S$ for $N$ such that $S_i=(x,y)$ or $S_i=(y,x)$ for some $i$, and $(x,y)$ is reducible until that point, i.e., $(x,y)\in\mathcal{C}(NS_{[:j]})$ for all $j<i$.}
\end{lemma}
\begin{proof}
\remie{Let $S$ be a minimal CPS for $N$. If $S$ contains $(x,y)$ or $(y,x)$ as $S_i$, and $(x,y)$ is a reducible pair in $NS_{[:j]}$ for all $j<i$, we are done, so assume that this is not the case.
Let $i>0$ be minimal such that $(x,y)\not\in\mathcal{C}(NS_{[:i]})$. 
Then $S_i=(x,z)$ or $S_i=(y,z)$ for some $z\neq x,y$ by Observation~\ref{obs:ChangingReducibleCherries}. Because $(x,y)\in\mathcal{C}(NS_{[:i-1]})$, $(x,y)\not\in\mathcal{C}(NS_{[:i]})$, and the second element of $S_i$ is $z$, $(y,z)$ must form a cherry in $NS_{[:i-1]}$.
}

%Ret cherry
First, suppose that~$(x,y)$ forms a reticulated cherry in $NS_{[:i-1]}$, and that~$S_i=(x,z)$. In that case, $NS_{[:i]}=NS_{[:i-1]}(x,y)$, so replacing $S_{i}$ with $(x,y)$ in $S$ gives a new minimal CPS for $N$ that contains $(x,y)$.
\yuki{Next, suppose that~$(x,y)$ forms a reticulated cherry in $NS_{[:i-1]}$, and that~$S_i=(y,z)$. Then, upon switching the roles of~$y$ and~$z$, we have~$NS_{[:i]}=NS_{[:i-1]}(z,y)$. 
Letting~$S'$ denote the sequence~$S_{[i:]}$ where each occurrence of~$z$ is replaced by~$y$, we obtain a minimal CPS~$S^{new} = S_{[:i-1]}(z,y)S'$ for~$N$.
In this sequence, we have that the minimal value~$k>0$ for which~$(x,y)\notin \mathcal{C}(N^{new}_{[:k]})$ satisfies~$k>i$.
We may repeat this until we enter the first case; such a process must terminate as the length of~$S$ is finite.
}\remie{
%Cherry
On the other hand if $(x,y)$ forms a cherry in $NS_{[:i-1]}$, then $x$, $y$ and $z$ share a common parent. Therefore, by Lemma~\ref{lem:reverseCherry}, there is a minimal CPS for $N$ that starts with $S_{[:i-1]}(x,y)$.}
\end{proof}

\begin{proposition}\label{prop:CPSorder}
Let $N$ be a \remie{non-binary} CPN with $c\in\mathcal{C}(N)$. Then $Nc$ is a CPN.
That is, there exists a CPS $S$ such that $cS$ is a CPS reducing $N$.
\end{proposition}
\begin{proof}
\remie{
Let $c=(x,y)$.
By Lemma~\ref{lem:ReduciblePairMayAppear}, there is a CPS $S'$ for $N$ that contains either $(x,y)$ or $(y,x)$, and $(x,y)$ is reducible untill that point in the sequence. If $S'_1=(x,y)$, %\todoYuki{This used to be~$S'_0$, but I've changed it to~$S'_1$. I think this is the only place where we start at 0.} 
then set $S:=S'_{[2:]}$ and we are done. Now suppose $S'_1$ is not equal to $(x,y)$. Note that there must be a smallest $i\geq1$ with $S'_i=(x,y)$ or $S'_i=(y,x)$.
\begin{itemize}
    \item \textbf{Suppose $\bm{S'_{i}=(x,y)}$.} Recall that we have $(x,y)\in\mathcal{C}(NS'_{[:j]})$ for all $j< i$. 
    Hence, by applying Lemma~\ref{lem:MoveForward} $i$~times, $N(x,y)$ is a CPN.
    \item \textbf{Suppose $\bm{S'_{i}=(y,x)}$.} Again, we have $(x,y)\in\mathcal{C}(NS'_{[:j]})$ for all $j<i$. Hence, $NS'_{[:i-1]}$ has both reducible pairs $(x,y)$ and $(y,x)$, and it must contain the cherry $(x,y)$. By Lemma~\ref{lem:reverseCherry} there is a CPS of~$N$ starting with $S'_{[:i-1]}(x,y)$. Redefining $S'$ as this sequence, we are in the previous case and thus $N(x,y)$ is a CPN. 
\end{itemize}
}
We conclude that $Nc$ is a CPN.
\end{proof}

The following theorem is a corollary of the previous proposition. It essentially states that a network can be cherry picked in any order.

% \todoLeo{Theorem 1 does not say ANYTHING about the order in which cherries can be picked. The definition of partial CPS is that S is a partial CPS if there exists a CPS starting with S. So if S is minimal, it is clear that it can be extended to a minimal CPS. However, this is probably not what you want to say! There are two problems here:\\
% 1 the definition of partial CPS already says that you can start with S\\
% 2 the restriction that each element reduces a cherry or reticulated cherry does not mean that this reduction changes the network, so also not that the full CPS is minimal.}
% \todoRemie{1 Yes, you can start some CPS with $S$, but not necessarily in a minimal CPS for $N$; this is what we prove. \\ 2 Changed so that the network has to change each step}
\begin{theorem}\label{thm:OrderDoesn'tMatter}
Let $N$ be a \remie{non-binary} CPN, and $S$ a partial CPS. If in each step of the reduction of $N$ by $S$, the network is changed, then there exists a minimal CPS $S'$ starting with $S$ that reduces $N$.
\end{theorem}

\subsection{Distinguishability}\label{subsec:Distinguishability}

% As the sets of CPSs for all CPNs of the two categories (semi-binary and binary) are disjoint, there is a clear notion of distinguishability for these classes of CPNs.\todoYuki{they are not disjoint.} 

By Theorem~\ref{thm:OrderDoesn'tMatter}, any order of picking reducible pairs gives a minimal CPS for a CPN.
This inherently implies that for a given CPN, there could be many CPSs that reduces it.
However, given a class $(A,B)$, every CPS uniquely constructs a CPN in that class by Lemma~\ref{lem:SBSFBinaryrefinement}.

\begin{remark}\label{rem:Distinguishability}
Within a CPN class, exactly one CPN can be constructed for each CPS.
On the other hand, a CPN can have more than one minimal CPS that reduces it.
\end{remark}

While this remark holds true for all eight of the CPN classes, only the classes that are reconstructible are interesting to examine.
The aim of this subsection is to set up some distinguishability notion of CPNs using their minimal CPSs.
That is, we would like to encode each CPN by delegating one of its minimal CPSs to be its representative, such that the sequence can be used to reconstruct the CPN.
Since there could be more than one CPN that can be reduced by the same minimal CPS within CPN classes that are not reconstructible, it makes no sense to consider these classes.
Therefore, we define a distinguishability notion only for the classes that are reconstructible.
% To reach any distinguishability notion of the CPNs using their CPSs, we would like a standardized way of comparing the CPSs of one network to that of another.
% From a practical standpoint, this seems infeasible: by Remark~\ref{rem:Distinguishability}, we know that many distinct CPSs can be obtained from the same network.

Within a reconstructible CPN class, each network can have many minimal CPSs that reduce it by Remark~\ref{rem:Distinguishability}.
To choose a representative from these minimal CPSs, we introduce an ordering on the CPSs.
Doing so allows us to prescribe a unique \emph{smallest} CPS to each CPN.
% This chooses a `representative' CPS for every CPN, and we show that comparing these CPSs gives a way of distinguishing two CPNs.
So let us take an arbitrary ordering on the leaves, and let us define a lexicographical ordering on the reducible pairs as follows. 
We say that~$(a,b) < (c,d)$ if and only if~$a<c$ or if~$a=c$ and~$b<d$.
We naturally extend this ordering to minimal CPSs.
Let~$S$ and~$S'$ be CPSs such that~$|S|\neq|S'|$.
If $|S|<|S'|$, then $S<S'$---this ensures the smallest CPS is minimal.
% and if $|S|>|S'|$ then $S>S'$. 
Now suppose $|S|=|S'|$ and let $i$ be the smallest index such that $S_i\neq S'_i$. If no such $i$ exists, then $S=S'$; otherwise, $S<S'$ if and only if $S_i<S'_i$.

By Theorem~\ref{thm:OrderDoesn'tMatter}, we may pick a CPN in any order.
We define a \emph{smallest} CPS as one that is obtained by picking the smallest reducible pair at each iteration (see Figure~\ref{fig:SmallestCPS}).
Such a sequence is naturally a minimal CPS.
By the following theorem, distinguishing two CPNs of the same reconstructible class comes down to finding their smallest CPS and checking whether these are the same.

\begin{figure}
    \centering
    \begin{tikzpicture}[every node/.style = {draw, circle, fill, inner sep = 0pt, minimum size = 2mm},
    square/.style = {regular polygon, regular polygon sides = 4, minimum size = 3 mm}]
    \tikzset {edge/.style = {very thick, shorten >= -0.5 pt}}

        %Nodes

        \node[] (-1) at (0.0,-0.0) {};
        \node[] (0) at (0.0,-0.5) {};
        \node[] (6) at (-0.5,-1.0) {};
        \node[] (8) at (0.5,-1.0) {};
        \node[] (7) at (-1.125,-2) {};
        \node[] (9) at (0.25,-1.25) {};
        \node[square] (10) at (0,-1.5) {};

        \node[draw=none, fill=none, left = 5mm of 0] {\large{$N$}};
        %Leaves

        \node[] (1) at (-1.5,-2.5) {};
        \node[draw=none, fill=none, below=1mm of 1] (leaf_1) {\large $1$};
        \node[] (2) at (-0.75,-2.5) {};
        \node[draw=none, fill=none, below=1mm of 2] (leaf_2) {\large $2$};
        \node[] (3) at (0.0,-2.5) {};
        \node[draw=none, fill=none, below=1mm of 3] (leaf_3) {\large $3$};
        \node[] (4) at (0.75,-2.5) {};
        \node[draw=none, fill=none, below=1mm of 4] (leaf_4) {\large $4$};
        \node[] (5) at (1.5,-2.5) {};
        \node[draw=none, fill=none, below=1mm of 5] (leaf_5) {\large $5$};

        %Edges

        \draw[edge] (-1) edge (0);
        \draw[edge] (0) edge (6);
        \draw[edge] (0) edge (8);
        \draw[edge] (6) edge (7);
        \draw[edge] (6) edge (10);
        \draw[edge] (7) edge (1);
        \draw[edge] (7) edge (2);
        \draw[edge] (10) edge (3);
        \draw[edge] (8) edge (9);
        \draw[edge] (8) edge (5);
        \draw[edge] (9) edge (4);
        \draw[edge] (9) edge (10);
    \end{tikzpicture}

    \caption{The smallest CPS for this network is~$(1,2),(3,2),(3,4),(4,5),(2,5)$. Initially we have the choice of picking either~$(1,2), (2,1),$ or~$(3,4)$. For the smallest CPS we pick the smallest reducible pair~$(1,2)$.}
    \label{fig:SmallestCPS}
\end{figure}
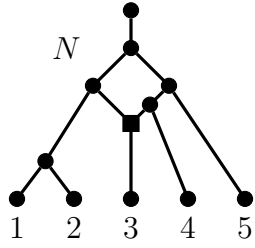

\begin{theorem}\label{thm:UniqueCPNConstruction}
Suppose we are given an ordering on the taxa set~$X$.
Every CPN on~$X$ has a unique smallest CPS.
In particular within a reconstructible CPN class, these CPSs can be used to reconstruct the CPN.
Every CPS can be used to construct a unique CPN within each of the eight CPN classes.
\end{theorem}
\begin{proof}
 Let~$N$ be a CPN on~$X$.
 Since we have a total ordering on the cherries of~$N$, we have that if there exists a smallest CPS then it is unique.
 Furthermore, we know that a smallest CPS exists: simply pick a smallest cherry at every iteration.
 Therefore every CPN on~$X$ has a unique smallest CPS.
 Within reconstructible CPN classes, no two networks have the same minimal CPSs.
 It then follows that a smallest CPS for a network can be used to construct said network.
 
 By Remark~\ref{rem:Distinguishability}, we have that every CPS gives rise to a unique CPN.
\end{proof}

The following corollary is a direct consequence of Theorem~\ref{thm:UniqueCPNConstruction}.

\begin{corollary}\label{cor:CPNIsomorphism}
Suppose we are given an ordering on the taxa set~$X$.
Within a \yuki{reconstructible} CPN class, two CPNs on~$X$ are isomorphic if and only if they have the same smallest CPS.
\end{corollary}

This leads to a polynomial-time algorithm for checking whether two CPNs of a reconstructible CPN class on the same set of taxa are isomorphic, which we describe in Section~\ref{sec:Computation}.

\section{Reduction and containment}\label{sec:reductioncontainment}
In this section, we prove that, within reconstructible CPN classes, the reduction of a network by a CPS for another network implies `containment' of the former in the latter. 
We also show that the converse does not always hold: containment of a network $N'$ in another network $N$ does not imply that there exists a minimal CPS of $N$ that reduces $N'$.
% We prove this in the following lemma using induction on the length of the sequence in reverse direction (i.e., from long to short). 

\subsection{Reduction implies containment}\label{subsec:reductionimpliescontainment}

We first formally define what it means for a network to contain another network.

% \begin{definition}
% Let~$N$ be a binary network on the set of taxa~$X$.
% A binary network~$N'$ on~$X$ is a \emph{subnetwork} of~$N$ if 
% %some subgraph~$M$ of~$N$ is a subdivision of~$N'$.
% $N'$ can be obtained from~$N$ by deleting reticulation edges and suppressing all degree-$2$ nodes in the resulting subgraph.
% \todoRemie{Now the rest of the definition does not work anymore \textbf{ym:} how bout now.}
% The \emph{embedding} of~$N'$ in~$N$ maps the nodes of~$N'$ to a subset of the nodes of~$N$, and it maps the edges of~$N'$ to edge disjoint paths of~$N$.%, such that all edges in~$M$ are covered.
% If~$N'$ is a subnetwork of~$N$, then we alternatively say that~$N$ \emph{contains} (or \emph{contains})~$N'$.
% A semi-binary stack-free network~$N'$ is a \emph{subnetwork} of another semi-binary stack-free network~$N$ if some binary refinement~$N'^b$ of~$N'$ is a subnetwork of some binary refinement~$N^b$ of~$N$.
% \end{definition}

\begin{definition}
\remie{Let~$N$ be a non-binary network on the set of taxa~$X$.
A non-binary network~$N'$ on~$X'\subseteq X$ is a \emph{subnetwork} of~$N$ if 
%some subgraph~$M$ of~$N$ is a subdivision of~$N'$.
$N'$ can be obtained from~$N$ by deleting reticulation edges, and then \emph{cleaning up w.r.t. $X'$}, i.e., applying the following changes until a network on $X'$ is obtained:
\begin{itemize}
    \item removing outdegree-0 nodes not labelled by $X'$, together with their incoming edges;
    \item suppressing all degree-$2$ nodes.
%    \item \todo[inline]{the other paper also has "remove parallel edges". I don't think we need it.}
\end{itemize}
Equivalently, $N'$ is a subnetwork of $N$ if there is an embedding of $N'$ in $N$: an injective map of the nodes of~$N'$ to a subset of the nodes of~$N$, and of the edges of~$N'$ to edge disjoint paths of~$N$, such that the mapping of the edges respects the mapping of the nodes.
A non-binary network~$N$ \emph{contains} another non-binary network~$N'$ if some refinement~$N'_b$ of~$N'$ is a subnetwork of~$N$, \remie{i.e., if $N'$ can be obtained from a subnetwork of $N$ by contracting edges.}}
\end{definition}

As seen above, a subnetwork of a network can be defined by deleting reticulation edges and cleaning up, but also with embeddings.
The equivalence of these two definitions has been shown for when the two networks are binary and on the same leaf-sets (Lemma~$1$ of \cite{murakami2019reconstructing}).
It is easy to extend this equivalence to non-binary networks on different leaf-sets (in which one leaf-set is a subset of the other), so we do not include this here. %\todoRemie{TODO: Seems like we should include it, it has changed quite a bit, and it also includes the other direction (very rough proof though)}
% Also note that the order in which we clean up a graph does not matter.

%%%%%%%%%%%%%%%%%%%%%%%%%%%%%%%%%%%%%%%%%%%
%%%%%%%%%%%%%%%%%%%%%%%%%%%%%%%%%%%%%%%%%%%
%%%%%%%%%%%%%%%%%%%%%%%%%%%%%%%%%%%%%%%%%%%
%%%%%%%%%%%%%%%%%%%%%%%%%%%%%%%%%%%%%%%%%%%
%%%%%%%%%%%%%%%%%%%%%%%%%%%%%%%%%%%%%%%%%%%
%%%%%%%%%%%%%%%%%%%%%%%%%%%%%%%%%%%%%%%%%%%
%%%%%%%%%%%%%%%%%%%%%%%%%%%%%%%%%%%%%%%%%%%
%%%%%%%%%%%%%%%%%%%%%%%%%%%%%%%%%%%%%%%%%%%

\begin{lemma}\label{obs:cleanup}
\remiee{Let $N$ and $N'$ be non-binary networks on $X$ and $X'\subseteq X$. 
$N'$ can be embedded into~$N$ if and only if $N'$ can be obtained from~$N$ by deleting a set of reticulation edges and then cleaning up w.r.t. $X'$.}
\end{lemma}
\begin{proof}
First suppose there is an embedding of~$N'$ into $N$. \yuki{Note that in this embedding, we may assume that the root of~$N'$ is mapped to the root of~$N$.} The image of this map (a subgraph of~$N$) is a subdivision of~$N'$. Let $R$ be the set of reticulation edges of $N$ not used in the embedding. We will show that $N'$ can be obtained from $N$ by removing $R$ and cleaning up w.r.t $X'$. To show this, we prove that the edges that are removed in the clean-up are exactly the edges not used by the embedding of $N'$ into $N$. 

Let $M$ be the network obtained from $N$ by removing $R$ and cleaning up w.r.t. $X'$. First note that no edges used by the embedding are removed in the process of removing $R$ and cleaning up: indeed, for each such edge, there is a path to a leaf of $X'$ using only edges of the embedding, which cannot be removed by cleaning up. Now suppose $M$ has an edge that is not used in the embedding of~$N'$ into~$N$. Consider a lowest such edge~$xy$.

Node~$y$ cannot be a leaf of~$N$, because all leaves of~$N$ are in the embedding of~$N'$ into~$N$; or they are removed in the clean-up because they are not in $X'$, in which case they cannot be part of $M$.

Now suppose that~$y$ is a tree node of~$N$. It is impossible for an outgoing edge of~$y$ to be in the embedding, because the root of the embedding is the root of~$N$.
Hence, the outgoing edges of~$y$ are not in the embedding. At least one of these outgoing edges of~$y$ is in~$M$, because, otherwise,~$y$ would have been deleted by cleaning up outdegree-0 nodes. Hence, at least one outgoing edge of~$y$ is in~$M$ but not in the embedding of~$N'$ into~$N$, contradicting the assumption that~$xy$ is a lowest such edge.

Lastly, suppose that~$y$ is a reticulation. If none of the other incoming edges of the reticulation are in the embedding, it follows, similarly to the previous case, that the outgoing edge of~$y$ is in~$M$ but not in the embedding. This contradicts the assumption that~$xy$ is a lowest such edge. Hence, at least one incoming edge of~$y$ is used by the embedding. This implies $xy$ is an element of $R$, and has been deleted, contradicting the assumption that~$xy$ is an edge of~$M$.

For the other direction, suppose $N'$ can be obtained from $N$ by removing a set of reticulation edges $R$ and cleaning up. By reversing the operations used to clean up, we get an embedding of $N'$ into $N$. Indeed, this only introduces reticulation edges not used by the embedding, and subdivides edges. When subdividing an edge used by the embedding, adapt the embedding accordingly, by mapping the edge of $N'$ to the resulting path.
\end{proof}

%%%%%%%%%%%%%%%%%%%%%%%%%%%%%%%%%%%%%%%%%%%
%%%%%%%%%%%%%%%%%%%%%%%%%%%%%%%%%%%%%%%%%%%
%%%%%%%%%%%%%%%%%%%%%%%%%%%%%%%%%%%%%%%%%%%
%%%%%%%%%%%%%%%%%%%%%%%%%%%%%%%%%%%%%%%%%%%
%%%%%%%%%%%%%%%%%%%%%%%%%%%%%%%%%%%%%%%%%%%
%%%%%%%%%%%%%%%%%%%%%%%%%%%%%%%%%%%%%%%%%%%
%%%%%%%%%%%%%%%%%%%%%%%%%%%%%%%%%%%%%%%%%%%

In what follows, in settings where we consider whether~$N'$ is a subnetwork of/contained in~$N$, we informally refer to~$N$ as the larger network and~$N'$ as the smaller network.
Note that the notions of a network being a subnetwork and a network being contained are not always synonymous.
If~$N'$ is a subnetwork of~$N$, then~$N'$ is contained in~$N$.
However, if~$N'$ is contained in~$N$, then it does not immediately follow that~$N'$ is a subnetwork of~$N$.
They are synonymous when the smaller network (i.e.,~$N'$) is binary.
The following lemma shows that for any non-binary network (not necessarily a CPN), the network obtained by reducing a pair is a subnetwork of the original network.

\begin{lemma}\label{lem:SubnetworkAfterReduction}
\remie{Let $N$ be a non-binary network, and $c$ a pair of leaves. Then $Nc$ is a subnetwork of $N$.}
\end{lemma}
\begin{proof}
\remie{For the embedding of $Nc$ into $N$, note that there is a natural map of the nodes of $Nc$ to the nodes of $N$. Each edge of $N$, corresponds naturally to an edge of $Nc$, or is part of a path (of two edges) if an endpoint got suppressed. These mappings form an embedding of $Nc$ into $N$, as the mapping of the nodes is respected, and no edge of $Nc$ is part of more than one corresponding path of $N$.%So the paths are edge disjoint.
}
\end{proof}

% In this section, we focus only on the reconstructible CPN classes.
% It is easy to see that the containment results will not hold for non-reconstructible CPN classes.
% The networks in Figure~\ref{fig:CPNNonUniqueConstruction} shows two networks that are not contained in each other, which can be reduced by the same CPS.
To show the relation between subsequences and containment, we first focus on the binary CPN class, $(\ref{Const:CherryResolved}, \ref{Const:RCherryResolved})$, for which the definitions of subnetwork and containment are synonymous.

\subsubsection{Subnetworks}

Intuitively, when a CPS~$S$ for a binary network~$N$ also reduces another binary network~$N'$, the embedding of the network~$N'$ in~$N$ can be found as follows.
Reconstruct the network $N$ from $S$, and we annotate the edges used by $N'$ in the process.
Let~$S'$ denote the CPS of ordered pairs in~$S$ that is used in the reduction of~$N'$.
Let $S_i = (x,y)$ be an ordered pair that appears in $S'$.
Then in $NS_{[:i]}$, label the paths $p_yy$ and $p_yx$ as `used'.
Upon reconstructing $N$, the embedding of $N'$ into $N$ can be seen as the subnetwork of $N$ which uses all labelled edges.

% \todo{check if still use this def.}
% \begin{definition}\label{def:EmbeddedNodes}
% Let~$v$ be a node in some network~$M$. If~$M$ is a subnetwork of~$M'$ then we say that~$v$ exists in~$M'$, and refer to~$v$ in~$M'$ by~$v^{M'}$.
% \end{definition}

% \todoRemie{Make a smaller version of this figure, like in the conference version. (or just use that one?)
% YM: I think we should make a smaller version of this one, since the one in the conf version is tree-child.}
\begin{figure}
    \centering
    $S=(2,1),(3,2),(3,4),(2,1),(1,4)$
    
    \vspace{0.5cm}
    \resizebox{\columnwidth}{!}{
    \begin{tikzpicture}[every node/.style = {draw, circle, fill, inner sep = 0pt, minimum size = 2mm},
    square/.style = {regular polygon, regular polygon sides = 4, minimum size = 3 mm}]
    \tikzset {edge/.style = {very thick, shorten >= -0.5 pt}}
        %Nodes

        \node[] (-1) at (0.0,-0.0){};
        \node[] (0) at (0.0,-0.5){};
        \node[] (5) at (-0.5,-1.0){};
        \node[opacity=0.2] (8) at (1,-1.5){};
        \node[] (6) at (-1.0,-1.5){};
        \node[opacity=0.2] (7) at (0.0,-1.5){};
        \node[square] (9) at (-0.5,-2.0){};
        \node[square, opacity=0.2] (10) at (0.5,-2.0){};

        \node[draw=none, fill=none, left = 5mm of -1] {\large{$N$}};
        %Leaves

        \node[] (1) at (-1.5,-2.5){};
        \node[draw=none, fill=none, below=1mm of 1] (leaf_1){\large $1$};
        \node[] (2) at (-0.5,-2.5){};
        \node[draw=none, fill=none, below=1mm of 2] (leaf_2){\large $2$};
        \node[opacity=0.2] (3) at (0.5,-2.5){};
        \node[draw=none, fill=none, below=1mm of 3, opacity=0.2] (leaf_3){\large $3$};
        \node[] (4) at (1.5,-2.5){};
        \node[draw=none, fill=none, below=1mm of 4] (leaf_4){\large $4$};

        %Edges

        \draw[edge] (-1) edge (0);
        \draw[edge] (0) edge (5);
        \draw[edge] (0) edge (8);
        \draw[edge] (5) edge (6);
        \draw[edge] (5) edge (7);
        \draw[edge] (6) edge (1);
        \draw[edge] (6) edge (9);
        \draw[edge] (9) edge (2);
        \draw[edge] (7) edge (9);
        \draw[edge, opacity=0.2] (7) edge (10);
        \draw[edge, opacity=0.2] (10) edge (3);
        \draw[edge, opacity=0.2] (8) edge (10);
        \draw[edge] (8) edge (4);

        \begin{scope}[shift = {(4,0)}]
        %Nodes

        \node[] (-1) at (0.0,-0.0){};
        \node[] (0) at (0.0,-0.5){};
        \node[] (5) at (-0.5,-1.0){};
        \node[opacity=0.2] (8) at (1,-1.5){};
        \node[opacity=0.2] (7) at (0.0,-1.5){};
        \node[square, opacity=0.2] (10) at (0.5,-2.0){};

        \node[draw=none, fill=none, left = 5mm of -1] {\large{$NS_{[:1]}$}};
        %Leaves

        \node[] (1) at (-1.5,-2.5){};
        \node[draw=none, fill=none, below=1mm of 1] (leaf_1){\large $1$};
        \node[] (2) at (-0.5,-2.5){};
        \node[draw=none, fill=none, below=1mm of 2] (leaf_2){\large $2$};
        \node[opacity=0.2] (3) at (0.5,-2.5){};
        \node[draw=none, fill=none, below=1mm of 3, opacity=0.2] (leaf_3){\large $3$};
        \node[] (4) at (1.5,-2.5){};
        \node[draw=none, fill=none, below=1mm of 4] (leaf_4){\large $4$};

        %Edges

        \draw[edge] (-1) edge (0);
        \draw[edge] (0) edge (5);
        \draw[edge] (0) edge (8);
        \draw[edge] (5) edge (1);
        \draw[edge] (5) edge (7);
        \draw[edge] (7) edge (2);
        \draw[edge, opacity=0.2] (7) edge (10);
        \draw[edge, opacity=0.2] (10) edge (3);
        \draw[edge, opacity=0.2] (8) edge (10);
        \draw[edge] (8) edge (4);
        \end{scope}
        
        \begin{scope}[shift = {(8,0)}]
        %Nodes

        \node[] (-1) at (0.0,-0.0){};
        \node[] (0) at (0.0,-0.5){};
        \node[] (5) at (-0.5,-1.0){};
        \node[opacity=0.2] (8) at (1,-1.5){};

        \node[draw=none, fill=none, left = 5mm of -1] {\large{$NS_{[:2]}$}};
        %Leaves

        \node[] (1) at (-1.5,-2.5){};
        \node[draw=none, fill=none, below=1mm of 1] (leaf_1){\large $1$};
        \node[] (2) at (-0.5,-2.5){};
        \node[draw=none, fill=none, below=1mm of 2] (leaf_2){\large $2$};
        \node[opacity=0.2] (3) at (0.5,-2.5){};
        \node[draw=none, fill=none, below=1mm of 3, opacity=0.2] (leaf_3){\large $3$};
        \node[] (4) at (1.5,-2.5){};
        \node[draw=none, fill=none, below=1mm of 4] (leaf_4){\large $4$};

        %Edges

        \draw[edge] (-1) edge (0);
        \draw[edge] (0) edge (5);
        \draw[edge] (0) edge (8);
        \draw[edge] (5) edge (1);
        \draw[edge] (5) edge (2);
        \draw[edge, opacity=0.2] (8) edge (3);
        \draw[edge] (8) edge (4);
        \end{scope}
        
        \begin{scope}[shift = {(12,0)}]
        %Nodes

        \node[] (-1) at (0.0,-0.0){};
        \node[] (0) at (0.0,-0.5){};
        \node[] (5) at (-0.5,-1.0){};

        \node[draw=none, fill=none, left = 5mm of -1] {\large{$NS_{[:3]}$}};
        %Leaves

        \node[] (1) at (-1.5,-2.5){};
        \node[draw=none, fill=none, below=1mm of 1] (leaf_1){\large $1$};
        \node[] (2) at (-0.5,-2.5){};
        \node[draw=none, fill=none, below=1mm of 2] (leaf_2){\large $2$};
        \node[] (4) at (1.5,-2.5){};
        \node[draw=none, fill=none, below=1mm of 4] (leaf_4){\large $4$};

        %Edges

        \draw[edge] (-1) edge (0);
        \draw[edge] (0) edge (5);
        \draw[edge] (0) edge (4);
        \draw[edge] (5) edge (1);
        \draw[edge] (5) edge (2);
        \end{scope}
        
        \begin{scope}[shift = {(16,0)}]
        %Nodes

        \node[] (-1) at (0.0,-0.0){};
        \node[] (0) at (0.0,-0.5){};

        \node[draw=none, fill=none, left = 5mm of -1] {\large{$NS_{[:4]}$}};
        %Leaves

        \node[] (1) at (-1.5,-2.5){};
        \node[draw=none, fill=none, below=1mm of 1] (leaf_1){\large $1$};
        \node[] (4) at (1.5,-2.5){};
        \node[draw=none, fill=none, below=1mm of 4] (leaf_4){\large $4$};

        %Edges

        \draw[edge] (-1) edge (0);
        \draw[edge] (0) edge (1);
        \draw[edge] (0) edge (4);
        \end{scope}
        
        \begin{scope}[shift = {(20,0)}]
        %Nodes

        \node[] (-1) at (0.0,-0.0){};

        \node[draw=none, fill=none, left = 5mm of -1] {\large{$NS$}};
        %Leaves

        \node[] (4) at (0,-2.5){};
        \node[draw=none, fill=none, below=1mm of 4] (leaf_4){\large $4$};

        %Edges

        \draw[edge] (-1) edge (4);
        \end{scope}
        
        \draw[edge] (1.5,-1.5) edge [<->] node[draw = none, fill = none, midway, yshift=10pt]{\large $(2,1)$} (2.5,-1.5);
        
        \draw[edge, opacity=0.2] (5.5,-1.5) edge [<->] node[draw = none, fill = none, midway, yshift=10pt, opacity=0.2]{\large $(3,2)$} (6.5,-1.5);
        
        \draw[edge, opacity=0.2] (9.5,-1.5) edge [<->] node[draw = none, fill = none, midway, yshift=10pt, opacity=0.2]{\large $(3,4)$} (10.5,-1.5);
        
        \draw[edge] (13.5,-1.5) edge [<->] node[draw = none, fill = none, midway, yshift=10pt]{\large $(2,1)$} (14.5,-1.5);
        
        \draw[edge] (17.5,-1.5) edge [<->] node[draw = none, fill = none, midway, yshift=10pt]{\large $(1,4)$} (18.5,-1.5);
        
    \end{tikzpicture}
    }
    \caption{A visualization of Lemma~\ref{lem:reductionImpliesContainment}.
    The binary CPN~$N$ from Figure~\ref{fig:Definition} (grey and black), together with one of its minimal CPSs~$S$.
    The subnetwork of~$N$ (black) is also reduced by $S$, and the embedding can be constructed by building both networks simultaneously and keeping track of the edges added by the pairs that change the subnetwork (black pairs and arrows).}
    \label{fig:SubsequenceToEmbedding}
\end{figure}
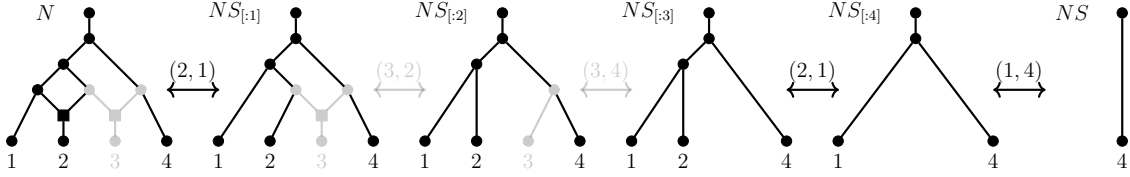

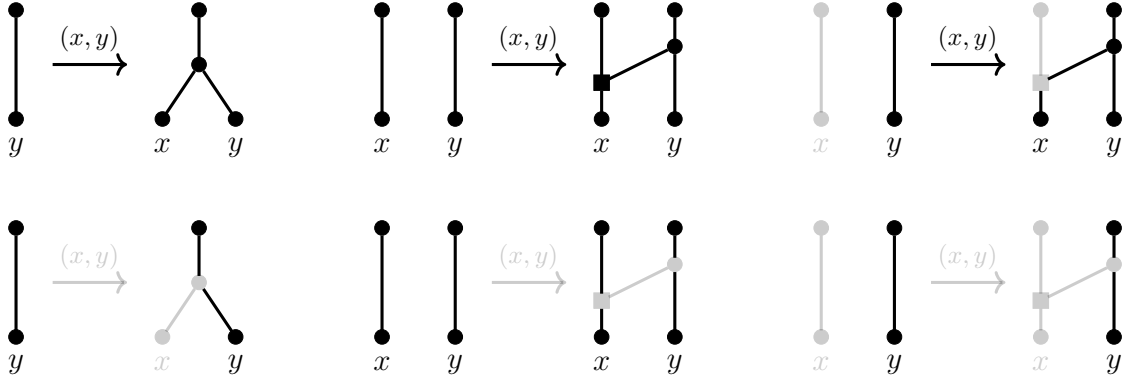
\begin{figure}
    \centering
    \resizebox{\columnwidth}{!}{
    \begin{tikzpicture}[every node/.style = {draw, circle, fill, inner sep = 0pt, minimum size = 2mm},
    square/.style = {regular polygon, regular polygon sides = 4, minimum size = 3 mm}]
    \tikzset {edge/.style = {very thick, shorten >= -0.5 pt}}
        \begin{scope}[shift = {(-6,0)}]
        %Nodes
        
        \node[] (5) at (0.5,-0.5){};

        %Leaves

        \node[] (2) at (0.5,-2.0){};
        \node[draw=none, fill=none, below=1mm of 2] (leaf_2){\large $y$};

        %Edges

        \draw[edge] (5) edge (2);
        \draw[edge] (1,-1.25) edge [->] 
        node[draw = none, fill = none, midway, yshift=10pt]{$(x,y)$}
        (2,-1.25);
        \end{scope}
        \begin{scope}[shift = {(-3,0)}]
        %Nodes
        
        \node[] (4) at (0,-0.5){};
        \node[] (5) at (0, -1.25){};

        %Leaves

        \node[] (1) at (-0.5,-2.0){};
        \node[draw=none, fill=none, below=1mm of 1] (leaf_1){\large $x$};
        \node[] (2) at (0.5,-2.0){};
        \node[draw=none, fill=none, below=1mm of 2] (leaf_2){\large $y$};

        %Edges

        \draw[edge] (4) edge (5);
        \draw[edge] (5) edge (1);
        \draw[edge] (5) edge (2);
        \end{scope}
        
        %%%%%%%%%%%%%%%%%%%%%%%%%%%%%%%%%%%%%%%%%%%%%%%%%%%%%%%%
        
        \begin{scope}[shift = {(0,0)}]
        %Nodes
        
        \node[] (4) at (-0.5,-0.5){};
        \node[] (5) at (0.5,-0.5){};

        %Leaves

        \node[] (1) at (-0.5,-2.0){};
        \node[draw=none, fill=none, below=1mm of 1] (leaf_1){\large $x$};
        \node[] (2) at (0.5,-2.0){};
        \node[draw=none, fill=none, below=1mm of 2] (leaf_2){\large $y$};

        %Edges

        \draw[edge] (4) edge (1);
        \draw[edge] (5) edge (2);
        
        \draw[edge] (1,-1.25) edge [->] 
        node[draw = none, fill = none, midway, yshift=10pt]{$(x,y)$}
        (2,-1.25);
        \end{scope}
        \begin{scope}[shift = {(3,0)}]
        %Nodes
        
        \node[] (4) at (-0.5,-0.5){};
        \node[] (5) at (0.5,-0.5){};
        \node[] (6) at (0.5,-1.0){};
        \node[square] (3) at (-0.5,-1.5){};

        %Leaves

        \node[] (1) at (-0.5,-2.0){};
        \node[draw=none, fill=none, below=1mm of 1] (leaf_1){\large $x$};
        \node[] (2) at (0.5,-2.0){};
        \node[draw=none, fill=none, below=1mm of 2] (leaf_2){\large $y$};

        %Edges

        \draw[edge] (4) edge (3);
        \draw[edge] (3) edge (1);
        \draw[edge] (5) edge (6);
        \draw[edge] (6) edge (3);
        \draw[edge] (6) edge (2);
        \end{scope}

        %%%%%%%%%%%%%%%%%%%%%%%%%%%%%%%%%%%%%%%%%%%%%%%%%%%%%%%%%%%%
        
        \begin{scope}[shift = {(6,0)}]
        %Nodes
        
        \node[opacity=.2] (4) at (-0.5,-0.5){};
        \node[] (5) at (0.5,-0.5){};

        %Leaves

        \node[opacity=.2] (1) at (-0.5,-2.0){};
        \node[draw=none, fill=none, below=1mm of 1, opacity=.2] (leaf_1){\large $x$};
        \node[] (2) at (0.5,-2.0){};
        \node[draw=none, fill=none, below=1mm of 2] (leaf_2){\large $y$};

        %Edges

        \draw[edge, opacity=.2] (4) edge (1);
        \draw[edge] (5) edge (2);
        \draw[edge] (1,-1.25) edge [->] 
        node[draw = none, fill = none, midway, yshift=10pt]{$(x,y)$}
        (2,-1.25);
        \end{scope}
        \begin{scope}[shift = {(9,0)}]
        %Nodes
        
        \node[opacity=.2] (4) at (-0.5,-0.5){};
        \node[] (5) at (0.5,-0.5){};
        \node[] (6) at (0.5,-1.0){};
        \node[square, opacity=.2] (3) at (-0.5,-1.5){};

        %Leaves

        \node[] (1) at (-0.5,-2.0){};
        \node[draw=none, fill=none, below=1mm of 1] (leaf_1){\large $x$};
        \node[] (2) at (0.5,-2.0){};
        \node[draw=none, fill=none, below=1mm of 2] (leaf_2){\large $y$};

        %Edges

        \draw[edge, opacity=.2] (4) edge (3);
        \draw[edge] (3) edge (1);
        \draw[edge] (5) edge (6);
        \draw[edge] (6) edge (3);
        \draw[edge] (6) edge (2);
        \end{scope}
        
        %%%%%%%%%%%%%%%%%%%%%%%%%%%%%%%%%%%%%%%%%%%%%%%%%%%
        %%%%%%%%%%%%%%%%%%%%%%%%%%%%%%%%%%%%%%%%%%%%%%%%%%%
        
        \begin{scope}[shift = {(-6,-3)}]
        %Nodes
        
        \node[] (5) at (0.5,-0.5){};

        %Leaves

        \node[] (2) at (0.5,-2.0){};
        \node[draw=none, fill=none, below=1mm of 2] (leaf_2){\large $y$};

        %Edges

        \draw[edge] (5) edge (2);
        \draw[edge, opacity=.2] (1,-1.25) edge [->] 
        node[draw = none, fill = none, midway, yshift=10pt, opacity=.2]{$(x,y)$}
        (2,-1.25);
        \end{scope}
        \begin{scope}[shift = {(-3,-3)}]
        %Nodes
        
        \node[] (4) at (0,-0.5){};
        \node[opacity=.2] (5) at (0, -1.25){};

        %Leaves

        \node[opacity=.2] (1) at (-0.5,-2.0){};
        \node[draw=none, fill=none, below=1mm of 1, opacity=.2] (leaf_1){\large $x$};
        \node[] (2) at (0.5,-2.0){};
        \node[draw=none, fill=none, below=1mm of 2] (leaf_2){\large $y$};

        %Edges

        \draw[edge] (4) edge (5);
        \draw[edge, opacity=.2] (5) edge (1);
        \draw[edge] (5) edge (2);
        \end{scope}
        
        %%%%%%%%%%%%%%%%%%%%%%%%%%%%%%%%%%%%%%%%%%%%%%%%%%%%%%%%
        
        \begin{scope}[shift = {(0,-3)}]
        %Nodes
        
        \node[] (4) at (-0.5,-0.5){};
        \node[] (5) at (0.5,-0.5){};

        %Leaves

        \node[] (1) at (-0.5,-2.0){};
        \node[draw=none, fill=none, below=1mm of 1] (leaf_1){\large $x$};
        \node[] (2) at (0.5,-2.0){};
        \node[draw=none, fill=none, below=1mm of 2] (leaf_2){\large $y$};

        %Edges

        \draw[edge] (4) edge (1);
        \draw[edge] (5) edge (2);
        
        \draw[edge, opacity=.2] (1,-1.25) edge [->] 
        node[draw = none, fill = none, midway, yshift=10pt, opacity=.2]{$(x,y)$}
        (2,-1.25);
        \end{scope}
        \begin{scope}[shift = {(3,-3)}]
        %Nodes
        
        \node[] (4) at (-0.5,-0.5){};
        \node[] (5) at (0.5,-0.5){};
        \node[opacity=.2] (6) at (0.5,-1.0){};
        \node[square, opacity=.2] (3) at (-0.5,-1.5){};

        %Leaves

        \node[] (1) at (-0.5,-2.0){};
        \node[draw=none, fill=none, below=1mm of 1] (leaf_1){\large $x$};
        \node[] (2) at (0.5,-2.0){};
        \node[draw=none, fill=none, below=1mm of 2] (leaf_2){\large $y$};

        %Edges

        \draw[edge] (4) edge (3);
        \draw[edge] (3) edge (1);
        \draw[edge] (5) edge (6);
        \draw[edge, opacity=.2] (6) edge (3);
        \draw[edge] (6) edge (2);
        \end{scope}

        %%%%%%%%%%%%%%%%%%%%%%%%%%%%%%%%%%%%%%%%%%%%%%%%%%%%%%%%%%%%
        
        \begin{scope}[shift = {(6,-3)}]
        %Nodes
        
        \node[opacity=.2] (4) at (-0.5,-0.5){};
        \node[] (5) at (0.5,-0.5){};

        %Leaves

        \node[opacity=.2] (1) at (-0.5,-2.0){};
        \node[draw=none, fill=none, below=1mm of 1, opacity=.2] (leaf_1){\large $x$};
        \node[] (2) at (0.5,-2.0){};
        \node[draw=none, fill=none, below=1mm of 2] (leaf_2){\large $y$};

        %Edges

        \draw[edge, opacity=.2] (4) edge (1);
        \draw[edge] (5) edge (2);
        \draw[edge, opacity=.2] (1,-1.25) edge [->] 
        node[draw = none, fill = none, midway, yshift=10pt, opacity=.2]{$(x,y)$}
        (2,-1.25);
        \end{scope}
        \begin{scope}[shift = {(9,-3)}]
        %Nodes
        
        \node[opacity=.2] (4) at (-0.5,-0.5){};
        \node[] (5) at (0.5,-0.5){};
        \node[opacity=.2] (6) at (0.5,-1.0){};
        \node[square, opacity=.2] (3) at (-0.5,-1.5){};

        %Leaves

        \node[opacity=.2] (1) at (-0.5,-2.0){};
        \node[draw=none, fill=none, below=1mm of 1, opacity=.2] (leaf_1){\large $x$};
        \node[] (2) at (0.5,-2.0){};
        \node[draw=none, fill=none, below=1mm of 2] (leaf_2){\large $y$};

        %Edges

        \draw[edge, opacity=.2] (4) edge (3);
        \draw[edge, opacity=.2] (3) edge (1);
        \draw[edge] (5) edge (6);
        \draw[edge, opacity=.2] (6) edge (3);
        \draw[edge] (6) edge (2);
        \end{scope}
    \end{tikzpicture}
    }
    \caption{Each of the cases of Lemma~\ref{lem:reductionImpliesContainment}. The subnetwork consists of all black edges and nodes. The network that contains it consists of the black and grey parts. A step indicated with a grey arrow and pair is one where the subnetwork was not reduced by this pair of the sequence.}
    \label{fig:EmbeddingSteps}
\end{figure}

\begin{lemma}\label{lem:ReduceInBothSubnetworks}
\remie{
Let $N$ and $N'$ be binary networks, and $c$ a reducible pair in $N$ and $N'$. If $N'c$ is a subnetwork of $Nc$, then $N'$ is a subnetwork of $N$.}
\end{lemma}
\begin{proof}
First note that there are three cases: neither $Nc$ nor $N'c$ contains $x$; only $Nc$ contains $x$; or both $Nc$ and $N'c$ contain $x$. To find an embedding~$\iota$ of~$N'$ into~$N$, we extend the embedding~$\iota_c$ of $N'c$ into $Nc$ in each of these cases. We denote the natural embeddings of $N'c$ into $N'$ and of $Nc$ into $N$ by $h'$ and $h$ (see Figure~\ref{fig:CDReducImpliesContain}).
\begin{itemize}
    \item \textbf{The leaf~$\bm{x}$ is in neither~$\bm{Nc}$ nor~$\bm{N'c}$.} Let $p$ and $p'$ be the parent of $y$ in $N$ and $N'$, and $g$ and $g'$ the grandparent of $y$ in $N$ and $N'$. An embedding $\iota$ can be constructed from $\iota_c$ by setting $\iota(e)=h(\iota_c(h'^{-1}e))$ for all edges of $N'$ other than $g'p'$, $p'y$, and $p'x$---$h'$ maps each edge to a path of length one, except for the edge incident to $y$. Then, we map the remaining edges by setting $\iota(g'p')=h(\iota(h'^{-1}(g')y))\setminus\{py\}$, $\iota(p'x)=px$, and $\iota(p'y)=py$---the node $h'^{-1}(g')$ is well defined because the embedding $h'$ is injective on the nodes, and there are exactly two more nodes in $N'$ than in $N'c$, which cannot be the image of any node as the edge incident to $x$ is not part of the embedding. This mapping gives an embedding, because the corresponding node mapping is still injective, no edge is in more than one image-path of an edge, and the endpoint relations are respected.
    \item \textbf{The leaf~$\bm{x}$ is only in~$\bm{Nc}$.} This case is almost the same as the previous case, except that $\iota(p'x)=p_yp_xx$. The only edge of $N$ that may be in the $\iota$-image of more than one edge, is $p_xx$. However, the only edge in $Nc$ that maps to $p_xx$ is the edge incident to $x$, and this edge is not in the image of $\iota_c$. Hence, as the image of each edge other than $g'p'$, $p'y$, and $p'x$ is defined by $\iota(e)=h(\iota_c(h'^{-1}e))$, no other edge of $N'$ is mapped to a path of $N$ that contains $p_xx$. Furthermore, the endpoint relations are still respected, so the map $\iota$ is an embedding of $N'$ into $N$.
    \item \textbf{The leaf~$\bm{x}$ is in both~$\bm{Nc}$ and~$\bm{N'c}$.} Let $p_x$ and $p_x'$ be the parent of $x$, $p_y$ and $p_y'$ the parent of $y$, $g_x\neq p_y$ and $g_x'\neq p_y'$ the other grandparent of $x$, and $g_y$ and $g_y'$ the grandparent of $y$ in $N$ and $N'$. An embedding $\iota$ can be constructed from $\iota_c$ by setting $\iota(e)=h(\iota_c(h'^{-1}e))$ for all edges of $N'$ other than $g_y'p_y'$, $p_y'y$, $p_y'p_x'$, $g_x'p_x'$ and $p_x'x$---$h'$ maps each edge to a path of length one, except for the edges incident to $x$ and $y$. Then, we map the remaining edges by setting $\iota(g_y'p_y')=h(\iota(h'^{-1}(g_y')y))\setminus\{py\}$, $\iota(p_y'y)=p_yy$, $\iota(p_y'p_x')=p_yp_x$, $\iota(g_x'p_x')=h(\iota(h'^{-1}(g_x')x))\setminus\{p_xx\}$, and $\iota(p_x'x)=p_xx$---the nodes $h'^{-1}(g_x')$ and $h'^{-1}(g_y')$ are well defined because the embedding $h'$ is injective on the nodes, and there are exactly two more nodes in $N'$ than in $N'c$, which cannot be the image of any node as the reticulation edge $p_y'p_x'$ is not part of the embedding. This mapping gives an embedding, because the corresponding node mapping is still injective, no edge is in more than one image-path of an edge, and the endpoint relations are respected.
\end{itemize}
\end{proof}

\begin{figure}
	\centering
	\begin{tikzcd}
        N' \arrow[r,red, "\iota"]
        & N \\
        N'c \arrow[u, "h'"] \arrow[r, "\iota_c"]
        & Nc \arrow[u, "h"]
    \end{tikzcd}
    \caption{The commutative diagram for the networks~$N, N', Nc$, and~$N'c$ as in the setting of Lemma~\ref{lem:ReduceInBothSubnetworks}.
    Given the maps~$h, h'$, and~$\iota_c$, we wish to produce the map~$\iota$.}
    \label{fig:CDReducImpliesContain}
\end{figure}

\begin{lemma}\label{lem:reductionImpliesContainment}
\remie{
Let $N$ and $N'$ be binary CPNs on taxa set~$X$ and~$X'\subseteq X$ respectively, and suppose that a CPS $S$ reduces both~$N$ and~$N'$,
such that all elements of~$S$ that reduce something in~$N'$ also reduce something in~$N$.
Then $N'$ is a subnetwork of $N$.}
\end{lemma}
\begin{proof}
\remie{
Let $S'$ denote the CPS of ordered pairs in $S$ such that every ordered pair in $S'$ is used in the reduction of $N'$.
Due to this, we have that for every~$i\in\{1,\dots,|S'|\}$ there exists a~$\pi(i)\in\{1,\dots,|S|\}$ such that~$S'_i = S_{\pi(i)}$.
By definition of~$S'$,~$\pi: \{1,\dots,|S'|\} \rightarrow \{1,\dots,|S|\}$ is a strictly increasing function (i.e., if~$i<j$ then~$\pi(i) < \pi(j)$).
Recall that~$S'_{[:i]}$ denotes the sequence obtained by taking the first~$i$ ordered pairs from~$S'$.
We show by induction on~$i$, for~$i = |S'|,\dots,0$, that the CPN~$N'S'_{[:i]}$ is a subnetwork of the CPN~$NS_{[:\pi(i)]}$, where~$N'S'_{[:0]} = N'$, and~$\pi(0) = \pi(1) - 1$.
\\
%%BASE:
For the base case, we prove the claim for~$i = |S'|$.
Let $S'_{|S'|}=(x,y)$. Then~$N'S'$ is the tree with one leaf~$y$. 
Since $(x,y)$ also reduces something in $NS_{[:\pi(|S'|)-1]}$, the network $NS_{[:\pi(|S'|)]}$ must still contain $y$.
As there is a path from the root to any leaf in a phylogenetic network, $N'S'=N'S'_{[:|S'|]}$ is a subnetwork of~$NS_{[:\pi(|S'|)]}$.
\\
%%Step:
So now assume~$0\leq i<|S'|$ and suppose we have proven that~$N'S'_{[:j]}$ is a subnetwork of~$NS_{[:\pi(j)]}$ for every~$j>i$.
As $i\geq 0$, there is an element $S'_{i+1}=S_{\pi(i+1)}$ in each sequence, and $N'S'_{[:i+1]}$ is a subnetwork of $NS_{[:\pi(i+1)]}$ by the induction hypothesis. By Lemma~\ref{lem:ReduceInBothSubnetworks}, $N'S'_{[:i]}$ is a subnetwork of $NS_{[:\pi(i+1)-1]}$ because $S'_{i+1}=S_{\pi(i+1)}$ acts on both networks. Applying Lemma~\ref{lem:SubnetworkAfterReduction} to $N'S'_{[:i]}$, $NS_{[:j]}$ and $NS_{[:j+1]}$ for all $j=\pi(i+1)-1,\ldots,\pi(i)$, we get that $N'S'_{[:i]}$ is a subnetwork of $NS_{[:j]}$ for all $j=\pi(i),\ldots,\pi(i+1)-1$. Hence, in particular, $N'S'_{[:i]}$ is a subnetwork of $NS_{[:\pi(i)]}$.
\\
Therefore,~$N'S'_{[:i]}$ is a subnetwork of~$NS_{[:\pi(i)]}$ for all~$i\in\{0,\dots,|S'|\}$.
In particular, $N' = N'S'_{[:0]}$ is a subnetwork of~$N=NS_{[:0]}$.
}
\end{proof}

Note that Lemma~\ref{lem:reductionImpliesContainment} was only shown for the class of binary CPNs ($(\ref{Const:CherryResolved},\ref{Const:RCherryResolved})$ CPN class).
This result generalizes to non-binary CPNs if we consider the notion of containment instead of subnetwork.
% easily for when the larger network is binary, and also by using the notion of containment rather than subnetwork.

\subsubsection{Containment results}

Recall that a network~$N'$ is contained in another network~$N$ if there exists a refinement of~$N'$ that is a subnetwork of~$N$.
We first show a result that follows immediately from Lemma~\ref{lem:reductionImpliesContainment} given that the larger network is binary.

\begin{theorem}\label{thm:BNBRedimpliesCon}
Let $N$ be a binary CPN, and $N'$ a non-binary CPN \remie{on $X$ and $X'\subseteq X$}, respectively. 
If a minimal CPS~$S$ for $N$ also reduces $N'$, then $N'$ is contained in $N$.
\end{theorem}
\begin{proof}
\remie{
Let $S'$ be the subsequence of $S$ consisting of the elements that change something in $N'$, and let $N'_b$ be the unique binary network corresponding to $S'$. Then, by Lemma~\ref{lem:SBSFBinaryrefinement}, $N'_b$ is a binary refinement of $N'$ with minimal CPS $S'$. As $N$ is the binary network corresponding to $S$, $N'_b$ is a subnetwork of $N$ by Lemma~\ref{lem:reductionImpliesContainment}. Hence, $N'$ is contained in $N$.
}
\end{proof}

Unfortunately for two general non-binary CPNs, an analogue of Lemma~\ref{lem:reductionImpliesContainment} is not possible to obtain.
For example, the two networks of the CPN class~$(\ref{Const:CherryResolved}, \ref{Const:RCherryStack})$ in Figure~\ref{fig:CPNNonUniqueConstruction} are not contained in one another, yet there is a common CPS that reduce both networks.
% Luckily, we may extend Lemma~\ref{lem:reductionImpliesContainment} by considering networks of the same reconstructible CPN class, and by using the notion of containment rather than subnetwork.
Therefore, we need to use the more relaxed notion of containment, rather than subnetwork.

\begin{lemma}\label{lem:SBSFdispaysIsBinaryResolvedSubnetwork}
Let~$C$ be a reconstructible class of CPNs.
Let~$N$ and~$N'$ be networks in~$C$ on taxa set~$X$ and~$X'\subseteq X$, respectively.
Then~$N'$ is \remie{contained in}~$N$ if and only if there exist binary refinements~$N'_b$ of~$N'$ and~$N_b$ of~$N$ such that~$N'_b$ is a subnetwork of~$N_b$.
%\todoYuki{The if direction is not true. Come back to this, as Lemma 16 needs it. \textbf{RJ:} Now it says contained instead of subnetwork. I think it is correct now.\textbf{YM:} This is now to be proved for all classes of reconstrcutible networks!}
\end{lemma}
%\todoRemie{I think we need $N_b$ and $N'_b$ to be CPNs; I think they are in this proof, but we should add it to the statement. YM: Why do we need this? The proof still works even if we assume that~$N$ and~$N'$ are not CPNs. We only need the fact that those two networks belong to the same class, in which all tt-edges, rr-edges, or both are contracted. I've changed the second half of the proof, but I don't think we need that the binary refinements are also CPNs.}
\begin{proof}
Suppose first that~$N'$ is contained in~$N$.
Then there exists a refinement~$N'_f$ of~$N'$ that is a subnetwork of~$N$.
There exists an embedding~$\iota$ of~$N'_f$ into~$N$.

Observe that every reticulation vertex in~$N'_f$ is mapped to a reticulation vertex in~$N$ of the same or higher degree.
Let~$r'$ be a reticulation vertex in~$N'_f$, such that~$\iota(r') = r$ for some reticulation vertex~$r$ in~$N$.
Let~$e'_1,\ldots,e'_a$ denote all incoming edges of~$r'$, and let~$e_1,\ldots, e_b$ denote all incoming edges of~$r$ where~$a\leq b$, such that the paths~$\iota(e'_i)$ contains the edge~$e_i$ for all~$i\in[a]$.
Resolve~$r'$ as a single path of reticulation vertices~$r'_1\cdots r'_{a-1}$ such that the edge~$e'_1$ is incident to~$r'_1$, and the edges~$e'_{i+1}$ are incident to~$r'_{i}$ for all~$i\in[a-1]$.
Resolve~$r$ as a single path of reticulation vertices~$r_1\cdots r_{b-1}$, such that the edge~$e_1$ is incident to~$r_1$, and the edges~$e_{i+1}$ are incident to~$r_{i}$ for all~$i\in[b-1]$.
We now show that~$N'_f$ with~$r'$ refined in this manner is a subnetwork of~$N$ with~$r$ refined in this manner.
We do this by altering the mapping~$\iota$ as follows.
The edges~$e'_i$ are still mapped to the same paths containing the edges~$e_i$ for~$i\in[a]$.
The reticulation edges~$r'_ir'_{i+1}$ are mapped to~$r_ir_{i+1}$ for~$i\in[a-2]$.
Let~$c$ denote the child of~$r'$ in~$N'_f$.
The edge~$r'_{a-1}c$ is mapped to the path from~$r_{a-1}$ to~$\iota(c)$.
All other mappings remain unchanged.

A similar observation can be made for tree vertices.
Now, observe that binary refinements of these CPNs are obtained by either resolving either all multifurcations, all multi-reticulations, or both (depending on~$C$).
Then we may apply the above to all multifurcations / multi-reticulations as needed to obtain binary refinements~$N'_b$ of~$N'$ and~$N_b$ of~$N$ such that~$N'_b$ is a subnetwork of~$N_b$.
\medskip

Now suppose that there exist binary refinements~$N'_b$ of~$N'$ and~$N_b$ of~$N$ such that~$N'_b$ is a subnetwork of~$N_b$.
We note that~$N$ can be obtained from~$N_b$ by contracting all tt-edges, all rr-edges, or both depending on if~$N$ is a CPN of class~$(\ref{Const:CherryResolved},\ref{Const:RCherrySF}), (\ref{Const:CherryUnresolved},\ref{Const:RCherryStack}),$ or~$(\ref{Const:CherryUnresolved},\ref{Const:RCherryUnresolved})$.
Now, if~$N_b = N$, then we are done.
So suppose that~$N$ is not a binary network.

The plan is to contract the edges of~$N_b$ to obtain~$N$.
In the process, we choose to either contract or not contract `corresponding' edges in~$N'_b$, ensuring at each step that the resulting network is a subnetwork of some refinement of the main network.
We start by looking at contracting the rr-edges of~$N_b$.
% Then we must contract all rr-edges in~$N_b$ to obtain~$N$.
We claim that contracting all rr-edges in~$N'_b$ that do not map to an rtr-path gives a refinement of~$N'$ that is a subnetwork of~$N$.

Let~$e$ be an edge in~$N'_b$.
If~$e$ is an rr-edge, then it is mapped to some path~$P_e$ in~$N_b$.
Since reticulation vertices are mapped to reticulation vertices in the embedding, the path~$P_e$ is an rr-path or an rtr-path.
If~$P_e$ is an rr-path, then contract all edges of~$P_e$ in~$N_b$.
Contract~$e$ in~$N'_b$.
It is easy to see that by mapping the reticulation vertex obtained by contracting~$e$ to the reticulation vertex obtained by contracting~$P_e$, we may extend the embedding of~$N'_b$ into~$N_b$ in a natural manner, thereby showing that the newly obtained network is a subnetwork of the other.
Now suppose that~$P_e$ was an rtr-path.
Then we contract all rr-edges contained in~$P_e$.
Observe that~$N'_b$ is still embedded in this newly obtained graph, since we may extend the original embedding by simply changing the mapping of the edge~$e$ to the newly contracted rtr-path.
On the other hand, suppose that~$e$ was not an rr-edge.
% Now suppose that there exists a tr-edge or an rt-edge in~$N'_b$ that is mapped to a path containing rr-edges in~$N_b$.
If it is mapped to a path containing rr-edges in~$N_b$, then we may simply contract the rr-edges of the path, and update the embedding by changing the mapping of the edge to the newly contracted path.
Finally suppose that there exists an rr-edge in~$N_b$ that is not used in the embedding.
Then contracting such an edge has no effect on the embedding.

Similarly, contracting all tt-edges in~$N'_b$ that do not map to a path containing a trt-path gives a refinement of~$N'$ that is a subnetwork of~$N$.

%\todoRemie{All contractions are done in $N_b$ and $N'_b$, but you actually need to do contractions in partly contracted versions of these networks. It seems clear the arguments still work, but the proof is not correct as it is, now. \textbf{Changed presentation.}}

We now obtain a sequence of networks~$N'_b = N'_0, N'_1, \ldots, N'_k$ and~$N_b = N_0,N_1, \dots, N_k = N$ such that~$N'_i$ is a subnetwork of~$N_i$ for all~$i\in[k]$, and~$N'_i$ is obtained from~$N'_{i-1}$ by contracting an rr-edge that does not map to an rtr-path in~$N_{i-1}$ (or by contracting a tt-edge that does not map to a trt-path in~$N_{i-1}$, depending on the CPN class) for~$i\in[k]$. 
The networks~$N_i$ are obtained by contracting the edges of the rr-path or the rtr-path (or the tt-path or the trt-path) as outlined in the previous paragraph.
Note also that~$N'_i$ is a refinement of~$N'$.
Then we have that~$N'_k$, which is a refinement of~$N'$, is a subnetwork of~$N$.
Therefore~$N'$ is contained in~$N$.
% Observe that none of these contractions done above depend on the degree of the endvertices of~$e$.
% This means that by iterating over all edges of~$N'_b$, we can contract some of the rr-edges (and/or the tt-edges, depending on the CPN class) to obtain a network~$N'_f$, whilst contracting all rr-edges (and/or tt-edges) of~$N_b$ to obtain~$N$.
% While this is taking place, we have ensured that the smaller network remains a subnetwork of the larger one by updating the embedding.
% Thus we have that~$N'_f$, which is necessarily a refinement of~$N'$, is a subnetwork of~$N$.
% Therefore~$N'$ is contained in~$N$.
\end{proof}

The complication in this lemma arises from the fact that since the smaller network may contain fewer leaves than the larger network, there could be many more rr-edges in the binary refinement of the smaller network than in that of the larger network.
This implies that there is a tree vertex in~$N_b$ with a reticulation above and a reticulation below, that is not used in the embedding of~$N'_b$ into~$N_b$.
Naively contracting all the rr-edges in~$N'_b$ leads to an issue of `over-contracting' the edges, which would sometimes make it impossible to be embedded into~$N$.

Finally, we present the lemma analogous to Lemma~\ref{lem:reductionImpliesContainment} for two CPNs within the same reconstructible class.

\begin{lemma}\label{lem:reductionImpliesContainmentSF}
Let~$C$ be a reconstructible class of CPNs.
Let $N$ and $N'$ be CPNs in~$C$ on taxa set~$X$ and~$X'\subseteq X$ respectively, and suppose that a minimal CPS $S$ for $N$ also reduces~$N'$.
%, such that all elements of~$S$ which reduce something in~$N'$ also reduce something in~$N$.
Then $N'$ is \remie{contained in }%a subnetwork of 
$N$.
\end{lemma}
\begin{proof}
\remie{
Let $S'$ be the subsequence of $S$ consisting of the elements that change something in $N'$, and let $N_b$ and $N'_b$ be the unique binary CPNs obtained from $S$ and $S'$ using the~$(\ref{Const:CherryResolved}, \ref{Const:RCherryResolved})$ construction. Then, by Lemma~\ref{lem:SBSFBinaryrefinement}, $N_b$ is a binary refinement of $N$ with minimal CPS $S$, and $N'_b$ is a binary refinement of $N'$ with minimal CPS $S'$. By Lemma~\ref{lem:reductionImpliesContainment}, $N'_b$ is a subnetwork of $N_b$.
By Lemma~\ref{lem:SBSFdispaysIsBinaryResolvedSubnetwork},~$N'$ is contained in $N$.
}
\end{proof}
% \begin{proof}
% Find binary refinements~$N^b$ of~$N$ and~$N'^b$ of~$N'$ such that the same elements of the CPS~$S$ used to reduce~$N$ and~$N'$ is used to reduce both networks~$N^b$ and~$N'^b$, respectively.
% Such refinements exist by Lemma~\ref{lem:SBSFBinaryrefinement}.
% By applying Lemma~\ref{lem:reductionImpliesContainment} to the networks~$N^b$ and~$N'^b$, and to the sequence~$S$, we see that~$N'^b$ is a subnetwork of~$N^b$.
% By definition,~$N'$ is a subnetwork of~$N$.
% \end{proof}

Note that Lemma~\ref{lem:reductionImpliesContainmentSF} does not hold if we used subnetwork instead of containment (see Figure~\ref{fig:reductionImpliesContainmentSF}.
The converse of Lemmas~\ref{lem:reductionImpliesContainment} and~\ref{lem:reductionImpliesContainmentSF} do not hold for general CPNs, and we show this in the following subsection.

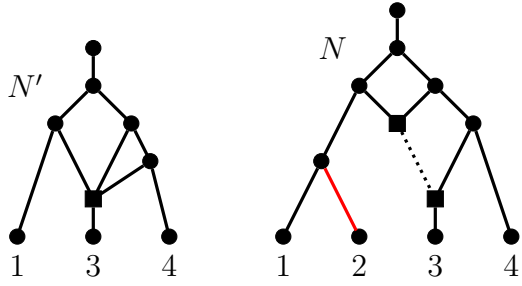
\begin{figure}
    \centering
    \begin{tikzpicture}[every node/.style = {draw, circle, fill, inner sep = 0pt, minimum size = 2mm},
square/.style = {regular polygon, regular polygon sides = 4, minimum size = 3 mm}]
    \tikzset {edge/.style = {very thick, shorten >= -0.5}}

        %Nodes
    
        \node[] (-1) at (0.0,-0.0) {};
        \node[] (0) at (0.0,-0.5) {};
        \node[] (5) at (-0.5,-1.0) {};
        \node[] (6) at (0.5,-1.0) {};
        \node[square] (7) at (0.0,-1.5) {};
        \node[] (9) at (1.0,-1.5) {};
        \node[] (8) at (-1,-2.0) {};
        \node[square] (10) at (0.5,-2.5) {};
        
        \node[draw=none, fill=none, left = 5mm of 0] {\large{$N$}};
        
        %Leaves
    
        \node[] (1) at (-1.5,-3.0) {};
        \node[draw=none, fill=none, below=1mm of 1] (leaf_1) {\large $1$};
        \node[] (2) at (-0.5,-3.0) {};
        \node[draw=none, fill=none, below=1mm of 2] (leaf_2) {\large $2$};
        \node[] (3) at (0.5,-3.0) {};
        \node[draw=none, fill=none, below=1mm of 3] (leaf_3) {\large $3$};
        \node[] (4) at (1.5,-3.0) {};
        \node[draw=none, fill=none, below=1mm of 4] (leaf_4) {\large $4$};
    
        %Edges
    
        \draw[edge] (-1) edge (0);
        \draw[edge] (0) edge (5);
        \draw[edge] (0) edge (6);
        \draw[edge] (5) edge (8);
        \draw[edge] (8) edge (1);
        \draw[edge] (5) edge (7);
        \draw[edge] (6) edge (7);
        \draw[edge] (6) edge (9);
        \draw[edge] (9) edge (10);
        \draw[edge] (9) edge (4);
        \draw[edge, red] (8) edge (2);
        \draw[edge, dotted] (7) edge (10);
        \draw[edge] (10) edge (3);
        
        \begin{scope}[shift={(-4,-0.5)}]
        %Nodes
    
        \node[] (-1) at (0.0,-0.0) {};
        \node[] (0) at (0.0,-0.5) {};
        \node[] (5) at (-0.5,-1.0) {};
        \node[] (6) at (0.5,-1.0) {};
        \node[] (7) at (0.75,-1.5) {};
        \node[square] (8) at (-0,-2.0) {};
        
     	\node[draw=none, fill=none, left = 5mm of 0] {\large{$N'$}};
        
        %Leaves
    
        \node[] (1) at (-1.0,-2.5) {};
        \node[draw=none, fill=none, below=1mm of 1] (leaf_1) {\large $1$};
        \node[] (2) at (0.0,-2.5) {};
        \node[draw=none, fill=none, below=1mm of 2] (leaf_2) {\large $3$};
        \node[] (3) at (1.0,-2.5) {};
        \node[draw=none, fill=none, below=1mm of 3] (leaf_3) {\large $4$};
    
        %Edges
    
        \draw[edge] (-1) edge (0);
        \draw[edge] (0) edge (5);
        \draw[edge] (0) edge (6);
        \draw[edge] (5) edge (1);
        \draw[edge] (5) edge (8);
        \draw[edge] (6) edge (8);
        \draw[edge] (6) edge (7);
        \draw[edge] (8) edge (2);
        \draw[edge] (7) edge (8);
        \draw[edge] (7) edge (3);
        \end{scope}
    \end{tikzpicture}
    \caption{An illustration of Lemma~\ref{lem:reductionImpliesContainmentSF}.
    The tree-child network~$N'$ is contained in the tree-child network~$N$, which can be seen by deleting leaf~$2$ and contracting the dotted edge in~$N$.
    The TCS~$(2,1),(3,4),(3,1),(3,4),(1,4)$ for~$N$ reduced~$N'$ as well.
    However,~$N'$ is clearly not a subnetwork of~$N$.}
    \label{fig:reductionImpliesContainmentSF}
\end{figure}

\subsection{Containment does not always imply reduction}\label{subsec:CounterExamples}

Following subsection~\ref{subsec:reductionimpliescontainment}, we give an example of a CPN~$N$ and a tree~$T$ on the same leaf-sets, for which~$T$ is a subnetwork of~$N$, such that there exists no minimal CPS for~$N$ that reduces~$T$ to a single leaf. This example is shown in Figure~\ref{fig:CPNTreeContainmentCounterExample}.

\begin{theorem}\label{thm:CPNTCFails}
There exists a binary CPN $N$ that contains a tree~$T$, such that no minimal CPS for $N$ reduces $T$. 
\end{theorem}

\begin{proof}
We refer by~$N$ and~$T$ to the network and the tree of Figure~\ref{fig:CPNTreeContainmentCounterExample}.

Note first that $\mathcal{C}(N)=\{(2,3),(6,5)\}$. So initially, we are required to pick one of the reticulated cherries $(2,3)$ or $(6,5)$.

Picking $(2,3)$ first reduces both $T$ and $N$, and in the next step we have the option of picking one of the reticulated cherries $(3,2)$ or $(6,5)$.
Picking $(3,2)$ does not affect $T(2,3)$, and reduces the reticulated cherry $(3,2)$ in $N(2,3)$.
In $N(2,3),(3,2)$, $3$ is no longer a child of a reticulation, and the up-down path connecting the leaves $1$ and $3$ contains the parent of $4$.
This implies that one of $3$ or $4$ is picked by the time we can pick the leaf $1$, which ultimately means that $T$ is not reduced to a leaf with any CPS starting with $(2,3),(3,2)$.
So we pick $(6,5)$.

Picking $(6,5)$ first reduces both $T$ and $N$, and in the next step we have the option of picking one of the reticulated cherries $(2,3)$ or $(5,6)$.
Picking $(5,6)$ does not affect $T(6,5)$, and reduces the reticulated cherry $(5,6)$ in $N(6,5)$.
The cherry $(5,7)$ in $T(6,5),(5,6)$ can only be picked in $N(6,5),(5,6)$ as the final cherry since the only up-down path from $5$ to $7$ in $N(6,5),(5,6)$ passes the child of the root.
This implies that a minimal CPS for $N$ starting with $(6,5),(5,6)$ cannot reduce $T$ to a single leaf, since the cherry $(5,7)$ occurs on one side of the tree.

Thus, the CPS must start with either $(2,3),(6,5)$ or $(6,5),(2,3)$. Note that $T(2,3),(6,5) = T(6,5),(2,3)$ and $N(2,3),(6,5) = N(6,5),(2,3)$ and so the order in which we pick the two reticulated cherries do not matter.
In both cases, we have the choice of picking one of the reticulated cherries $(3,2)$ or $(5,6)$.
However, doing so results in a CPS that does not reduce $T$ to a single leaf, by the argument above.
Since every CPS of $N$ starts with these cherries, we have that $T$ is not reduced to a single leaf for any CPS of $N$. 
\end{proof}

Since this particular network is semi-binary stack-free, it serves as an example for when reduction does not imply containment for the~$(\ref{Const:CherryResolved},\ref{Const:RCherryResolved})$ and the~$(\ref{Const:CherryResolved},\ref{Const:RCherrySF})$ classes.
For the other two reconstructible classes, the claim may still hold true.
We speculate that similar results follow, which we discuss in Section~\ref{sec:Discussion}.
% \todoYuki{To come back to later. Either find a counter-example or prove it!}.

\begin{figure}
    \centering
    \begin{subfigure}[b]{0.49\textwidth}
    \centering
    \begin{tikzpicture}[every node/.style = {draw, circle, fill, inner sep = 0pt, minimum size = 2mm},
square/.style = {regular polygon, regular polygon sides = 4, minimum size = 3 mm}]
    \tikzset {edge/.style = {very thick, shorten >= -0.5 pt}}

        %Nodes

        \node[] (-1) at (0.0,-0.0) {};
        \node[] (0) at (0.0,-0.5) {};
        \node[] (8) at (-0.5,-1.0) {};
        \node[] (11) at (0.5,-1.0) {};
        \node[] (9) at (-1,-1.5) {};
        \node[] (12) at (0.5,-1.5) {};
        \node[] (10) at (-0.75,-2.0) {};

        \node[draw=none, fill=none, left = 5mm of 0] {\large{$T$}};
        %Leaves

        \node[] (1) at (-1.5,-2.5) {};
        \node[draw=none, fill=none, below=1mm of 1] (leaf_1) {\large $1$};
        \node[] (2) at (-1.0,-2.5) {};
        \node[draw=none, fill=none, below=1mm of 2] (leaf_2) {\large $2$};
        \node[] (3) at (-0.5,-2.5) {};
        \node[draw=none, fill=none, below=1mm of 3] (leaf_3) {\large $3$};
        \node[] (4) at (0.0,-2.5) {};
        \node[draw=none, fill=none, below=1mm of 4] (leaf_4) {\large $4$};
        \node[] (5) at (0.5,-2.5) {};
        \node[draw=none, fill=none, below=1mm of 5] (leaf_5) {\large $5$};
        \node[] (6) at (1.0,-2.5) {};
        \node[draw=none, fill=none, below=1mm of 6] (leaf_6) {\large $6$};
        \node[] (7) at (1.5,-2.5) {};
        \node[draw=none, fill=none, below=1mm of 7] (leaf_7) {\large $7$};

        %Edges

        \draw[edge] (-1) edge (0);
        \draw[edge] (0) edge (8);
        \draw[edge] (0) edge (11);
        \draw[edge] (8) edge (9);
        \draw[edge] (8) edge (4);
        \draw[edge] (9) edge (1);
        \draw[edge] (9) edge (10);
        \draw[edge] (10) edge (2);
        \draw[edge] (10) edge (3);
        \draw[edge] (11) edge (12);
        \draw[edge] (11) edge (7);
        \draw[edge] (12) edge (5);
        \draw[edge] (12) edge (6);
    \end{tikzpicture}
    \end{subfigure}%
    \hfill
    \begin{subfigure}[b]{0.49\textwidth}
    \centering
    \begin{tikzpicture}[every node/.style = {draw, circle, fill, inner sep = 0pt, minimum size = 2mm},
square/.style = {regular polygon, regular polygon sides = 4, minimum size = 3 mm}]
    \tikzset {edge/.style = {very thick, shorten >= -0.5 pt}}

        %Nodes

        \node[] (-1) at (0.0,-0.0) {};
        \node[] (0) at (0.0,-0.5) {};
        \node[] (8) at (-0.5,-1.0) {};
        \node[] (16) at (0.5,-1.0) {};
        \node[] (9) at (-0.75,-1.5) {};
        \node[] (11) at (-0.25,-1.5) {};
        \node[] (17) at (1,-2.25) {};
        \node[] (10) at (-1,-2.25) {};
        \node[] (12) at (0,-2.25) {};
        \node[square] (13) at (-0.5,-2.75) {};
        \node[square] (18) at (0.5,-2.75) {};
        \node[] (14) at (-0.5,-3.25) {};
        \node[] (19) at (0.5,-3.25) {};
        \node[square] (15) at (-1.25,-3.5) {};
        \node[square] (20) at (1.25,-3.5) {};

        \node[draw=none, fill=none, left = 5mm of 0] {\large{$N$}};
        %Leaves

        \node[] (1) at (-2.0,-4.0) {};
        \node[draw=none, fill=none, below=1mm of 1] (leaf_1) {\large $1$};
        \node[] (2) at (-1.25,-4.0) {};
        \node[draw=none, fill=none, below=1mm of 2] (leaf_2) {\large $2$};
        \node[] (3) at (-0.5,-4.0) {};
        \node[draw=none, fill=none, below=1mm of 3] (leaf_3) {\large $3$};
        \node[] (4) at (0.25,-1.75) {};
        \node[draw=none, fill=none, below=1mm of 4] (leaf_4) {\large $4$};
        \node[] (5) at (0.5,-4.0) {};
        \node[draw=none, fill=none, below=1mm of 5] (leaf_5) {\large $5$};
        \node[] (6) at (1.25,-4.0) {};
        \node[draw=none, fill=none, below=1mm of 6] (leaf_6) {\large $6$};
        \node[] (7) at (2.0,-4.0) {};
        \node[draw=none, fill=none, below=1mm of 7] (leaf_7) {\large $7$};

        %Edges

        \draw[edge, red] (-1) edge (0);
        \draw[edge, red] (0) edge (8);
        \draw[edge, red] (0) edge (16);
        \draw[edge, red] (8) edge (9);
        \draw[edge, red] (8) edge (11);
        \draw[edge, bend right, red] (9) edge (1);
        \draw[edge, red] (9) edge (10);
        \draw[edge] (10) edge (15);
        \draw[edge, red] (10) edge (13);
        \draw[edge, red] (15) edge (2);
        \draw[edge, red] (13) edge (14);
        \draw[edge, red] (14) edge (15);
        \draw[edge, red] (14) edge (3);
        \draw[edge, red] (11) edge (4);
        \draw[edge] (11) edge (12);
        \draw[edge] (12) edge (13);
        \draw[edge] (12) edge (18);
        \draw[edge, red] (18) edge (19);
        \draw[edge] (19) edge (20);
        \draw[edge, red] (19) edge (5);
        \draw[edge, red] (20) edge (6);
        \draw[edge, red] (16) edge (17);
        \draw[edge, bend left, red] (16) edge (7);
        \draw[edge, red] (17) edge (18);
        \draw[edge, red] (17) edge (20);
    \end{tikzpicture}
    \end{subfigure}
    \caption{The tree~$T$ is a subnetwork of a CPN~$N$, and one embedding of $T$ into $N$ is shown (red edges).
    There exists no minimal CPS of $N$ that reduces $T$ to a single leaf.}
    \label{fig:CPNTreeContainmentCounterExample}
\end{figure}
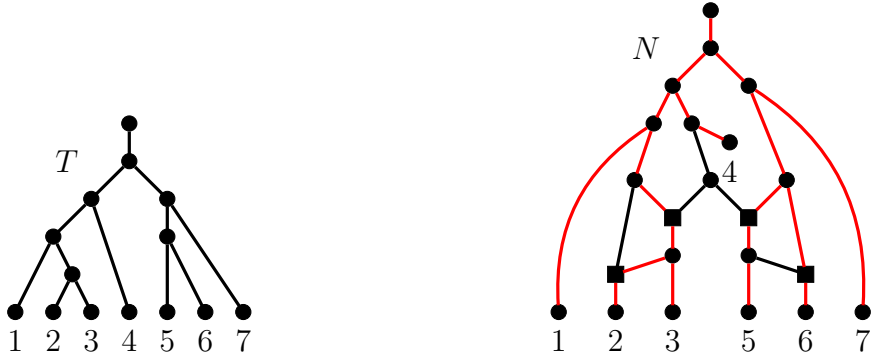

In what follows, we show that if we restrict our scope to the class of tree-child networks, then the converse does hold. That is, for those networks, containment implies reduction. %\todoRemie{TODO after making table: But only if they are in the same CPN class? (And we basically only focus on subnetworks instead of contained networks?) YM: subnetwork --> redux also holds for non-binary. It's just that in this case, redux --> containment doesn't hold.}
%converse of Lemmas~\ref{lem:reductionImpliesContainment} and~\ref{lem:reductionImpliesContainmentSF} do hold.

% \noindent Maybes:
% \begin{itemize}
%     \item Crown FPT algorithm; not possible anymore!
% \end{itemize}

% \noindent Potential problems:
% \begin{itemize}
%     \item Construct a network from CPS. Does reducing the cherries / ret. cherries in a different order cause any problems?---Solved by the lemma of moving cherries around.
% \end{itemize}

\section{Tree-child sequences}\label{sec:TCNetworkContainment}

Recall that a condition was imposed on a sequence of ordered pairs to define CPSs as those that can be used to construct networks.
We show here that imposing an additional condition on the CPSs can ensure the constructed network to be tree-child.
Recall that a network is tree-child if every non-leaf vertex has a child that is a tree vertex or a leaf.

In this section, we assume that we work in CPN classes that use the~$\ref{Const:RCherrySF}$ or the~$\ref{Const:RCherryUnresolved}$ reticulated cherry construction, to ensure that the network constructed from the sequences is stack-free.
Outside of stacks, the only forbidden structure of tree-child networks are tree vertices whose children are all reticulations.
To ensure these structures do not appear in the constructed networks, we impose a condition on our sequences, that the first coordinate of each pair does not appear as a second coordinate of another pair in the remainder of the sequence.
We will refer to this condition as the \emph{tree-child condition}.

\begin{definition}
A CPS is a \emph{tree-child sequence (TCS)} if every leaf appearing as the first coordinate does not appear as a second coordinate in the rest of the sequence.
\end{definition}

\begin{lemma}
Let~$S$ be a TCS.
Then each of the networks obtained by choosing a construction~$\ref{Const:RCherrySF}$ or~$\ref{Const:RCherryUnresolved}$ for each reticulated cherry is tree-child.
\end{lemma}
\begin{proof}
By construction, the networks obtained from~$S$ are CPNs and they are stack-free.
It remains to show that each tree vertex has at least one child that is a tree vertex or a leaf.

Let~$S_i=(x,y)$ be a reticulated cherry, and consider the network~$N$ after having just added~$S_i$.
In~$N$, the tree vertex parent~$p_y$ of~$y$ currently has at least one leaf child~$y$.
Suppose that all its other children were reticulations.
For this vertex~$p_y$ to have all reticulation children in the fully constructed network, we require some reticulation vertex to be inserted between~$p_y$ and~$y$, which can only happen if we add some ordered pair~$S_j = (y,z)$ where~$j<i$.
However, this would mean that~$y$ appears as a first coordinate of some pair~$S_j$ and also as a second coordinate of some pair~$S_i$ later on in the sequence, which contracts our tree-child condition.

Now, if~$S_i=(x,y)$ was a cherry, then the parent~$p$ of~$x$ cannot be a parent of only reticulations after adding more pairs to the network.
Indeed, this would imply that we have added some reducible pair~$(y,z)$ later on to the network---and hence it would appear earlier in the sequence---which again contradicts our tree-child condition.

Hence all tree vertices of a network obtained from a TCS has at least one child that is a tree vertex or a reticulation.
Therefore such a network is tree-child.
\end{proof}

% \begin{lemma}
% Let~$N$ be a non-binary TCN.
% Then there exists a minimal TCS that reduces it.
% \end{lemma}
% \begin{proof}
% \textbf{TODO}
% \end{proof}

Tree-child networks (TCNs) always contain a reducible pair, and after reducing one of these, we obtain a new tree-child network (Lemma~4.1 of \citep{bordewich2016determining}).
Naturally, this implies that TCNs are CPNs.
As we have seen in the previous section, given a CPN that contains a tree on the same set of taxa, there may not exist a minimal CPS of the network that reduces the tree (Theorem~\ref{thm:CPNTCFails}).
In this section, we make the switch from CPSs to TCSs, and show that this is no longer an issue for TCSs.
In fact, we prove a stronger result: within a reconstructible CPN class, a TCN contains another TCN on the same taxa if and only if every minimal TCS of the first TCN reduces the second TCN. 
If, in addition, the first TCN is binary, then if \emph{any} minimal TCS for the larger network reduces the smaller network, then the smaller network is a subnetwork of the larger network.
We also show that---similar to Proposition~\ref{prop:CPSorder} for CPNs---we may pick reducible pairs in any order for TCNs.
We start by showing that every TCN has a minimal TCS. 
This was shown implicitly for semi-binary TCNs in the proof of Lemma 3.4 of~\cite{linz2019attaching}; however we include it here for completeness and to generalize for non-binary TCNs.
% \todoRemie{TODO REMIE: Change text to say we actually do it here}
%I can't really find where they implicitly do this. Can we point to some spot in their paper? Otherwise, I think we should prove it ourselves. (Remove all reticulation arcs, then each component is a rooted tree with at least one leaf (as the network is tree-child). There is a partial order on these trees, where one is above the other if there is a reticulation edge from the one to the other. Remove the lowest (non-trivial) tree by cherry picking. If there is a trivial tree below it, this results in reticulated cherries that may be picked.))}

%%%%%%%%%%%%%%%%%%%%%%%%%%%%%%%%%%%%%%%%%%%%%%%%%
%%%%%%%%%%%%%%%%%%%%%%%%%%%%%%%%%%%%%%%%%%%%%%%%%
%%%%%%%%%%%%%%%%%%%%%%%%%%%%%%%%%%%%%%%%%%%%%%%%%
%%%%%%%%%%%%%%%%%%%%%%%%%%%%%%%%%%%%%%%%%%%%%%%%%
%%%%%%%%%%%%%%%%%%%%%%%%%%%%%%%%%%%%%%%%%%%%%%%%%
%%%%%%%%%%%%%%%%%%%%%%%%%%%%%%%%%%%%%%%%%%%%%%%%%
%%%%%%%%%%%%%%%%%%%%%%%%%%%%%%%%%%%%%%%%%%%%%%%%%

\begin{lemma}
\remiee{Let $N$ be a non-binary TCN. Then there is a TCS for $N$.}
\end{lemma}
\begin{proof}
Let $N$ be a network on $X$. For a partial TCS $S$, denote by $F(S)\subseteq X$ the set of first elements of pairs of $S$. Because $N$ is tree-child, for each node $v$, there is a path to a leaf $t(v)$ in which all internal nodes are tree nodes, and $t(v)=t(c)$ for some child $c$ of $v$ if $v$ is not a leaf node.

Assume $S$ is a partial TCS such that $t(v)$ is still below $v$ via a tree-path for all nodes $v$ of $NS$, and for each $l\in F(S)$, the nodes $v$ for which $t(v)=l$, are $l$, and possibly the parent of $l$ if it is a reticulation. Suppose $N$ is not fully reduced by $S$, we show that we can find a longer sequence $S'$ starting with $S$ to which the conditions also apply.

Let $p$ be a lowest tree node in $NS$, then each node below $p$ is either a leaf, or a reticulation which is directly above a leaf. The leaf $t(p)$ is directly below $p$ (as there is a tree-path from $p$ to $t(p)$) and $t(p)\not\in F(S)$ by the assumption on $S$.
By reducing all cherries and reticulated cherries involving $p$ using $t(p)$ as second element, we obtain a new partial tree-child sequence $S'$. In $NS'$, there is still a tree-path from the parent $g$ of $p$ to $t(p)$, so, if $t(g)=t(p)$, there is still a tree-path from $g$ to $t(g)$; the tree paths for none of the other nodes are affected the reductions between $NS$ and $NS'$. Furthermore, each leaf $l\in F(S')$ is either directly below a reticulation, or directly below a tree node $v$ with $t(v)\neq l$, as in $NS$, $l$ was below a reticulation below $p$; or $l\in F(S)$ and $l$ was already below $v$, so $t(v)\neq l$ in $NS$ by assumption.

Starting with the empty partial TCS, we can repeat this process until $NS$ has no more tree nodes. When this is the case, we have found a TCS $S$ for $N$.
\end{proof}

%%%%%%%%%%%%%%%%%%%%%%%%%%%%%%%%%%%%%%%%%%%%%%%%%
%%%%%%%%%%%%%%%%%%%%%%%%%%%%%%%%%%%%%%%%%%%%%%%%%
%%%%%%%%%%%%%%%%%%%%%%%%%%%%%%%%%%%%%%%%%%%%%%%%%
%%%%%%%%%%%%%%%%%%%%%%%%%%%%%%%%%%%%%%%%%%%%%%%%%
%%%%%%%%%%%%%%%%%%%%%%%%%%%%%%%%%%%%%%%%%%%%%%%%%
%%%%%%%%%%%%%%%%%%%%%%%%%%%%%%%%%%%%%%%%%%%%%%%%%

\subsection{Subnetwork / Containment implies reduction}

As in the previous sections, we will try to keep the results as general as possible.
We first show results on reduction and subnetworks.

\begin{lemma}\label{lem:SubnetworkReduction}
Let $N$ be a \remie{non-binary} tree-child network, $N'$ a tree-child subnetwork of $N$ with the same leaf set, and $S$ a TCS. Then $N'S$ is a subnetwork of $NS$.
\end{lemma}
\begin{proof}
We prove this fact inductively on the length of $S$. If $S$ is empty, then $N'S=N'$ and $NS=N$, so $N'S$ is a subnetwork of $NS$.

Now suppose that for any TCS $S'$ of length at most $j$, $N'S'$ is a subnetwork of $NS'$. We prove that for any TCS $S$ of length $j+1$, $N'S$ is a subnetwork of $NS$. Let us denote $S=S'(x,y)$. Note that $S'$ is of length at most~$j$.
Hence, by the induction hypothesis,~$N'S'$ is a subnetwork of~$NS'$.
We consider the following cases:
\begin{itemize}
\item \textbf{$\bm{N'S'}$ has only one of $\bm{x}$ and $\bm{y}$.} Because $S=S'(x,y)$ is a TCS, $y$ is not the first coordinate in any element of $S'$. Hence, $N'S'$ must still contain $y$, and $x$ must have been deleted from $N'$ by applying $S'$. This means the edge of $NS'$ deleted by applying $(x,y)$ is not used by the embedding of $N'S'$ into $NS'$, and $N'S=N'S'$ can still be embedded in $NS$.
\item \textbf{$\bm{N'S'}$ has both $\bm{x}$ and $\bm{y}$.} There are a few cases we must consider, depending on whether there are reducible pairs $(x,y)$ in $N'S'$ and $NS'$.
\begin{itemize}
    \item \textbf{$\bm{NS'}$ has a cherry $\bm{(x,y)}$.} As $N'S'$ also contains both leaves $x$ and $y$, $N'S'$ also has the cherry $(x,y)$, which is mapped to the corresponding cherry in $NS'$ by the embedding. The reduction of $(x,y)$ in both networks removes the pendant edge leading to $x$ in both networks, not changing the embedding otherwise.
    \item \textbf{$\bm{NS'}$ has a reticulated cherry $\bm{(x,y)}$.} 
    \begin{itemize}
        \item \textbf{The edge $\bm{p_yp_x}$ is used by the embedding of $\bm{N'S'}$ into $\bm{NS'}$.}
            First note that $N'S'$ must have either a cherry or a reticulated cherry $(x,y)$: if the edge $p_yp_x$ is used by the embedding, then the only way to reach $x$ and $y$ in $NS'$, is by using the edges $p_xx$ and $p_yy$, making $(x,y)$ either a cherry or a reticulated cherry in $N'S'$. Now applying $(x,y)$ to $N'S'$ deletes the edge using $p_yp_x$ of $NS'$ in the embedding. %\todoYuki{I think this should be a path, since for a cherry~$(x,y)$ its a length 2 path.}
            Hence, upon deleting both these edges by applying $(x,y)$ to both networks, we may naturally extend the embeddings.
        \item \textbf{Otherwise.} $N'S'$ can be embedded in $NS'$ without the edge $p_yp_x$. Hence, $N'S'$ is a subnetwork of $NS'$ after the removal of this edge $p_yp_x$ (i.e., the network $NS$). As $N'S$ is a subnetwork of $N'S'$ %\todoYuki{This hasn't been proven yet? It's Lemma 10.}\todoRemie{I think Lemma~10 says $NcS\subseteq NS$, whereas here we use $NSc\subseteq NS$. The first is non-trivial, whereas the second is trivial (reducing more can only give you a subnetwork).} 
        and $N'S'$ is a subnetwork of $NS$, $N'S$ is a subnetwork of $NS$. 
    \end{itemize}

    \item \textbf{Otherwise.} The network $NS'$ contains neither a cherry, nor a reticulated cherry on $x$ and $y$. This means $NS=NS'$. As $N'S$ is a subnetwork of $N'S'$, and $N'S'$ is a subnetwork of $NS'$, $N'S$ is a subnetwork of $NS (=NS')$.  
\end{itemize}
\end{itemize}
\end{proof}

The following corollary follows immediately, as subnetworks of single-leaf networks are single-leaf networks.

\begin{corollary}\label{cor:ContainmentImpliesReductionTC}
Let $N$ and $N'$ be non-binary tree-child networks on the same leaf set, with $N'$ a subnetwork of $N$. If a TCS $S$ reduces $N$, then $S$ also reduces $N'$.
\end{corollary}

Observe that in the setting of Corollary~\ref{cor:ContainmentImpliesReductionTC}, the two networks are reduced to the same single-leaf network, with the same leaf label.

\begin{theorem}\label{the:SubnetworkIffTCSReduces}
Let~$C$ be a reconstructible class of CPNs.
Let $N$ and $N'$ be tree-child networks in~$C$ on the same leaf set.
Then $N'$ is contained in $N$ if and only if \emph{any} TCS of $N$ reduces $N'$.
\end{theorem}
\begin{proof}
Follows immediately from Corollary~\ref{cor:ContainmentImpliesReductionTC}, Lemma~\ref{lem:reductionImpliesContainment}, and Lemma~\ref{lem:reductionImpliesContainmentSF}.
\end{proof}
Note that in Theorem~\ref{the:SubnetworkIffTCSReduces}, it is necessary for the two networks to be contained in the same reconstructible class of CPNs.
Consider the networks in the~(\ref{Const:CherryResolved}, \ref{Const:RCherryUnresolved}) and the~(\ref{Const:CherryUnresolved}, \ref{Const:RCherryUnresolved}) classes in Figure~\ref{fig:CPNClass}, which are both tree-child and can be reduced by the same tree-child sequence.
The latter network is not contained in the former.

For semi-binary TCNs, we show that the notions of containment and subnetwork are equivalent.

\begin{lemma}\label{lem:Contain=SubnetSBSF}
Let~$N,N'$ be both semi-binary TCNs (TCNs of the $(\ref{Const:CherryResolved},\ref{Const:RCherrySF})$-class) on the same leaf-set.
Then~$N$ contains~$N'$ if and only if~$N'$ is a subnetwork of~$N$.
\end{lemma}
\begin{proof}
One direction is clear by definition of containment and subnetwork, so suppose that~$N'$ is contained in~$N$.
Then there exists some refinement~$N'_f$ of~$N'$ that is a subnetwork of~$N$.
If~$N'_f = N'$, then we are done, so there must exist an rr-edge~$e$ of~$N'_f$ that is mapped to an rtr-path~$P_e$ of~$N$ in the embedding.
By definition of tree-child networks, there exist tt-paths from each vertex to some leaf in~$N$.
Let~$t$ be a tree vertex on~$P_e$, and let~$l$ denote the leaf that can be reached from~$t$ via a tt-path.
Since this tt-path is not used in the embedding, this inherently implies that~$l$ does not appear as a leaf in the network~$N'_f$ and therefore in~$N'$.
This is a contradiction as~$N$ and~$N'$ have the same leaf-sets.
It follows then that reticulations in~$N'$ need not be refined to obtain~$N'_f$.
However, since all tree vertices in~$N'$ are binary, all refinements of~$N'$ arise from refining its reticulation vertices.
Then~$N'_f = N'$, and therefore~$N'$ is a subnetwork of~$N$.
\end{proof}

This lemma does not hold in general, for all other network classes for which the smaller network is not binary (see Figure~\ref{fig:ContainNoImplySubnet}).
Furthermore the lemma does not hold for when the larger network~$N$ is not a TCN.

\begin{figure}
    \begin{subfigure}[b]{0.32\textwidth}
    \centering
    \begin{tikzpicture}[every node/.style = {draw, circle, fill, inner sep = 0pt, minimum size = 2mm},
square/.style = {regular polygon, regular polygon sides = 4, minimum size = 3 mm}]
    \tikzset {edge/.style = {very thick, shorten >= -0.5 pt}}

        %Nodes

        \node[] (-1) at (0.0,-0.0) {};
        \node[] (0) at (0.0,-0.5) {};

        \node[draw=none, fill=none, left = 5mm of 0] {\large{$T$}};
        %Leaves

        \node[] (1) at (-1.5,-2.0) {};
        \node[draw=none, fill=none, below=1mm of 1] (leaf_1) {\large $1$};
        \node[] (2) at (-0.5,-2.0) {};
        \node[draw=none, fill=none, below=1mm of 2] (leaf_2) {\large $2$};
        \node[] (3) at (0.5,-2.0) {};
        \node[draw=none, fill=none, below=1mm of 3] (leaf_3) {\large $3$};
        \node[] (4) at (1.5,-2.0) {};
        \node[draw=none, fill=none, below=1mm of 4] (leaf_4) {\large $4$};

        %Edges

        \draw[edge] (-1) edge (0);
        \draw[edge] (0) edge (1);
        \draw[edge] (0) edge (2);
        \draw[edge] (0) edge (3);
        \draw[edge] (0) edge (4);
    \end{tikzpicture}
    \end{subfigure}
    \begin{subfigure}[b]{0.32\textwidth}
    \centering
    \begin{tikzpicture}[every node/.style = {draw, circle, fill, inner sep = 0pt, minimum size = 2mm},
square/.style = {regular polygon, regular polygon sides = 4, minimum size = 3 mm}]
    \tikzset {edge/.style = {very thick, shorten >= -0.5 pt}}

        %Nodes

        \node[] (-1) at (0.0,-0.0) {};
        \node[] (0) at (0.0,-0.5) {};
        \node[] (5) at (-0.5,-1.0) {};

        \node[draw=none, fill=none, left = 5mm of 0] {\large{$T_f$}};
        %Leaves

        \node[] (1) at (-1.5,-2) {};
        \node[draw=none, fill=none, below=1mm of 1] (leaf_1) {\large $1$};
        \node[] (2) at (-0.5,-2) {};
        \node[draw=none, fill=none, below=1mm of 2] (leaf_2) {\large $2$};
        \node[] (3) at (0.5,-2) {};
        \node[draw=none, fill=none, below=1mm of 3] (leaf_3) {\large $3$};
        \node[] (4) at (1.5,-2) {};
        \node[draw=none, fill=none, below=1mm of 4] (leaf_4) {\large $4$};

        %Edges

        \draw[edge] (-1) edge (0);
        \draw[edge] (0) edge (5);
        \draw[edge] (0) edge (3);
        \draw[edge] (0) edge (4);
        \draw[edge] (5) edge (1);
        \draw[edge] (5) edge (2);
    \end{tikzpicture}
    \end{subfigure}
    \begin{subfigure}[b]{0.32\textwidth}
    \centering
    \begin{tikzpicture}[every node/.style = {draw, circle, fill, inner sep = 0pt, minimum size = 2mm},
    square/.style = {regular polygon, regular polygon sides = 4, minimum size = 3 mm}]
    \tikzset {edge/.style = {very thick, shorten >= -0.5 pt}}

        %Nodes

        \node[] (-1) at (0.0,-0.0) {};
        \node[] (0) at (0.0,-0.5) {};
        \node[] (5) at (-0.5,-1.0) {};
        \node[] (6) at (0.5,-1.0) {};
        \node[] (7) at (-0.5,-1.5) {};
        \node[] (8) at (0.5,-1.5) {};
        \node[square] (9) at (0.0,-2.0) {};

        \node[draw=none, fill=none, left = 5mm of 0] {\large{$N$}};
        %Leaves

        \node[] (1) at (-1.0,-2.5) {};
        \node[draw=none, fill=none, below=1mm of 1] (leaf_1) {\large $1$};
        \node[] (2) at (0.0,-2.5) {};
        \node[draw=none, fill=none, below=1mm of 2] (leaf_2) {\large $2$};
        \node[] (3) at (1.0,-2.5) {};
        \node[draw=none, fill=none, below=1mm of 3] (leaf_3) {\large $3$};
        \node[] (4) at (2.0,-2.5) {};
        \node[draw=none, fill=none, below=1mm of 4] (leaf_4) {\large $4$};

        %Edges

        \draw[edge, red] (-1) edge (0);
        \draw[edge, red, bend right = 30] (0) edge (5);
        \draw[edge, bend left = 30] (0) edge (5);
        \draw[edge, bend right = 30] (0) edge (6);
        \draw[edge, red, bend left = 30] (0) edge (6);
        \draw[edge, red] (5) edge (7);
        \draw[edge, red] (6) edge (8);
        \draw[edge, red] (7) edge (1);
        \draw[edge, red] (7) edge (9);
        \draw[edge, red] (9) edge (2);
        \draw[edge] (8) edge (9);
        \draw[edge, red] (8) edge (3);
        \draw[edge, red, bend left] (0) edge (4);
    \end{tikzpicture}

    \end{subfigure}
    \caption{The tree~$T$ is contained in the tree-child network~$N$ using the refinement~$T_f$, but it is not a subnetwork of~$N$.
    In particular,~$T$ and~$N$ belong to the~$(\ref{Const:CherryUnresolved}, \ref{Const:RCherryStack})$ and the~$(\ref{Const:CherryUnresolved}, \ref{Const:RCherryUnresolved})$ classes.}
    \label{fig:ContainNoImplySubnet}
\end{figure}
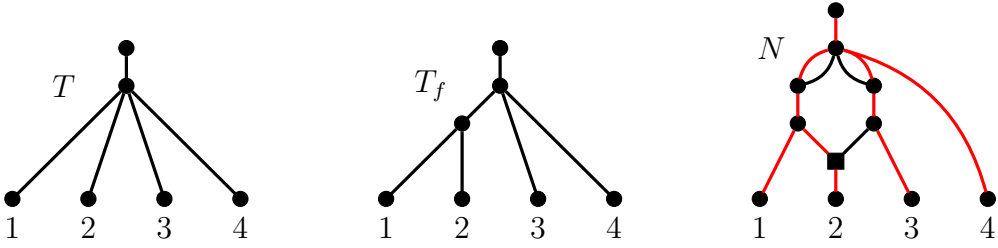

\begin{theorem}\label{the:SubnetworkIffTCSReduces3}
Let $N$ and $N'$ be semi-binary TCNs on the same leaf set.
Then $N'$ is a subnetwork of $N$ if and only if any TCS of $N$ reduces $N'$.
\end{theorem}
\begin{proof}
Follows immediately from Theorem~\ref{the:SubnetworkIffTCSReduces} and Lemma~\ref{lem:Contain=SubnetSBSF}
\end{proof}
Note that we could assume $N'$ is non-binary, but to be a subnetwork of a semi-binary network, it has to be semi-binary as well, because $N$ and $N'$ are both TCNs on the same leaf set.
Note that Theorem~\ref{the:SubnetworkIffTCSReduces3} is no longer true if we allow for the networks to have different leaf-sets.
Let~$N'$ be the cherry on leaf-set~$\{1,2\}$, and let~$N$ be the balanced tree on leaf-set~$\{1,2,3,4\}$ with cherries~$(1,3)$ and~$(2,4)$.
Then both networks are semi-binary tree-child, and~$N'$ is a subnetwork of~$N$.
However, the TCS~$(1,3),(2,4),(3,4)$ of~$N$ does not reduce~$N'$.

\subsection{Order doesn't matter in TCSs}
Theorem~\ref{thm:OrderDoesn'tMatter} states that we may pick cherries from a CPN in any order and still obtain a minimal CPS. In this section, we show that this also holds for TCSs.

\begin{lemma}\label{lem:TCseqSubsequence}
Let $N$ be a non-binary TCN, $S$ a partial TCS and $c$ an ordered pair of leaves such that $cS$ is also a partial TCS. Then $NcS$ is a subnetwork of $NS$.
\end{lemma}
\begin{proof}
We consider the networks $NcS_{[:j]}$ and $NS_{[:j]}$ for $j\geq 0$. We prove, using induction on~$j$, that $NcS_{[:j]}$ is a subnetwork of $NS_{[:j]}$. Obviously, this is true for $j=0$, as $Nc$ is a subnetwork of $N$.

Write $S_{j+1}=(x,y)$ for the~$j+1$-th element in the sequence. First note that if $x$ is not a leaf of $NcS_{[:j]}$, then $NcS_{[:j]}$ is still a subnetwork of $NS_{[:j]}(x,y)$, as the embedding of $NcS_{[:j]}$ in $NS_{[:j]}$ does not use the removed edge that leads only to $x$. Therefore $NcS_{[:j+1]}=NcS_{[:j]}$ is a subnetwork of $NS_{[:j+1]}= NS_{[:j]}(x,y)$.

Now assume $x$ is a leaf of $NcS_{[:j]}$. 
If reducing $NS_{[:j]}$ with $S_{j+1}=(x,y)$ does not remove an edge $e$ used by the embedding of $NcS_{[:j]}$, then $NcS_{[:j+1]}$ is still a subnetwork of $NS_{[:j+1]}$. 

So, we now assume reducing $S_{j+1}=(x,y)$ in $NS_{[:j]}$ removes an edge $e$ used by the embedding of $NcS_{[:j]}$.
Suppose for a contradiction that $NcS_{[:j+1]}$ is not a subnetwork of $NS_{[:j+1]}$.
For this to be true, the edge of $NcS_{[:j]}$ mapped to the path containing $e$ must not be removed by reducing $S_{j+1}$. This can only happen if~$NcS_{[:j]}$ does not contain the reducible pair~$(x,y)$, whereas~$NS_{[:j]}$ does have it.
The reducible pair $(x,y)$ in $NS_{[:j]}$ must be a reticulated cherry, as the only other option is that it is a cherry, but then, $NcS_{[:j]}$ also contains this cherry, as $x$ and $y$ are both part of $NcS_{[:j]}$ and $NcS_{[:j]}$ is displayed by $NS_{[:j]}$.
Hence,~$NcS_{[:j]}$ does not contain a cherry nor a reticulated cherry~$(x,y)$, and~$NS_{[:j]}$ contains a reticulated cherry~$(x,y)$.
As $NcS_{[:j]}$ is a subnetwork of $NS_{[:j]}$ whose embedding uses the reticulation edge of the reticulated cherry $(x,y)$, the leaf $y$ must have been deleted from $NcS_{[:j]}$ already. This means $y$ must be the first element of a pair in the partial TCS $cS_{[:j]}$. However, $cS$ is a partial TCS with $y$ as the second coordinate of $S_j$, so we have a contradiction. We conclude that if the reduction removes an edge from $NS_{[:j]}$, it also removes the corresponding edge (i.e., the edge of~$NcS_{[:j]}$ that was mapped to the path of~$NS_{[:j]}$ containing the edge) from $NcS_{[:j]}$. 

We conclude that $NcS_{[:j+1]}$ is a subnetwork of $NS_{[:j+1]}$.
%\delete{Now assume $x$ is a leaf of $NcS_{[:j]}$. Suppose for a contradiction that reducing $NS_{[:j]}$ with $S_{j+1}=(x,y)$ removes an edge used for the embedding of $NcS_{[:j]}$, but reducing $NcS_{[:j]}$ with $(x,y)$ does not remove the edge corresponding to the path including the removed edge of $NS_{[:j]}$. This means $NcS_{[:j]}$ must not have the reducible pair $(x,y)$, whereas $NS_{[:j]}$ does have it. As $NcS_{[:j]}$ is a subnetwork of $NS_{[:j]}$ whose embedding uses the reticulation edge of the reticulated cherry $(x,y)$, the leaf $y$ must have been deleted from $NcS_{[:j]}$ already. This means $y$ must be the first element of a pair in the partial TCS $cS_{[:j]}$. However, $cS$ is a partial TCS with $y$ as the second coordinate of $S_j$, so we have a contradiction. We conclude that if the reduction removes an edge from $NS_{[:j]}$, it also removes the corresponding edge from $NcS_{[:j]}$. This implies that $NcS_{[:j+1]}$ is a subnetwork of $NS_{[:j]}$.}
\end{proof}

\begin{lemma}\label{lem:intermediateTCSforSubsequence}%\todoRemie{I agree we need this, I'm not happy with using $T$ for a TCS, though. Can't we just use $S=S_0,\ldots,S_k=S'$? Or do we then get to many S's? YM: Yea I think we do get too many S's, especially when we would write~$S^{|S|-|S'|}$.}
Let~$S$ and~$S'$ be TCSs such that~$S'$ is a subsequence of~$S$.
Then there exist a sequence of TCSs~$S=\Sigma^0,\dots,\Sigma^{|S|-|S'|}=S'$ such that~$|\Sigma^i| = |\Sigma^{i-1}| - 1$ for~$i\in[|S|-|S'|]$.
\end{lemma}
\begin{proof}
Let~$S''$ denote the sequence (not necessarily a CPS) obtained by taking the elements of~$S$ which do not occur in~$S'$ (in order).
For~$i\in [|S''|]$, let~$f(i)$ denote the position where~$\Sigma^{i-1}_{f(i)} = S_i''$, which is the element in the sequence that will be deleted to obtain~$\Sigma^{i+1}$.
% Let~$f(0) = 0$.
We claim that if~$\Sigma^{i}$ is a TCS, then we can obtain a TCS~$\Sigma^{i+1}$ by removing the element~$S''_i = \Sigma^{i-1}_{f(i)}$.
%which occurs for the first time in~$\Sigma^i_{[f(i-1):]}$.
Upon repeating this for all~$i$, we obtain the sequence of TCSs that we are after.

To show this, suppose for a contradiction that~$\Sigma^i$ is a TCS, but removing~$S''_i$ from~$\Sigma^i$ results in a sequence~$\Sigma^{i+1}$ that is not a TCS.
Note first that the `tree-child property' of the sequence is retained when deleting elements from TCSs: indeed, every leaf appearing as a first coordinate
still does not appear as a second coordinate in the rest of the sequence.
So we must have that~$\Sigma^{i+1}$ is not a CPS, that is, there exists a leaf that appears as a second coordinate, but not as a first coordinate in the remaining sequence.
%Let~$j$ denote the position in the TCS~$\Sigma^i$ where~$\Sigma^i_j = S''_i$.
Then the CPS property is violated for an element which occurs in~$\Sigma^i_{[:f(i)-1]}$.
Note that~$\Sigma^i_{[:f(i)-1]}$ forms the first~$f(i)-1$ elements of~$S'$, since we defined~$S''$ as the elements of~$S$ which do not occur in~$S'$ in order.
Furthermore, note that at each step we do not add elements to the sequence.
Then~$\Sigma^{|S|-|S'|} = S'$ cannot be a CPS, let alone a TCS, which contradicts our assumption that~$S'$ was a TCS.
Therefore~$\Sigma^{i+1}$ must be a TCS and the lemma follows.
\end{proof}

\begin{corollary}\label{cor:TCSubsequences}
Let $N$ be a \remie{non-binary} TCN with $S$ and $S'$ TCSs such that $S'$ is a subsequence of $S$. Then $NS$ is a subnetwork of $NS'$.
\end{corollary}
\begin{proof}
By Lemma~\ref{lem:intermediateTCSforSubsequence}, we only have to prove this for $S'$ of length one less than $S$, so suppose $S=S'_{[:i]}sS'_{[i+1:]}$, where the first or the last part may be empty.

We consider $NS=NS'_{[:i]}sS'_{[i+1:]}$, by applying the three parts of the sequence separately. First we note that by writing $N':=NS'_{[:i]}$, we have $NS=N'sS'_{[i+1:]}$ and $NS'=N'S'_{[i+1:]}$. Because $S$ is a TCS, the sequence $sS'_{[i+1:]}$ is in particular also a TCS. Hence, by Lemma~\ref{lem:TCseqSubsequence}, $NS$ is a subnetwork of $NS'$.
\end{proof}

The following proposition says that we can find a TCS for a TCN by picking an arbitrary reducible pair, reducing it, and repeating this process.
\begin{proposition}\label{prop:TCSOrder}
%\todoRemie{Can this be made to work for non-binary TCNs? YM: Probably, but Observation 4 is not true in general for non=binary. We may have to show that there exists some TCS that contains (x,y) from a large number of cases.}
Let $N$ be a \remie{semi-binary} TCN and $S$ a partial TCS with every element reducing something in $N$. Then there exists a minimal TCS for $N$ starting with $S$.  
\end{proposition}
\begin{proof}
%THE thing to take care of in this proof is that we cannot just add any sequence to $S$, they must together be TC again.
We prove this using induction on the length $l$ of $S$. If $l$ is $0$, then, \remie{as each TCN has a minimal TCS,} there is a TCS $S'=SS'$ for $N$, which starts with $S$. Now suppose for any partial TCS $S$ of length $l<L$ reducing something in $N$ in every step, there is a minimal TCS $SS'$ of $N$. We prove that the same holds for any such sequence of length $L$.

Let $S(x,y)$ be such a sequence of length $L$ (where $(x,y)$ is the last element of this sequence). Because each element of $S(x,y)$ reduces something in $N$, we know in particular that $(x,y)\in\mathcal{C}(NS)$. By the induction hypothesis, there is a TCS $SS'$ for $N$ starting with $S$. The part $S'$ of this sequence must contain an element $(x,y)$ or $(y,x)$ as $N$ is semi-binary (Observation~\ref{obs:ChangingReducibleCherries}). Let $S'_i$ be the first occurrence of such an element. 

Each of the intermediate networks $NSS'_{[:j]}$ for $j< i$ has the reducible pair $(x,y)$. This means the only pairs involving $x$ reducing something in these networks have $x$ as first coordinate, or are equal to $(y,x)$. As $S'_i$ is the first occurrence of $(x,y)$ or $(y,x)$ in $S'$, all $S'_j$ with $j<i$ cannot have $x$ as second coordinate. This means that $S(x,y)S'$ is a TCS, and it reduces $N$ by Corollary~\ref{cor:TCSubsequences}. 

Note that each element of $S(x,y)$ reduces something in $N$. This means that there is an element of $S'$ that reduces nothing in $N$ in the sequence $S(x,y)S'$ (otherwise $N$ has $k$ as well as $k+1$ reticulations). Removing this element gives a minimal TCS for $N$ starting with~$S(x,y)$.
\end{proof}

\begin{proposition}
Let~$N$ be a non-binary TCN and~$S$ a partial TCS with every element reducing something in~$N$.
Then there exists a minimal TCS for~$N$ starting with~$S$.
\end{proposition}
\begin{proof}
Let~$N'$ denote the refinement of~$N$ obtained by applying the~$(\ref{Const:CherryResolved}, \ref{Const:RCherrySF})$ construction on~$NS$ for the sequence~$S$.
Observe that~$N'$ is a TCN since~$NS$ is tree-child.
Because of this, every multifurcation of~$N'$ has a child that is a leaf or a tree vertex, which we call~$t$.
By refining each multifurcation as a path in which the lowest vertex is the parent of~$t$, we obtain a semi-binary refinement~$N'_{sb}$ of~$N'$ that is tree-child.
In particular, the tree-vertices of the reducible pairs in~$N'_{sb}$ that are reduced by~$S$ are binary, because of the construction we have used to obtain~$N'$.
The way in which we have obtained~$N'_{sb}$ ensures that~$S$ is a partial TCS that reduces something in~$N'_{sb}$.
By Proposition~\ref{prop:TCSOrder}, there exists a minimal TCS~$S'$ for~$N'_{sb}$ starting with~$S$.
Since~$N'_{sb}$ is a refinement of~$N$, it follows then that~$S'$ must also reduce~$N$.
\end{proof}

\section{Computational aspects of containment problems}\label{sec:Computation}
{\sc Tree Containment} is a well studied problem, where one asks whether a tree is contained in a given network (with a common set of taxa). In this section, we look at the more general problem {\sc Network Containment}, where the aim is to determine whether a network is contained in another network. We restrict our attention to the problem where the input networks are both semi-binary, tree-child %, are either both semi-binary or both binary, \todoRemie{Check why assuming binary is interesting, it seems to be covered by semi-binary}
networks with the same leaf set. We will give an algorithm for this problem that runs in linear time.
% To prove the correctness, we need several structural results about the order of TCSs, like for CPSs in Section~\ref{sec:OrderCPS}.
We also show that it is possible to check in linear time whether two CPNs are isomorphic.\medskip

\fbox{
\parbox{0.8\linewidth}{
{\sc Network Containment}\\
{\bf Instance:} Two networks~$N$ and~$N'$ on the same leaf-set.\\
{\bf Question:} Does~$N$ contain~$N'$?}
}

\subsection{Tree-child Network Containment}
In this section, we give the linear time algorithm (Algorithm~\ref{alg:Subnetwork}) for {\sc Network Containment}. We use a few small subroutines (Algorithms~\ref{alg:CheckCherry},~\ref{alg:CheckRetCherryFirstCoordinate},~\ref{alg:ReducePair},~\ref{alg:FIndTCS},~\ref{alg:TCSReduceNetwork}). 
The idea of the algorithm follows from Theorem~\ref{the:SubnetworkIffTCSReduces} and Proposition~\ref{prop:TCSOrder}.
For two TCNs~$N$ and~$N'$, find a minimal TCS of~$N$ by picking reducible pairs in any order.
See if this TCS reduces~$N'$ to a network on a single leaf: if it does, then~$N$ contains~$N'$; otherwise,~$N$ does not contain~$N'$.
Assume in this subsection that every network is semi-binary stack-free, unless stated otherwise.

Within semi-binary networks, tree vertices are of outdegree-$2$.
This means that each leaf appears as a second coordinate in at most one reducible pair in a network.
Algorithm~\ref{alg:CheckCherry} finds such a reducible pair, if it exists, for a given leaf, in constant time.

\begin{algorithm}[H]\label{alg:CheckCherry}
 \KwData{A semi-binary network $N$ and a leaf $x$}
 \KwResult{The set with only element the reducible pair of $N$ having $x$ as second coordinate if it exists, $\emptyset$ otherwise.}
Let $p$ be the parent of $x$\;
\If{$p$ is a tree node}{
 let $c(p)$ be the child of $p$ that is not~$x$\;
 \uIf{$c(p)$ is a leaf}{
  \Return $\{(c(p),x)\}$\;
 }
 \If{$c(p)$ is a reticulation and the child $c(c(p))$ of $c(p)$ is a leaf}{
  \Return $\{(c(c(p)),x)\}$\;
 }
}
\Return $\emptyset$\;
\caption{{\sc FindRP2nd}$(N,x)$}
\end{algorithm}

\begin{lemma}\label{lem:AlgCheckCherry}
Let $N$ be semi-binary, and $x$ a leaf of $N$. If a reducible pair with $x$ as the second element of the pair exists, then Algorithm~\ref{alg:CheckCherry} finds this pair. Otherwise, it returns the empty set. The algorithm runs in constant time.
\end{lemma}
\begin{proof}
Each leaf is a second coordinate of at most one reducible pair, and this pair is found by taking the unique path down from the parent of this leaf. This algorithm runs in constant time: using in- and out-adjacency lists, we can check in constant time whether a node is a tree node, a reticulation node, or a leaf, and the out-list of each node has size at most 2 (as tree nodes have outdegree~2, reticulations outdegree~1, and leaves outdegree~0). 
\end{proof}

Algorithm~\ref{alg:CheckRetCherryFirstCoordinate} on the other hand finds all reticulated cherries that contain a given leaf as the first coordinate of the reducible pair.
The running time for this algorithm depends on the indegree of the parent of the given leaf, as this gives the maximum possible number of such reticulated cherries.

\begin{algorithm}[H]\label{alg:CheckRetCherryFirstCoordinate}
 \KwData{A semi-binary network $N$ and a leaf $x$}
 \KwResult{The set $\{(l,k)\in\mathcal{C}(N): l=x\}$ of reticulated cherries in $N$ having $x$ as first coordinate}
Let $p$ be the parent of $x$\;
Set~$\mathcal{C}_r = \emptyset$\;
\If{$p$ is a reticulation}{
 \For{every parent~$g$ of~$p$}{
  let~$c(g)$ be the other child of~$g$\;
  \uIf{$c(g)$ is a leaf}{
   $\mathcal{C}_r = \mathcal{C}_r \cup \{(x,c(g))\}$
  }
 }
}
\Return $\mathcal{C}_r$\;
\caption{{\sc FindRC1st}$(N,x)$}
\end{algorithm}

\begin{lemma}
Let~$x$ be a leaf in a semi-binary network~$N$ and let~$p_x$ denote the parent of~$x$.
Let~$I$ denote the indegree of~$p_x$.
Algorithm~\ref{alg:CheckRetCherryFirstCoordinate} finds the set of all reticulated cherries with the reticulation on the leaf~$x$ in~$O(I)$ time.
%in constant time, if the indegree of the reticulation nodes is bounded. 
\end{lemma}
\begin{proof}
We simply check that the parent of~$x$ is a reticulation, and that~$(x,y)$ is a reticulated cherry if a grandparent of~$x$ is the parent of~$y$.
The for loop iterates at most~$I$ times.
The steps within the for loop runs in constant time, since we may use the in- and out-adjacency lists as stated in the proof of Lemma~\ref{lem:AlgCheckCherry}.
Therefore, Algorithm~\ref{alg:CheckRetCherryFirstCoordinate} runs in~$O(I)$ time.
%constant time if the indegree of the reticulation nodes is bounded.
\end{proof}

\begin{algorithm}[H]\label{alg:ReducePair}
 \KwData{A non-binary network $N$ and a pair $(x,y)$}
 \KwResult{The network $N(x,y)$}
 \uIf{$(x,y)$ is a cherry in $N$}{
  Let $p$ be the parent of $x$ and $y$\;
  Remove edge $px$ from $N$\; 
  Suppress $p$ (if necessary) and remove~$x$ in $N$\;
 }
 \If{$(x,y)$ is a reticulated cherry in $N$}{
  Let $p_x$ be the parent of $x$ and $p_y$ the parent of $y$\;
  Remove edge $p_yp_x$ from $N$\; 
  Suppress $p_y$ and~$p_x$ (if necessary) in $N$\;
 }
\Return $N$\;
\caption{{\sc ReducePair}$(N,(x,y))$}
\end{algorithm}

\begin{lemma}
Algorithm~\ref{alg:ReducePair} reduces a given reducible pair in a non-binary network $N$. 
%If $N$ is a semi-binary network, then 
The algorithm runs in constant time.
\end{lemma}
\begin{proof}
The algorithm takes exactly the steps which define a reduction of a pair in a network, so the algorithm is correct. Then for the running time: comparing the unique parents of the leaves, we can check whether a pair constitutes a cherry or a reticulated cherry in constant time. Indeed, if a pair $(x,y)$ forms a cherry, then the parents of $x$ and $y$ are the same; and if $(x,y)$ forms a reticulated cherry, then the unique other child of the parent of $y$ is the parent of $x$. The removal of the edge with subsequent suppression takes at most constant time. 
Hence, the subroutine runs in constant time. 
\end{proof}

% \begin{algorithm}[H]\label{alg:FIndTCS}
%  \KwData{A TCN $N$}
%  \KwResult{A TCS $S$ for $N$}
% Set $\mathcal{C}=\emptyset$ \;
% \For{$x\in L(N)$}{
%  $\mathcal{C}\cup ${\sc FindRP2nd}$(N,x)$\;
% }

% Let $S$ be an empty sequence\;
% Set $F=\emptyset$ the set of forbidden leaves\;
% \While{$\mathcal{C}\neq\emptyset$}{
%  Choose $(x,y)\in \mathcal{C}$ with $y\not\in F$\;
%  Set $S=S(x,y)$\;
%  Set $F=F\cup\{x\}$\;
%  $N=$ {\sc ReducePair}$(N,(x,y))$\;
%  \uIf{$(x,y)$ is a cherry in $N$}{
%   $\mathcal{C}=\mathcal{C}\setminus\{(x,y),(y,x)\}\cup${\sc FindRP2nd}$(N,y)\cup${\sc FindRC1st}$(N,y)$\;
%  }
%  \If{$(x,y)$ is a reticulated cherry in $N$}{
%   $\mathcal{C}=\mathcal{C}\setminus\{(x,y)\}\cup${\sc FindRC1st}$(N,x)\cup${\sc FindRP2nd}$(N,y)\cup${\sc FindRC1st}$(N,y)$\;
%  }
% }
% \Return $S$\;
% \caption{{\sc FindTCS}$(N)$}
% \end{algorithm}

\begin{algorithm}[H]\label{alg:FIndTCS}
 \KwData{A semi-binary TCN $N$}
 \KwResult{A minimal TCS $S$ for $N$}
Set $\mathcal{C}=\emptyset$ \;
\For{$x\in L(N)$}{
 $\mathcal{C}\cup ${\sc FindRP2nd}$(N,x)$\;
}

Let $S$ be an empty sequence\;
% Set $F=\emptyset$ the set of forbidden leaves\;
\While{$\mathcal{C}\neq\emptyset$}{
 Choose $(x,y)\in \mathcal{C}$\;
%  with $y\not\in F$\;
 Set $S=S(x,y)$\;
%  Set $F=F\cup\{x\}$\;
 $N'=$ {\sc ReducePair}$(N,(x,y))$\;
 \uIf{$(x,y)$ is a cherry in $N$}{
  $\mathcal{C}=\mathcal{C}\setminus\{(x,y),(y,x)\}\cup${\sc FindRP2nd}$(N',y)\cup${\sc FindRC1st}$(N',y)$\;
 }
 \If{$(x,y)$ is a reticulated cherry in $N$}{
  $\mathcal{C}=\mathcal{C}\setminus\{(x,y)\}\cup${\sc FindRP2nd}$(N',y)\cup${\sc FindRC1st}$(N',y)$\;
 }
 $N=N'$\;
}
\Return $S$\;
\caption{{\sc FindTCS}$(N)$}
\end{algorithm}

Recall that TCSs have the additional rule, on top of that of the CPSs, that any leaf appearing as a first coordinate of a pair in the sequence must not appear as a second coordinate of a pair later on in the sequence.
Therefore we may have a cherry~$(x,y)$ in the network that can only be picked as~$(x,y)$ and not~$(y,x)$.
We define a notion of when a reducible pair is \emph{forbidden} before proving the correctness of Algorithm~\ref{alg:FIndTCS}.

\begin{definition}
Given a partial TCS~$S$, a reducible pair~$(x,y)$ is \emph{forbidden} after $S$ if~$y$ has appeared as the first coordinate of a pair in~$S$.
\end{definition}
%\todoYuki{Cite TC-reconstruction paper.}

\begin{lemma}\label{lem:FindTCS}
Let~$N$ be a semi-binary TCN on taxa set~$X$ and reticulation number~$r$.
Algorithm~\ref{alg:FIndTCS} finds a minimal TCS for~$N$ in~$O(|X|+r)$ time.
\end{lemma}
\begin{proof}

The first for loop finds all reducible pairs of $N$, as they all have a unique second coordinate, and each leaf is a second coordinate in at most one pair in a binary network. Now by Proposition~\ref{prop:TCSOrder}, each choice of non-forbidden pair gives a partial TCS for $N$. Hence, the algorithm may choose any non-forbidden pair and reduce it, so suppose that we pick~$(x,y)$.
Regardless of whether~$(x,y)$ is a cherry or a reticulated cherry, the reduction takes constant time as~{\sc ReducePair}$(N,(x,y))$ runs in constant time.
If~$(x,y)$ is a cherry, then we remove the pairs~$(x,y),(y,x)$ from our list of pairs~$\mathcal{C}$. All other pairs of $\mathcal{C}$ are still reducible in $N(x,y)$ by Observation~\ref{obs:ChangingReducibleCherries}, and they are not forbidden because $x$ is not a leaf of $N(x,y)$, so clearly there is no reducible pair with $x$ as second coordinate.
Noting that~$x$ is no longer a leaf in~$N(x,y)$, the only new cherries in~$N(x,y)$ must involve~$y$.
Since~$y$ is not a forbidden leaf, any cherry pair involving~$y$ is not forbidden: therefore we update our list~$\mathcal{C}$ by appending both {\sc FindRP2nd}$(N,y)$ and {\sc FindRC1st}$(N,y)$.

On the other hand if~$(x,y)$ is a reticulated cherry, then the new cherry pairs in~$N(x,y)$ must involve $x$ or $y$,
% \todoRemie{Shit, they might involve $y$ as a first element of the pair!!!}, 
so removing the pairs $(x,y)$ from the current list $\mathcal{C}$ and then by adding the new pairs involving $x$ and $y$, we get the updated list of pairs $\mathcal{C}$ for the TCN~$N(x,y)$; we ensure that this updated list contains only non-forbidden pairs.

Observe that all reducible pairs of~$N(x,y)$ that involve~$x$ are already contained in the set~$\cal{C}$. Indeed, if~$x$ is involved in a cherry or a reticulated cherry~$(x,z)$, then~$(x,z)$ must have formed a reticulated cherry in~$N$, which would have been found in the previous iteration of the while loop or in the initial search for all reducible pairs with second coordinate~$z$.
Note that~$x$ cannot be involved in a reticulated cherry where its parent~$p_x$ is a tree node, as this would imply that~$p_x$ was the parent of two reticulations in~$N$, contradicting the tree-child property of~$N$.
Furthermore, since~$x$ has just been picked as a first element of some reducible pair, all pairs involving~$x$ as a second coordinate are now forbidden.
Therefore, all possible new cherry pairs not in~$\cal{C}$ that appear as a result of reducing~$(x,y)$ from~$N$ must involve the leaf~$y$.
These pairs are made up of those pairs contained in {\sc FindRP2nd}$(N,y)$ or in {\sc FindRC1st}$(N,y)$.
% We note that all possible new cherry pairs involving~$x$ or~$y$ are contained in one of the four sets {\sc FindRP2nd}$(N,x)$, {\sc FindRP2nd}$(N,y)$, {\sc FindRC1st}$(N,x)$, or {\sc FindRC1st}$(N,y)$.
% Now, the cherry contained in {\sc FindRP2nd}$(N,x)$ (if it exists) is forbidden since we have just picked~$x$ as the first element.
% The other three sets all contain cherries that are not forbidden.
The only potential problem that we may face is when~$(y,x) \in ${\sc FindRC1st}$(N,y)$.
However this is impossible; indeed, if~$N(x,y)$ contained a reticulated cherry~$(y,x)$, then it is immediately clear that~$N$ was not a tree-child network.
The grandparent of~$x$ that is not the parent of~$y$ is a tree node parent of two reticulations in~$N$.
Therefore appending the cherries from the two sets {\sc FindRP2nd}$(N,y)$ and {\sc FindRC1st}$(N,y)$ to~$\mathcal{C}$ ensures that there are no forbidden cherry pairs in the updated list~$\mathcal{C}$.
Furthermore, updating~$\cal{C}$ in this way ensures that we only look for each reducible pair once.

In the above cases, we must also make sure that the list~$\mathcal{C}$ is non-empty as long as the network has not been reduced to a single leaf.
But this is immediate by Proposition~\ref{prop:TCSOrder}, and since we add all non-forbidden newly possible reducible pairs to the list.
Therefore, Algorithm~\ref{alg:FIndTCS} is correct; repeating the while loop will eventually completely reduce the network, and give a TCS.
In particular, since all ordered pairs in the sequence reduce a cherry or a reticulated cherry, the output TCS is minimal.

We now compute the running time of the algorithm. Let~$|X|$ and~$r$ denote the number of leaves and the reticulation number of the network, respectively.
Each call of {\sc FindRP2nd} takes constant time, so the for loop takes~$O(|X|)$ time.
Within the while loop, the algorithms {\sc ReducePair}, {\sc FindRP2nd}, and {\sc FindRC1st} are called~$|X|+r-1$ times since the length of a minimal TCS is of length~$|X|+r-1$ by Lemma~\ref{lem:OptimumCPN}.
Each call of both {\sc ReducePair} and {\sc FindRP2nd} run in constant time.
Since every reducible pair is found at most once, Lines 5-7 of {\sc FindRC1st} are called at most~$r$ times in total.
This means that although {\sc FindRC1st} will be called~$|X|+r-1$ times, the part of the algorithm that is dependent on the indegree of the reticulation vertex will only run at most~$r$ times.
It follows then that the while loop runs in time at most~$O(|X|+r)$,
% The while loop is traversed $|X|+r-1$ times since the length of a minimal TCS is~$|X|+r-1$.
% %, where $X$ is the leaf set of $N$ and $k$ the reticulation number of $N$. 
% Within the while loop, the algorithm {\sc ReducePair} takes constant time.
% The algorithm {\sc FindRP2nd} also takes constant time; {\sc FindRC1st} on the other hand takes constant time, given that the indegree of the reticulation nodes is bounded.
% However, since every reducible pair is found exactly once in the algorithm, we have that {\sc FindRC1st} outputs a non-empty set at most~$k$ times.
% This inherently implies that
% each of the instructions within the while loop takes constant time,
 and therefore that Algorithm~\ref{alg:FIndTCS} runs in~$O(|X|+r)$ time.
%\todoRemie{What about choosing a non-forbidden pair? I think that might screw us?}
\end{proof}

% \begin{algorithm}[H]\label{alg:TCSReduceNetwork}
%  \KwData{A CPN $N$ and a CPS $S$}
%  \KwResult{Yes if $S$ reduces $N$, No otherwise.}
% Set $\mathcal{C}=\emptyset$ \;
% \For{$x\in L(N)$}{
%  $\mathcal{C}\cup ${\sc FindRP2nd}$(N,x)$\;
% }

% Let $S$ be an empty sequence\;
% \While{$c=(x,y)\in S$}{
%  \uIf{$(x,y)$ is a cherry in $N$}{
%   $N=${\sc ReducePair}$(N,(x,y))$\;
%   $\mathcal{C}=\mathcal{C}\setminus\{(x,y),(y,x)\}\cup${\sc FindRP2nd}$(N,y)$\;
%  }
%  \If{$(x,y)$ is a reticulated cherry in $N$}{
%   $N=$  {\sc ReducePair}$(N,(x,y))$\;
%   $\mathcal{C}=\mathcal{C}\setminus\{(x,y)\}\cup${\sc FindRP2nd}$(N,x)\cup${\sc FindRP2nd}$(N,y)$\;
%  }
% }
% \If{$\mathcal{C}=\emptyset$ \yuki{and~$N$ is a network on a single leaf}}{
%  \Return Yes\;
% }
% \Return No\;

% \caption{{\sc CPSReducesNetwork}$(N,S)$}
% \end{algorithm}
% \todoYuki{I think you can rewrite alg 5 like this. I've toggled the old alg 5.}
% \todoRemie{I think so, if we cannot get other networks with one leaf. But we don't because those are not CPNs, and reducing stuff in a CPN gives you a CPN. Conclusion: YES}
\begin{algorithm}[H]\label{alg:TCSReduceNetwork}
 \KwData{A CPN $N$ and a CPS $S$}
 \KwResult{Yes if $S$ reduces $N$, No otherwise.}

\For{$i = 1,\dots,|S|$}{
 $N =$ { \sc ReducePair}$(N,S_i)$
}
\If{$N$ is a network on a single leaf}{
 \Return Yes\;
}
\Return No\;

\caption{{\sc CPSReducesNetwork}$(N,S)$}
\end{algorithm}

\begin{lemma}
Let~$N$ be a CPS on taxa set~$X$ and reticulation number~$r$, and let~$S$ be a CPS of length~$|S|$.
Algorithm~\ref{alg:TCSReduceNetwork} correctly checks whether~$S$ reduces~$N$ and runs in~$O(|S|)$ time if $N$ is semi-binary.
\end{lemma}
\begin{proof}
The correctness of the algorithm follows straight from the definition reducing a pair from CPNs.
To compute the running time, note that we iterate over the for loop~$|S|$ times.
The step within the for loop, Algorithm~\ref{alg:ReducePair}, runs in constant time for semi-binary networks.
Hence, it is clear that the algorithm runs in~$O(|S|)$ time.
\end{proof}

\begin{algorithm}[H]\label{alg:Subnetwork}
 \KwData{Two semi-binary TCNs $N$ and $N'$ on the same set of taxa}
 \KwResult{Yes if $N$ contains $N'$, No otherwise.}
Set $S=${\sc FindTCS}$(N)$\;
\Return {\sc CPSReducesNetwork}$(N',S)$\;

\caption{{\sc TCNContains}$(N,N')$}
\end{algorithm}

\begin{theorem}
Given two semi-binary TCNs $N$ and $N'$ on taxa set~$X$ where the reticulation number of~$N$ is~$r$, it can be decided in time~$O(|X|+r)$ whether $N'$ is a subnetwork of $N$.
\end{theorem}
\begin{proof}
We prove the correctness of Algorithm~\ref{alg:Subnetwork} and show that it runs in linear time.
We use Algorithms~\ref{alg:FIndTCS} and~\ref{alg:TCSReduceNetwork}. 
First we find a TCS for $N$ in~$O(|X|+r)$ time with Algorithm~\ref{alg:FIndTCS}. Then, again in~$O(|X|+r)$ time, we see if this TCS reduces $N'$ using Algorithm~\ref{alg:TCSReduceNetwork}. 
If this TCS does indeed reduce~$N'$, then~$N'$ is a subnetwork of~$N$ by Theorem~\ref{the:SubnetworkIffTCSReduces}. 
Hence, combining the two algorithms, the second algorithm outputs yes precisely if $N'$ is a subnetwork of $N$.
Furthermore, the algorithm runs in linear time as the combination of the two algorithms run in $O(|X|+r)$ time.
\end{proof}

% If we let~$d$ denote the maximum indegree of a reticulation node in our network, Algorithm~\ref{alg:Subnetwork} runs in~$O((|X|+r)d)$ time.%\todoYuki{Do we want this? From Leo: I think I mentioned this before, but it would be nice to know what the running time is if you take the maximum indegree into account.}

The theorem has the following corollary regarding {\sc network Isomorphism}, which asks whether two given networks are isomorphic.
The problem for tree-child networks was previously shown to be solvable in~$O(|X|^2)$ time~\citep{cardona2009comparison}.
We present the first linear-time algorithm for checking whether two tree-child networks are isomorphic.%\todoYuki{Since we have this, we don't need the result about graph-isom-completeness of tree-sibling networks?}
% which may be surprising because this problem is Graph-Isomorphism-complete for tree-sibling networks \citep{cardona2014comparison}.  %This result is not new \citep{TCIsomorphismPaper}, but is proven differently here. 

% \todoRemie{Can't find any paper with this result for tree-child...}

\begin{corollary}
Given two semi-binary TCNs $N$ and $N'$ on taxa set~$X$ where the reticulation number of~$N$ is~$r$, it can be decided in~$O(|X|+r)$ time whether $N$ is isomorphic to $N'$.
\end{corollary}

\subsection{Isomorphism results for CPNs}
%\todoYuki{Please read this subsection again. I rewrote the proofs a bit, and separated the algorithm for SBSF and binary.}

We mentioned in Section~\ref{subsec:Distinguishability} that we can distinguish CPNs within reconstructible classes by comparing their smallest CPSs.
Indeed, by altering Algorithm~\ref{alg:FIndTCS}, we can find the smallest CPS for a given CPN in 
%\remie{almost quadratic}\todoRemie{Or should we just say polynomial for simplicity here?}
polynomial time.%, given that the indegree of the reticulation nodes is bounded.
The change is simple: at the start of the while loop, simply choose the smallest cherry instead of choosing any arbitrary cherry.
Furthermore, every time we update the cherry list~$\mathcal{C}$, if we have just reduced a reticulated cherry~$(x,y)$ from~$N$, then we add another set {\sc FindRP2nd}$(N,x)$.%\todoYuki{We can also not include the Algorithm cuz it takes up space.}}

There is another thing to take into consideration.
Tree-child networks are defined to be stack-free; binary CPNs are not necessarily stack-free.
Because of this, upon reducing a reticulated cherry~$(x,y)$ from a binary CPN, it is possible that a new reticulated cherry with~$x$ as the first coordinate is in the reduced network.
This means that when updating the list of reducible pairs ($\mathcal{C}$ in the algorithms), we must also append elements in which~$x$ occurs as the first element, by calling the algorithm~{\sc FindRC1st}$(N,x)$.
Observe that we do not need to do this if the CPN is stack-free, due to the same argument as stated in the proof of Lemma~\ref{lem:FindTCS}.

We start by introducing an algorithm that finds a smallest CPS for a semi-binary stack-free CPN, and then show that simply changing one of the lines in the algorithm allows for an input of binary CPNs.
A slightly more involved change allows for an input of non-binary reconstructible CPNS.
% \todo{Why does it say stack-free in the algorithm? It looks like you fixed it so it works for CPNs with stacks. \textbf{RJ:} Maybe not: for stacks, we need to find more pairs with $x$ as first element after reducing ret cherry $(x,y)$.}

\begin{algorithm}[H]\label{alg:FindSmallestCPS}
 \KwData{A \yuki{semi-binary stack-free} CPN $N$, an ordering on the leaf sets of~$N$}
 \KwResult{A smallest CPS $S$ for $N$}
Set $\mathcal{C}=\emptyset$ \;
\For{$x\in L(N)$}{
 $\mathcal{C}\cup ${\sc FindRP2nd}$(N,x)$\;
}

Let $S$ be an empty sequence\;
% Set $F=\emptyset$ the set of forbidden leaves\;
\While{$\mathcal{C}\neq\emptyset$}{
 Choose $(x,y)\in \mathcal{C}$ such that~$(x,y)$ is smallest\;
%  with $y\not\in F$\;
 Set $S=S(x,y)$\;
%  Set $F=F\cup\{x\}$\;
 $N'=$ {\sc ReducePair}$(N,(x,y))$\;
 \uIf{$(x,y)$ is a cherry in $N$\label{algline:NBexchange}}{
  $\mathcal{C}=\mathcal{C}\setminus\{(x,y),(y,x)\}\cup${\sc FindRP2nd}$(N',y)\cup${\sc FindRC1st}$(N',y)$\;
 }
 \If{$(x,y)$ is a reticulated cherry in $N$}{
  $\mathcal{C}=\mathcal{C}\setminus\{(x,y)\}\cup${\sc FindRP2nd}$(N',x)\cup${\sc FindRP2nd}$(N',y)\cup${\sc FindRC1st}$(N',y)$\; \label{algline:SBBexchange}
 }
}
\Return $S$\;
\caption{{\sc FindCPS}$(N)$}
\end{algorithm}

\begin{lemma}\label{lem:FindSBCPS}
Let~$N$ be a \yuki{semi-binary stack-free} CPN on taxa set~$X$ with reticulation number~$r$.
Algorithm~\ref{alg:FindSmallestCPS} finds a smallest CPS for~$N$ in~$O((|X|+r)|X|\log(|X|))$ time, if an ordering is given on~$X$.
\end{lemma}

\begin{proof}

Initially, we find the full set of reducible pairs by noting that a leaf appears as a second coordinate in a reducible pair at most once (the first for loop).
Within the while loop, we always pick a smallest reducible pair.
After picking a reducible pair~$(x,y)$, we update the list of all possible reducible pairs by adhering to Lemma~\ref{lem:PossibleCherriesAfterReduction}: all new reducible pairs must contain one of the leaves~$x$ or~$y$.
As was the case for Algorithm~\ref{alg:FIndTCS}, there is no need to include the set of reducible pairs~$\textsc{FindRC1st}(N,x)$ as these pairs would have been found in one of the previous iterations.
This ensures that all possible reducible pairs are put into the set~$\mathcal{C}$.
Therefore, the algorithm iteratively picks the smallest reducible pair from a list of all possible reducible pairs---hence, it returns a smallest CPS of the input CPN (minimal in particular, since all the elements of the sequence reduces a cherry or a reticulated cherry in the CPN).

To compute the running time, observe first that the for loop takes~$O(|X|)$ time.
Within the while loop, we first choose a reducible pair that is smallest.
This takes~$O(|\mathcal{C}|\log(|\mathcal{C}|))$ time by running some sorting algorithm, and since any leaf can appear as a second coordinate of at most one reducible pair, we have that~$|\mathcal{C}| = O(|X|)$.
Therefore, choosing a smallest reducible pair takes at most~$O(|X|\log(|X|))$ time.
By the same argument used in the proof of Lemma~\ref{lem:FindTCS}, we have that all other steps within the while loop takes~$O(|X|+r)$ time.
% The while loop is still iterated~$|X| + k -1$ times as seen in the proof of Lemma~\ref{lem:FindTCS}.
% Within the while loop, choosing a smallest reducible pair takes~$O(|\mathcal{C}|\log(|\mathcal{C}|))$ time.
% Since any leaf can appear as a second coordinate of at most one reducible pair, we have that~$|\mathcal{C}| \leq |X|$.
%Then choosing a smallest CPS takes~$O(|X|\log(|X|))$ time. 
%The steps with {\sc FindRC1st} take $O(d)$ time, where $d$ is the maximal indegree of the reticulation nodes. As $d<k+1$, this adds a factor of $O(k)$ within the loop.
Therefore the algorithm runs in~$O((|X|+r)|X|\log(|X|))$ time.
\end{proof}

\yuki{For the algorithm to work on binary CPNs, simply exchange Line~\ref{algline:SBBexchange} of Algorithm~\ref{alg:FindSmallestCPS} by the following:}
\[  \mathcal{C}=\mathcal{C}\setminus\{(x,y)\}\cup \textsc{FindRP2nd}(N,x)\cup\textsc{ FindRC1st}(N,x)\cup\textsc{FindRP2nd}(N,y)\cup\textsc{FindRC1st}(N,y);\]

\begin{lemma}
Let~$N$ be a \yuki{binary} CPN on taxa set~$X$ with reticulation number~$r$.
Algorithm~\ref{alg:FindSmallestCPS} with the modification to Line~\ref{algline:SBBexchange} finds a smallest CPS for~$N$ in~$O((|X|+r)|X|\log(|X|))$ time, if an ordering is given on~$X$.
\end{lemma}
\begin{proof}
The correctness of the algorithm follows from the proof of Lemma~\ref{lem:FindSBCPS}, with the addition that upon reducing the network by a pair~$(x,y)$, the newly introduced reducible pairs may have~$x$ as a first coordinate (which appear due to stacks).
We resolve this by appending~$\textsc{FindRC1st}(x,y)$ to the list of possible reducible pairs (Line~\ref{algline:SBBexchange}).

For the running time analysis, observe first that the for loop runs in~$O(|X|)$ time.
The while loop is iterated~$O(|X|+r)$ times, as a minimal CPS for a network on~$|X|$ leaves and reticulation number~$r$ has length~$|X|+r-1$ by Lemma~\ref{lem:OptimumCPN}.
As shown in the proof of Lemma~\ref{lem:FindSBCPS}, the running time of finding the smallest reducible pair from~$\mathcal{C}$ takes~$O(|X|\log(|X|))$ time.
All other steps within the while loop take constant time: in particular,~$\textsc{FindRC1st}$ takes constant time as the indegree of the reticulation vertices is bounded in a binary network. 
Hence, the while loop runs in~$O((|X|+r)|X|\log(|X|))$ time, so the algorithm runs in~$O((|X|+r)|X|\log(|X|))$ time.
\end{proof}

In case of non-binary reconstructible CPNs, the operations require a slight care.
This precaution stems from two reasons. Firstly, tree vertices are now allowed to have
outdegree greater than~$2$. This means that every leaf no longer appears at most once as
a second coordinate of a reducible pair: they can be a second coordinate of a pair multiple
times. Therefore Algorithm~$\textsc{FindRP2nd}$ may now return a set containing more than one
element. This poses a potential problem in the original Algorithm~\ref{alg:FindSmallestCPS},
since upon reducing the network, we update the list of reducible pairs by adding
either~$\textsc{FindRP2nd}(N',x)$ or~$\textsc{FindRP2nd}(N',y)$ (or both).
Indeed, this could mean that we could obtain the same reducible pairs multiple times if 
there exist cherries or reticulated cherries with a multifurcation. To avoid this, we only invoke
these updates whenever the outdegree of the parent of~$x$ and~$y$ (in either the original
or the reduced network) is~$2$. This helps avoid double or triple counting of the same 
reducible pairs.
Secondly, as seen in Observation~\ref{obs:ChangingReducibleCherries} reducible 
pairs of the original network may not be reducible pairs in a reduced network, although the pair
itself was not reduced. For example if a network contains cherries~$(x,y)$ and~$(x,z)$,
then upon reducing by the pair~$(x,y)$, the new network no longer contains the 
cherry~$(x,z)$, since~$x$ is no longer a leaf in the network. Due to this, we must
replace the lines~\ref{algline:NBexchange} --~\ref{algline:SBBexchange} of 
Algorithm~\ref{alg:FindSmallestCPS} by the following.

\renewcommand{\thealgocf}{}

\begin{algorithm}[H]
\setcounter{AlgoLine}{9}
  \uIf{$(x,y)$ is a cherry in $N$}{
    \For{$l\in L(N)$}{
        \uIf{$(x,l)$ is a cherry in~$N$}{
            $\mathcal{C}=\mathcal{C}\setminus\{(x,l),(l,x)\}$\;
        }
    }
    $\mathcal{C}=\mathcal{C}\cup${\sc FindRC1st}$(N',y)$\;
    Let~$p_y$ denote the parent of~$y$ in~$N$\;
    \uIf{$outdegree(p_y) = 2$}{
        $\mathcal{C}=\mathcal{C}${\sc FindRP2nd}$(N',y)$\;
    }
  }
 \If{$(x,y)$ is a reticulated cherry in $N$}{
    \For{$l\in L(N)$}{
        \uIf{$(y,l)$ is a cherry in~$N$}{
            $\mathcal{C}=\mathcal{C}\setminus\{(x,l)\}$\;
        }
    }
    $\mathcal{C}=\mathcal{C}\setminus\{(x,y)\}\cup${\sc FindRC1st}$(N',y)$\;
    Let~$p_x$ denote the parent of~$x$ in~$N$\;
    Let~$p_y$ denote the parent of~$y$ in~$N$\;
    \uIf{$indegree(p_x)=2$}{
        $\mathcal{C} = \mathcal{C}\cup${\sc FindRC1st}$(N',x)\cup${\sc FindRP2nd}$(N',x)$\;
    }
    \uIf{$outdegree(p_y)=2$}{
        $\mathcal{C} = \mathcal{C}\cup${\sc FindRP2nd}$(N',y)$\;
    }
  }
\end{algorithm}

\begin{lemma}
Let~$N$ be a \yuki{non-binary} reconstructible CPN on taxa set~$X$ with reticulation number~$r$.
Algorithm~\ref{alg:FindSmallestCPS} with the modification to Lines~\ref{algline:NBexchange} --~\ref{algline:SBBexchange} finds a smallest CPS for~$N$ in~$O((|X|+r)|X|^2\log(|X|^2)$ time, if an ordering is given on~$X$.
\end{lemma}
\begin{proof}
The correctness of the algorithm follows from the proof of Lemma~\ref{lem:FindSBCPS}, with the changes outlined as above.
As before, the first for loop finds all reducible pairs within the networks.
Upon reducing the smallest reducible pair, we update the list of reducible pairs by using the above replacement of the algorithm.
Suppose that we have just reduced by a reticulated cherry~$(x,y)$.
Let~$N$ denote the original network and~$N'=N(x,y)$, as in the algorithm.
As mentioned before, there may be reducible pairs in~$N$ that are no longer reducible pairs in~$N'$.
These are all reticulated cherries of the form~$(x,z)$, for which the reticulation edge was between the parent of~$y$ and the parent of~$x$.
Then,~$y$ and~$z$ must have formed a cherry in~$N$---therefore we check for all such cherries (line 20), and remove the corresponding reducible pairs from the list.
When updating the list, all new reducible pairs involve either the leaf~$x$ or~$y$.
To avoid seeking for pairs that are already contained in the list, we only invoke the~\textsc{FindRP2nd} algorithm whenever new pairs arise in~$N'$, that is, whenever the reticulation parent of~$x$ is of indegree-2 or when the tree vertex parent of~$y$ is of outdegree-2 in~$N$.
This holds similarly for updating the list upon reducing cherries.
Therefore, we find each reducible pair exactly once in the algorithm.

For the running time, observe that the for loop takes at most~$O(|X|^2)$ time, as there can be at most~$\binom{|X|}{2}$ possible reducible pairs in a network.
Therefore~$|\mathcal{C}| = O(|X|^2)$.
The while loop is iterated~$O(|X|+r)$ times, as a minimal CPS for a network on~$|X|$ leaves and reticulation number $r$ has length~$|X|+r-1$ by Lemma~\ref{lem:OptimumCPN}.
Each call of \textsc{ReducePair} runs in constant time.
The running time for finding the smallest reducible pair from~$|\cal{C}|$ takes~$O(|X|^2\log(|X|^2))$ time.
Upon reducing~$(x,y)$ from a network~$N$, we enter a for loop that checks whether there are reducible pairs in the list that are no longer reducible pairs in the network.
This step takes~$O(|X|)$ time.
Since each reducible pair is found at most once, the part of the algorithms \textsc{FindRP2nd} and \textsc{FindRC1st} that has a dependency on the outdegree and the indegree, respectively will only run at most~$O(\binom{|X|}{2}) = O(|X|^2)$ times.
Therefore the steps within the while loop runs in time at most~$O(|X|^2\log(|X|^2)$ time, and therefore Algorithm~\ref{alg:FindSmallestCPS} with the proposed changes runs in~$O((|X|+r)|X|^2\log(|X|^2)$ time.
\end{proof}

\begin{theorem}\label{thm:CPNIsomPolyTime}
Suppose we are given an ordering on the taxa set~$X$. Let $N$ and $N'$ be two networks with $r$ reticulations within the same reconstructible class.
Then it can be decided in~$O((|X|+r)|X|^2\log(|X|^2)$ time whether they are isomorphic.
In particular if the classes are~$(\ref{Const:CherryResolved}, \ref{Const:RCherryResolved})$ (binary) or~$\ref{Const:CherryResolved}, \ref{Const:RCherrySF})$ (semi-binary stack-free), then this can be decided in~$O((|X|+r)|X|\log(|X|)$ time.
% , both from the class of binary networks, or the class of semi-binary, stack-free networks. Then it can be decided in~$O((|X|+r)|X|\log(|X|))$ time %\todoLeo{again, do we need that the indegree are bounded? can we leave this out if we let k be the number of reticulation edges? Yuki: I think it would be weird if we considered ret edge number here, and reticulation number above, so imo we should leave it like this.} 
% whether they are isomorphic.
\end{theorem}
\begin{proof}
Use Algorithm~\ref{alg:FindSmallestCPS} twice to find the smallest CPSs for the two CPNs. This can be done in~$O((|X|+r)|X|^2\log(|X|^2))$ time (in~$O((|X|+r)|X|^2\log(|X|^2)$ for the classes~$(\ref{Const:CherryResolved}, \ref{Const:RCherryResolved})$ (binary) or~$\ref{Const:CherryResolved}, \ref{Const:RCherrySF})$ (semi-binary stack-free)).
By Corollary~\ref{cor:CPNIsomorphism}, these CPSs are the same if and only if the CPNs are isomorphic.
\end{proof}

\section{Implementation}\label{sec:Implementation}
Algorithm~\ref{alg:Subnetwork}, which checks whether a given tree-child network is a subnetwork of another, 
%presented in Section~\ref{sec:Computation} 
was implemented in Python to test the theoretical linear bound in practice. In this section, we present running time results of the implementation on a large randomly generated data set. We show that the theoretically proven linear running time is indeed achievable in practice. The tests were run on a Linux system with a quad-core Intel Xeon W3570 running at 1.7GHz and 24GB of DDR3 RAM clocked at 1333MHz. The operating system was Debian GNU/Linux 9 with a 4.19.46-64 Linux kernel. The software was written in Python version 3.7.3.

\subsection{Generating the datasets}
For the test data, we generated 131200 instances of the tree-child network containment problem: two yes-instances and two no-instances for all $n,r,r'\in\{25,50,\ldots,975,1000\}$ with $r'\leq r$, where $n$ is the number of leaves of both networks, $r$ is the reticulation number in the first network, and $r'$ the reticulation number in the second network.
Each instance consists of two semi-binary tree-child networks on the same leaf-set, for which we asked whether the first network contained the second network. 

For each instance, we generated the first network with $n$ leaves and reticulation number $r$ using Algorithm~\ref{alg:RandomTCS}. 
The second network was generated depending on whether it was a yes- or a no-instance.
If it was a yes-instance, a subnetwork with reticulation number $r'$ was obtained using Algorithm~\ref{alg:RandomSubTCS}; for a no-instance, a network on the same number of leaves and reticulation number $r'$ was randomly generated with the same process as the first network (using Algorithm~\ref{alg:RandomTCS}).

This way, each generated yes-instance is always a yes-instance for the {\sc Network Containment} problem. For the no-instances, however, the random generation of the second network could also give a subnetwork of the first network, but the probability of that event is very small, as the number of tree-child networks grows very quickly with the number of leaves and reticulations \citep{cardona2019generation}.
%Binary tree child networks on n leaves asymptotically: $2^{2n\log(n+O(n))}$

The dataset used for the experiment along with the code for generating random datasets, and the actual implementation of Algorithm~\ref{alg:Subnetwork} can be found on \url{https://github.com/RemieJanssen/Cherry-picking_TC_Network_Containment}.

\subsubsection{Generating random networks}\label{sec:GeneratingNetworks}
The tree-child networks were randomly generated as TCSs using Algorithm~\ref{alg:RandomTCS}. This algorithm takes two positive integers $n$ and $r$, and outputs a tree-child network with $n$ leaves and reticulation number $r$. It starts with the cherry $(1,2)$, and successively adds leaves as cherries, and reticulated cherries between two leaves that already exist in the network (respecting the tree-child condition). 

In the algorithm, this is achieved by building a tree-child sequence backwards. It chooses to add a reducible pair corresponding to a cherry or reticulated cherry uniformly at random until we have added the required number of leaves and reticulation number. To make sure the sequence is a tree-child sequence, we keep a list $NF$ of taxa that are `non-forbidden', which, in this case, means that the taxon is not currently the child of a tree node that has a reticulation as the other child (i.e., the leaf has not appeared as a second coordinate element of a reducible pair). If a taxon is in $NF$, it is possible to take this taxon as the first element of a pair appended at the start of the sequence. As a tree-child network always has a cherry or a reticulated cherry, $NF$ is never empty. This implies that the algorithm should never output False, but lines~$15$ and~$16$ are kept so that the algorithm can easily be adapted to return only binary tree-child networks. To achieve this, one only has to add the line ``$NF = NF\setminus\{\mathrm{first\_element}\}$'' between lines 21 and 22 in the algorithm.

Finally, the algorithm outputs a TCS, from which we can uniquely construct a TCN.

\begin{algorithm}[H]\label{alg:RandomTCS}
 \KwData{A set of taxa $X=\{1,\ldots,n\}$, and a reticulation number $r$.}
 \KwResult{A TCS $S$ on $X$ of weight $r$.}
 
 Initialize $Y = \{1,2\}$ the current set of taxa\;
 Initialize $S = (2,1)$ the current sequence\;
 Initialize $L = n-2$\;
 Initialize $R = r$\;
 Initialize $NF = \{2\}$\;
 \While{$L>0$ or $R>0$}
 {
    type\_added = None\;
    \uIf{$|NF|>0$ and $L>0$ and $R>0$}
    {
        With probability $\frac{L}{L+R}$, type\_added = L\;
        Otherwise, type\_added = R\;
    }
    \uElseIf{$|NF|>0$ and $R>0$}
    {
        type\_added = R\;
    }
    \uElseIf{$L>0$}
    {
        type\_added = L\;
    }
    \Else
    {
        \Return False\;
    }
    
    first\_element = None\;
    second\_element = None\;
    \uIf{type\_added = R}
    {
       Set first\_element an element of $NF$ chosen uniformly at random\;
       Set $R = R-1$\;
    }
    \Else
    {
       Set first\_element to the first element of $X\setminus Y$\;
       Set $L = L-1$\;
       Set $Y = Y\cup \{\text{first\_element}\}$\;
       Set $NF = NF\cup\{\text{first\_element}\}$\;
    }
    Set second\_element to an element of $Y\setminus\{\text{first\_element}\}$ chosen uniformly at random\;
    $NF = NF\setminus\{\text{second\_element}\}$\;
    $S = (\text{first\_element},\text{second\_element})S$\;
 }
\Return $S$\;
\caption{{\sc RandomTCS}$(X,r)$}
\end{algorithm}

Note that each tree-child network has positive probability of appearing for this process. In fact, each tree-child sequence ending with $(1,2)$ has positive probability.
\medskip

Let us now turn to the procedure to generate a tree-child subnetwork (i.e., generating the second network in a yes-instance). 
For this purpose, we again work with the representation of the networks as tree-child sequences. 

We first select ordered pairs from the sequence of the first network, such that the resulting subsequence corresponds to a tree. This is simply done by selecting a pair with first element $x$ for all $x\in X$ uniformly at random. Because the sequence we started with is a tree-child sequence, the subsequence consisting of the chosen pairs is a tree-child sequence as well: suppose $(x,y)$ and $(y,z)$ are selected. Then $(y,z)$ must appear after $(x,y)$, because otherwise $y$ appears as a first element after it has appeared as a second element in the original sequence. 

After selecting the pairs that form a \emph{base tree} of the network (a spanning tree of the network with leaf-set~$X$), we select $r'$ additional pairs that will form the $r'$ reticulations of the subnetwork. By a similar argument as for the base tree, this subsequence is a tree-child sequence. And as it is reduced by the subsequence, it is also reduced by the sequence of the original network. Hence, the network corresponding to the chosen pairs is a tree-child subnetwork of the original network.

\begin{algorithm}[H]\label{alg:RandomSubTCS}
 \KwData{A TCS $S$ on $X$ of length $|X|+r-1$, and a number $r'\leq r$.}
 \KwResult{A sub-TCS $S'$ of $S$ on $X$ of length $|X|+r'-1$}
    Let $S$ be indexed by $\{1,\ldots,|S|\}$\;
    Set $I_{S'}=\emptyset$\;
    \For{$x\in X$}{
        Let $I_x$ be the set of indices of pairs of $S$ with $x$ as first element\;
        Pick $i_x$ uniformly at random from $I_x$\;
        Set $I_{S'} = I_{S'}\cup\{i_x\}$\;
    }
    Randomly add $r'$ elements from $\{1,\ldots,|S|\}\setminus I_{S'}$ to $I_{S'}$\;
    Let $S'$ be the subsequence of $S$ consisting of the elements indexed by $I_{S'}$\;
    \Return $S'$\;
\caption{{\sc RandomSubTCS}$(S,r')$}
\end{algorithm}

\subsection{Results}
For all yes-instance tests in which the second network was a subnetwork of the first (i.e., the ones generated by Algorithm~\ref{alg:RandomSubTCS}), the algorithm correctly returned a true value. Similarly, for all no-instance tests in which the second network was generated randomly and independently from the first network, the algorithm correctly found that the second network was not a subnetwork of the first. This means that, even though there was a non-zero probability that the second network was a subnetwork, this did not happen in any of the instances.
We expected this, as the probability of this happening is extremely small.

Note that the largest test instances (1000 leaves, 1000 reticulations) had a running time of approximately $0.1$s. This is expected to scale well for even larger instances, as the linear fit of the data is very good. The $R^2$ values for the fits and the linear dependence of the running time on the number of leaves and reticulations can be found in Table~\ref{table:RunningTimes}. For this fit, we performed a standard linear regression with an intercept of 0 (i.e., forced through the origin), which makes sense because the running time should be zero for an empty instance.

Note that the fits become much better when we split the data in instances where the second network is or is not a subnetwork of the first (i.e., between the yes- and the no- instances), even though the dependence of the running time on the parameters does not change much after this split. The most striking difference we can see in this analysis, is the dependence on the reticulation number of the second network.

\begin{figure}
    \centering
    \includegraphics[width=.49\textwidth]{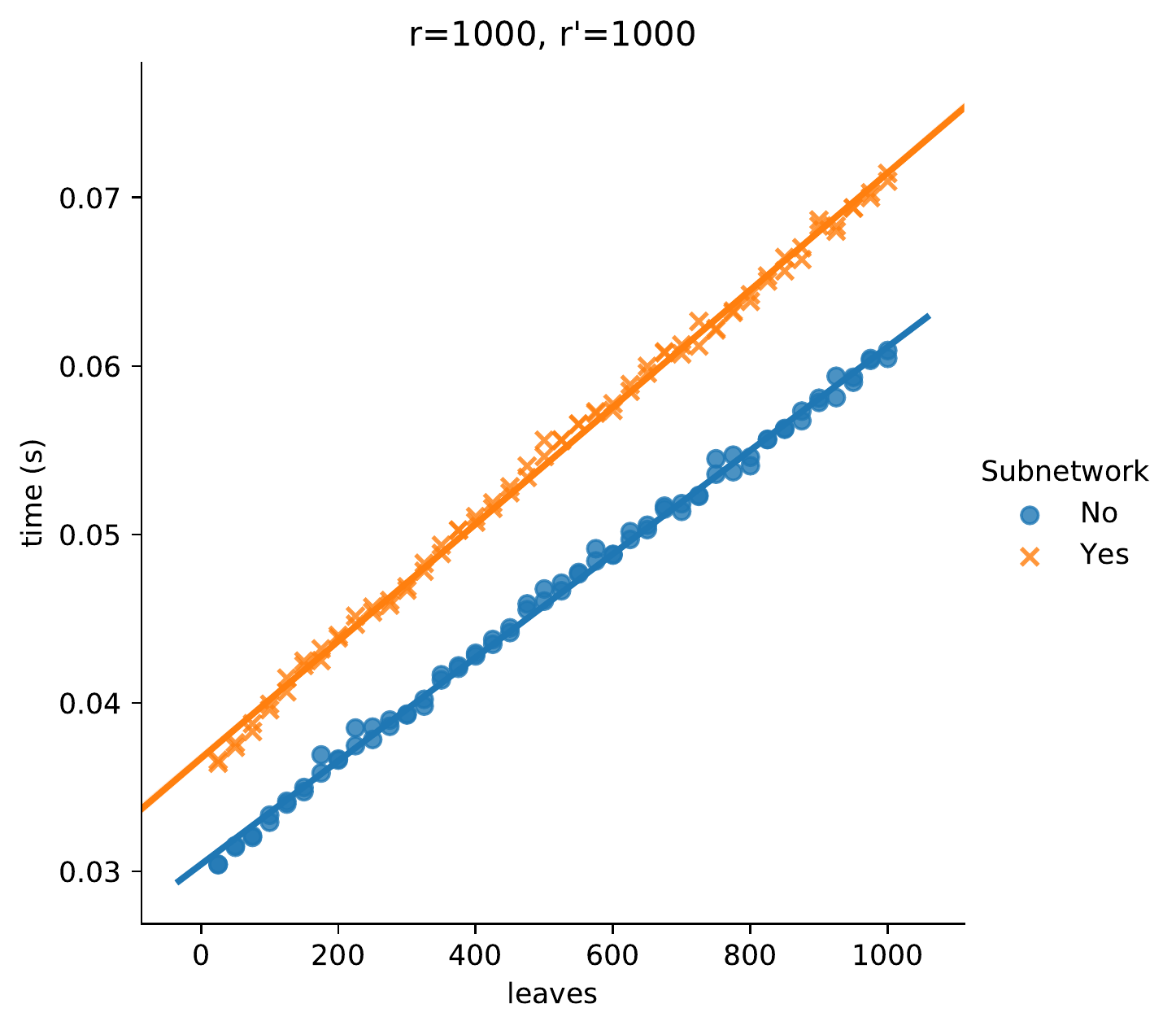}
    \includegraphics[width=.49\textwidth]{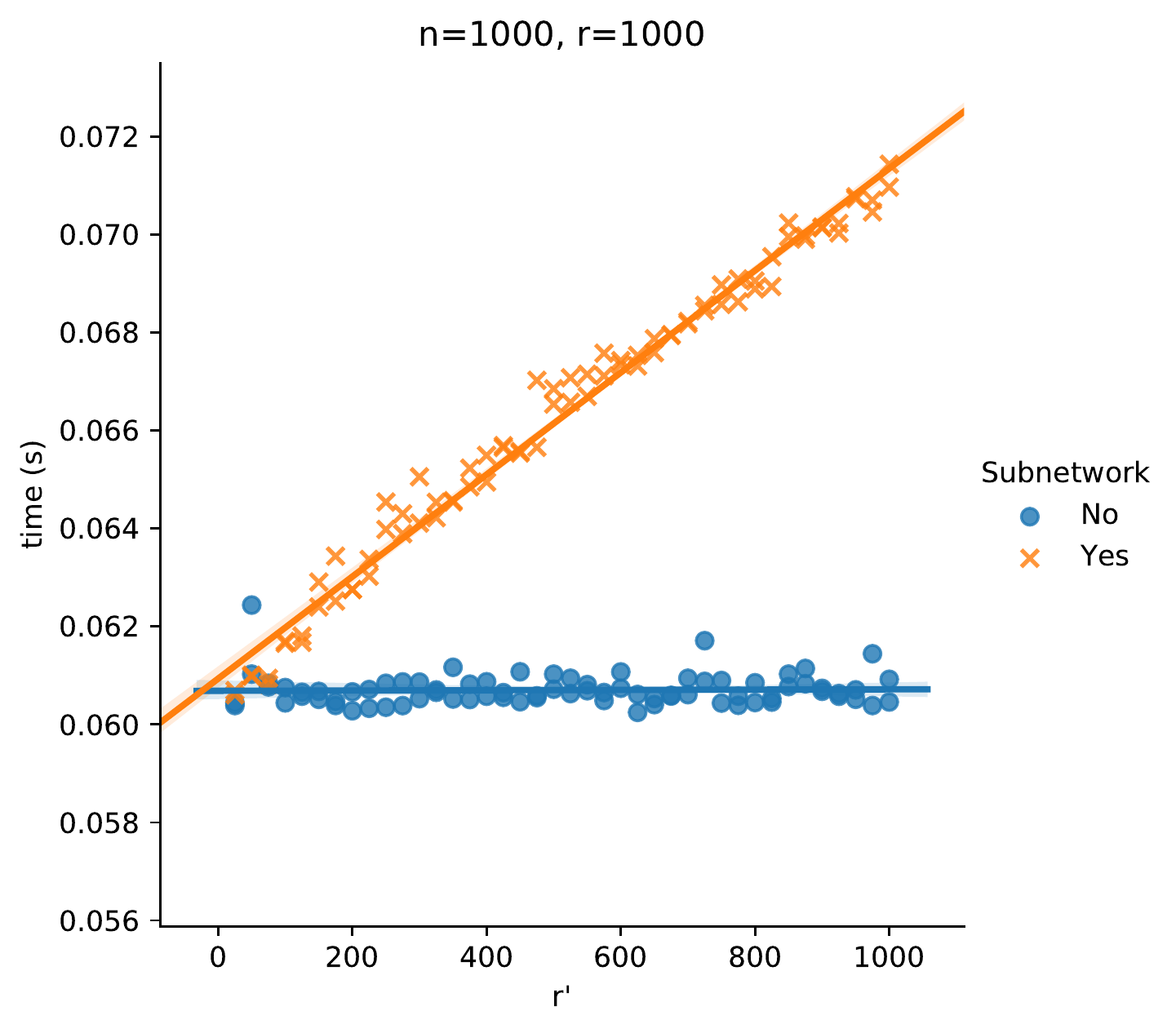}
    \caption{The dependence of the running time on the number of leaves $n$ (left) and the number of reticulations in the second network $r'$ (right). This was visualized by fixing the other parameters to a set value of 1000 in both plots.}
    \label{fig:DataYesNo}
\end{figure}

As shown in Figure~\ref{fig:DataYesNo}, the no-instances were consistently, and marginally, faster than the yes-instances.
For varying 

For varying leaf numbers, the instances where the second network was not a subnetwork (no-instances), were consistently, but marginally, faster than when the second network was a subnetwork (yes-instances) (Figure~\ref{fig:DataYesNo},~Left).
This was similarly true for when we varied the reticulation number~$r'$ of the second network.
The effect of varying~$r'$ on instances for when the second network was not a subnetwork (no-instances) was negligible.
This can be seen in the right figure of Figure~\ref{fig:DataYesNo}, but also in Table~\ref{table:RunningTimes}, where the order of the slope of~$r'$ for the no-instances is far smaller than all other slopes in all the instances.
For the yes-instances, the running time of the algorithm showed a linear dependence on~$r'$, which was in the same order as the other parameters.
% For varying reticulation number on the second network, when the second network was not a subnetwork, the effect of the reticulation number of the second network $r'$ on the running time was negligable, but when it is a subnetwork, the effect was linear in the same order as the other parameters (Figure~\ref{fig:DataYesNo},~Right). 
This can be explained as follows. When the second network is not a subnetwork, the second part of the algorithm (Algorithm~\ref{alg:TCSReduceNetwork}) will seldom need to reduce a pair: it will check whether the pair in the sequence is reducible in the second network. As the second network is randomly generated independently of the first network, it will not have many pairs in common with the first network, which means it will not have to reduce pairs often.

\begin{table}
\caption{Linear regression analysis for tree-child network containment on 131200 instances, for which half were yes-instances and the other half no-instances. The high~$R^2$ value indicates that the fit of the curve is essentially linear (where an~$R^2$ value of~$1$ indicates a perfect linear fit) and the slopes indicate the change in running time for every increase in the number of leaves, reticulation number~$r$, and reticulation number~$r'$.}
\label{table:RunningTimes}
\begin{tabular}{l|r|r|r|r|}
\cline{2-5}
                                       & \multicolumn{1}{c|}{\multirow{2}{*}{$R^2$}} & \multicolumn{3}{c|}{slope}                                                                                                      \\ \cline{3-5} 
                                       & \multicolumn{1}{c|}{}                                              & \multicolumn{1}{c|}{leaves (s/leaf)} & \multicolumn{1}{c|}{$r$ (s/reticulation)} & \multicolumn{1}{c|}{$r'$ (s/reticulation)} \\ \hline
\multicolumn{1}{|l|}{all data}        & $0.9725659$                                                 & $3.03328079\cdot10^{-5}$             & $2.99713310\cdot10^{-5}$                              & $4.75681146\cdot10^{-6}$                              \\ \hline
\multicolumn{1}{|l|}{subnetwork: YES} & $0.9966596$                                                  & $3.14405850\cdot10^{-5}$                       & $2.89850496\cdot10^{-5}$                              & $9.54505907\cdot10^{-6}$                              \\ \hline
\multicolumn{1}{|l|}{subnetwork: NO}  & $0.9976078$                                                 & $2.92250310\cdot10^{-5}$                       & $3.09576119\cdot10^{-5}$                              & $-3.14361016\cdot10^{-8}$                             \\ \hline
\end{tabular}
\end{table}

\section{Discussion}\label{sec:Discussion}
In this paper, we have looked at cherry-picking sequences and how they can be used to solve the {\sc Network Containment} problem for tree-child networks. 
We have introduced the class of cherry-picking networks and showed that, within reconstructible classes, they are uniquely determined by a smallest sequence that reduces them (given an ordering on the leaves).
This correspondence can be used to check whether two CPNs in the same reconstructible class are isomorphic in polynomial time.
Moreover, we showed that if a sequence for a network reduces another network, then the second network is contained in the first network.
For tree-child networks, we further showed that the converse was also true, that a tree-child network contained in another tree-child network will be reduced by any tree-child sequence of the second network, \yuki{given that the networks are in the same reconstructible class.}
%\todoRemie{TODO after making the table: But only if they are in the same reconstructible CPN class? YM: Yes, added. It works for general non-binary too but then we only get one direction (subnet ==> redux).}
%We showed that TCSs can be used to characterize network containment in a similar way as it characterizes tree containment.
On the computational side, we have shown that our Python implementation for solving {\sc Network Containment} for tree-child networks runs in linear time in the number of leaves and the reticulation number.

An apparent limitation of the CPS approach is that they cannot be used to characterize every phylogenetic network, but only CPNs.
The reason for this stems from the inability to pick a \emph{double-reticulated cherry}.
Two leaves~$x$ and~$y$ form a double-reticulated cherry if both parents~$p_x,p_y$ of~$x,y$ respectively are reticulations, and~$p_x$ and~$p_y$ share a common parent.
As seen in other literature, the double-reticulated cherry poses a challenge when trying to correctly add back a once removed reticulation edge~\citep{bordewich2016determining, murakami2019reconstructing}.
A CPS may contain a double-reticulated cherry, as long as there is an `escape' at one of the sides of the double-reticulated cherry: a reticulated cherry which can be reduced, and whose reduction turns the double-reticulated cherry into a regular reticulated cherry.
With no escapes however, there is currently no way of dealing with double-reticulated cherries.

One way of resolving this issue would be to consider leaf attachments as done by \citet{linz2019attaching}.
Given any phylogenetic network, add a minimal number of leaves to make it a CPN.
Recall that every reconstructible CPN has a unique smallest CPS (Theorem~\ref{thm:UniqueCPNConstruction}).
Is it possible to extend this result to general networks by attaching leaves?
% That is, we wish to extend our cherry-picking characterization to general networks by simply adding leaves.
If so, will the placement of the leaves affect the uniqueness of its smallest CPS?
Furthermore, what is the minimal number of additional leaves to turn a general network to one that is orchard?
To answer these questions, one could study the forbidden structures of an orchard network; currently, very little is known about these.
% Extend sequence characterization to general networks.
% How does the placement of the leaves affect the uniqueness of its smallest CPS?

% While these are interesting questions, adding leaves to a network to make it a CPN is, with the techniques presented in this paper, more complicated than it has to be.\todoRemie{What do you mean, more complicated than it has to be?}
% Since we may pick cherries in any order, we initially reduce the network as much as possible.
% Then, upon reaching a network (on at least two leaves) that can no longer be reduced, we add a leaf and continue.
% Such a process will indeed find the minimum number of leaves to add to a network to make it a CPN.\todoRemie{Why does it not depend on where you add a leaf each time you are stuck? \textbf{Add as conjecture.}}
% However, a more direct approach is favourable, one that incorporates the forbidden structures of a CPN.
% Unfortunately, not all forbidden structures of CPNs are known at this point.
Obviously, crowns---undirected cycles within a network whose bipartition into tree nodes and reticulation nodes makes it a bipartite subgraph---cannot be contained in CPNs.
Aside from this, we are not aware of other forbidden structures of CPNs.
In any case, we may gauge where the class of CPNs resides with respect to other prominent network classes.
The class of CPNs does not contain the class of tree-based networks~\citep{francis2015phylogenetic}, since tree-based networks may contain crowns (see Figure~\ref{fig:ObutnotTB}).
%We have seen earlier that double-reticulated cherries may be contained in CPNs. 
However, as there exists a CPS (in fact, a TCS) for every TCN, we conclude that class of CPNs strictly contains the class of TCNs.

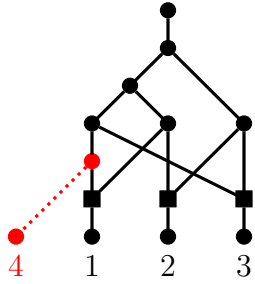
\begin{figure}
    \centering
    \begin{tikzpicture}[every node/.style = {draw, circle, fill, inner sep = 0pt, minimum size = 2mm},
    square/.style = {regular polygon, regular polygon sides = 4, minimum size = 3 mm}]
    \tikzset {edge/.style = {very thick, shorten >= -0.5 pt}}

        %Nodes
        \node[] (-1) at (0.0,-0.0) {};
        \node[] (0) at (0.0,-0.5) {};
        \node[] (4) at (-0.5,-1.0) {};
        \node[] (7) at (1,-1.5) {};
        \node[] (5) at (-1,-1.5) {};
        \node[] (6) at (0,-1.5) {};
        \node[square] (8) at (-1.0,-2.5) {};
        \node[square] (10) at (1.0,-2.5) {};
        \node[square] (9) at (0.0,-2.5) {};
        \node[red] (11) at (-1,-2) {};

        %Leaves
        \node[] (1) at (-1.0,-3) {};
        \node[draw=none, fill=none, below=1mm of 1] (leaf_1){\large $1$};
        \node[] (2) at (0.0,-3) {};
        \node[draw=none, fill=none, below=1mm of 2] (leaf_2){\large $2$};
        \node[] (3) at (1.0,-3) {};
        \node[draw=none, fill=none, below=1mm of 3] (leaf_3){\large $3$};
        \node[red] (12) at (-2.0,-3) {};
        \node[draw=none, fill=none, below=1mm of 12, red] (leaf_4){\large $4$};

        %Edges
        \draw[edge] (-1) edge (0);
        \draw[edge] (0) edge (4);
        \draw[edge] (0) edge (7);
        \draw[edge] (4) edge (5);
        \draw[edge] (4) edge (6);
        \draw[edge] (7) edge (9);
        \draw[edge] (7) edge (10);
        \draw[edge] (5) edge (11);
        \draw[edge] (11) edge (8);
        \draw[edge] (5) edge (10);
        \draw[edge] (6) edge (8);
        \draw[edge] (6) edge (9);
        \draw[edge] (8) edge (1);
        \draw[edge] (10) edge (3);
        \draw[edge] (9) edge (2);
        \draw[edge, dotted, red] (11) edge (12);
    \end{tikzpicture}
    \caption{A tree-based network on leaves~$\{1,2,3\}$ that is not orchard, since it contains a crown (the subgraph induced by the parents and the grandparents of the leaves~$1,2,3)$ without the red vertex). The network on~$\{1,2,3,4\}$ is orchard, since it provides a reticulated cherry~$(1,4)$, an `escape' for the crown.}
    \label{fig:ObutnotTB}
\end{figure}

We briefly comment on the correspondence between containment and reduction, and pose conjectures for remaining open questions.
Let~$N$ and~$N'$ be CPNs in the same reconstructible class on the same leaf-sets.
If a CPS~$S$ of~$N$ reduces~$N'$, then~$N'$ is contained in~$N$ (Lemma~\ref{lem:reductionImpliesContainmentSF}).
We have shown that the converse of this does not hold for the reconstructible classes~$(\ref{Const:CherryResolved}, \ref{Const:RCherryResolved})$ and~$(\ref{Const:CherryResolved}, \ref{Const:RCherrySF})$ (Theorem~\ref{thm:CPNTCFails}).
For the reconstructible classes~$(\ref{Const:CherryUnresolved}, \ref{Const:RCherryStack})$ and~$(\ref{Const:CherryUnresolved}, \ref{Const:RCherryUnresolved})$, the problem remains open.
Therefore we conjecture the following.
\begin{conjecture}
There exist CPNs $N$ and~$N'$ on the same leaf-set in the~$(\ref{Const:CherryUnresolved}, \ref{Const:RCherryStack})$ class where~$N'$ is contained in~$N$, but no CPS of~$N$ reduces~$N'$.
\end{conjecture}
\begin{conjecture}
Let~$N$ and~$N'$ be CPNs on the same leaf-set in the~$(\ref{Const:CherryUnresolved}, \ref{Const:RCherryUnresolved})$ class, where~$N'$ is contained in~$N$.
Then there exists a CPS of~$N$ that reduces~$N'$.
\end{conjecture}
Upon restricting~$N$ and~$N'$ to be tree-child, we have shown that the converse holds for all reconstructible classes: if~$N'$ and $N$ are both in the same reconstructible class and~$N'$ is contained in~$N$, then there exists a TCS of~$N$ that reduces~$N'$.
We believe that this still holds when~$N$ and~$N'$ are general non-binary tree-child networks, by requiring a stronger condition.
Note that it is easy to find an example of two non-binary tree-child networks that are not contained in one another, such that a TCS for one network reduces the other (see Figure~\ref{fig:CPNNonUniqueConstruction} d).
\begin{conjecture}
Let $N$ and $N'$ be non-binary tree-child networks on the same leaf set.
Then $N'$ is contained in $N$ if and only if \emph{all} TCSs of $N$ reduce $N'$.
\end{conjecture}

% Another issue with the class of CPNs is that we do not have a characterization from their topological properties.
% The double-reticulated cherries mentioned above can be contained within some CPNs.
% Moreover, there exists a CPS that reduces a given TCN.
% Therefore, the class of CPNs strictly contain the class of TCNs.
% Crowns, which are undirected cycles within a network whose bipartition into tree nodes and reticulation nodes makes it a bipartite subgraph, cannot be contained in CPNs: therefore the class of CPNs do not contain tree-based networks \citep{francis2015phylogenetic}.

%\todoRemie{Shorten this part about {\sc Hybridization}? It seems unnecessary now we have a paper about that. YM: Deleted next paragraph about TC Hybridization. I think this paragraph we can keep, since the other paper talks mainly about TC hybridization and not the general Orch hybridization (except in the discussion).}
In Section~\ref{sec:Introduction} we briefly mentioned how cherry-picking sequences originated as a tool to tackle the problem of finding a `simple' network that contains a given set of trees.
This problem is called {\sc Hybridization}, and our results on cherry-picking sequences---in particular Lemma~\ref{lem:reductionImpliesContainment}---provide insight into how we may solve generalizations of this problem, where we have inputs consisting of networks.
% We have also looked at the fundamental differences between tree-child sequences and cherry-picking sequences. We could take Lemma~\ref{lem:reductionImpliesContainment} as indication that we might not need tree-child sequences when studying {\sc Hybridization}.
Indeed, if we find a CPS that reduces the input set of networks, then the network corresponding to this CPS contains the input set. However, it is not quite clear when this gives an optimal CPN with minimal reticulation number. For example, take the network~$N$ and the tree~$T$ in Figure~\ref{fig:CPNTreeContainmentCounterExample}. If a CPS reduces the tree and the network, the corresponding CPN must have at least one more reticulation than the optimal network, which is clearly the input network~$N$.
This means that given an input of networks, it may not be possible to find a CPS that corresponds to the CPN with minimal reticulation number that contains all inputs.
It would be interesting to see if this problem also occurs when the input consist solely of trees.
That is, given an input of trees, can we always find the CPS which corresponds to a CPN with minimal reticulation number contains all these trees?
This problem aside, the problem of finding a minimal CPS gives an upper bound for {\sc Hybridization}, and a lower bound for {\sc TC cherry-picking}, where~{\sc TC cherry-picking} asks to find a minimal TCS which reduces all trees in the input set. Hence, it might also be of merit to see what exactly happens when only trees are considered as the input: does the minimal CPS give useful information about any of these problems?

\paragraph{Acknowledgements} We would like to thank Leo van Iersel for providing feedback on multiple versions of our paper.
We would like to thank Robbert Huijsman for implementing and testing our pseudocode in Python.

\bibliographystyle{chicago}

\end{document}